\RequirePackage{fix-cm}
\documentclass[smallextended]{svjour3}       
\renewcommand\qed{\hfill \ensuremath{\Box}}

\usepackage{graphicx}

\usepackage{amsmath} 
\usepackage{amsfonts}
\usepackage{amssymb} 
\usepackage{bm}           
\usepackage{mathrsfs} 
\usepackage{mathabx}
\usepackage{amsxtra}
\numberwithin{equation}{section}
\usepackage[numbers,sort&compress]{natbib} 
\usepackage[perpage,symbol*]{footmisc}  

\usepackage{float,epsf}
\usepackage[english]{babel}
\usepackage{enumerate}
\usepackage{tikz}
\usepackage{srcltx}
\usepackage{verbatim}
\usepackage{hyperref}

\usepackage[titletoc,title]{appendix}

\newcommand*\xbar[1]{%
	\hbox{%
		\vbox{%
			\hrule height 0.5pt 
			\kern0.5ex
			\hbox{%
				\kern-0.1em
				\ensuremath{#1}%
				\kern-0.1em
			}%
		}%
	}%
}

\newcommand{\VERTiii}[1]{{\left\vert\kern-0.3ex\left\vert\kern-0.3ex\left\vert #1
		\right\vert\kern-0.3ex\right\vert\kern-0.3ex\right\vert}}
\newcommand{\VERT}{\vert\kern-0.3ex\vert\kern-0.3ex\vert}
\newcommand{\VERTl}{\left\vert\kern-0.3ex\left\vert\kern-0.3ex\left\vert}
\newcommand{\VERTr}{\right\vert\kern-0.3ex\right\vert\kern-0.3ex\right\vert}
\newcommand{\VERTbig}{\big\vert\kern-0.3ex\big\vert\kern-0.3ex\big\vert}
\newcommand{\VERTBig}{\Big\vert\kern-0.3ex\Big\vert\kern-0.3ex\Big\vert}

\makeatletter
\DeclareFontFamily{OMX}{MnSymbolE}{}
\DeclareSymbolFont{MnLargeSymbols}{OMX}{MnSymbolE}{m}{n}
\SetSymbolFont{MnLargeSymbols}{bold}{OMX}{MnSymbolE}{b}{n}
\DeclareFontShape{OMX}{MnSymbolE}{m}{n}{
	<-6>  MnSymbolE5 <6-7>  MnSymbolE6 <7-8>  MnSymbolE7 <8-9>  MnSymbolE8 <9-10> MnSymbolE9 <10-12> MnSymbolE10 <12->   MnSymbolE12
}{}
\DeclareFontShape{OMX}{MnSymbolE}{b}{n}{
	<-6>  MnSymbolE-Bold5 <6-7>  MnSymbolE-Bold6 <7-8>  MnSymbolE-Bold7 <8-9>  MnSymbolE-Bold8 <9-10> MnSymbolE-Bold9 <10-12> MnSymbolE-Bold10 <12->   MnSymbolE-Bold12
}{}

\let\llangle\@undefined
\let\rrangle\@undefined
\DeclareMathDelimiter{\llangle}{\mathopen}%
{MnLargeSymbols}{'164}{MnLargeSymbols}{'164}
\DeclareMathDelimiter{\rrangle}{\mathclose}%
{MnLargeSymbols}{'171}{MnLargeSymbols}{'171}
\makeatother

\newcommand{\R}{\mathbb{R}}

\newcommand{\Fc}{\mathcal{F}}

 \newcommand{\GG}{\mathcal{G}}

\newcommand{\ve}{\varepsilon}

 \renewcommand{\t}{\theta}
 
 \newcommand{\sgn}{\text{\rm sgn}}

 \newcommand{\supp}{\text{\rm supp}\,}

 \renewcommand{\supp}{\text{\rm supp}\,}

\newcommand{\Lip}{\text{\rm Lip}}
\newcommand{\ban}[1]{\left\langle  #1 \right\rangle}  
 \newcommand{\F}{\mathcal{ F}}

\newcommand{\norm}[1]{\left\|#1\right\|}

\newcommand{\FF}{{\boldsymbol F}}
\renewcommand{\GG}{{\boldsymbol G}}

\newcommand{\ep}{{\varepsilon}}
\newcommand{\vv}{{\boldsymbol v}}

\newcommand{\tildef}[1]{\widetilde{ #1}} 
\newcommand{\field}[1]{\mathbb{#1}}

\newcommand{\N}{\field{N}}

\newcommand{\Sph}{\mathbb{S}} 
\newcommand{\DM}{\mathcal D\mathcal M} 
\newcommand{\redb }{\partial^{*}} 
\newcommand{\eps}{\varepsilon}
\renewcommand{\div}{\text{\sl div}} 
\newcommand{\dist}{\mathrm{dist}} 
\newcommand{\weakto}{\rightharpoonup} 
\newcommand{\weakstarto}{\stackrel{*}{\rightharpoonup}} 

\newcommand{\dd}{{\rm d\hspace{0.1mm}}}
\newcommand{\dr}{{\rm d}}

\newcommand{\res}{\mathop{\hbox{\vrule height 7pt width .5pt depth 0pt
\vrule height .5pt width 6pt depth 0pt}}\nolimits}

\newcommand{\Haus}[1]{{\mathscr H}^{#1}} 
\newcommand{\Leb}[1]{{\mathscr L}^{#1}} 

\def\intave#1{\int_{#1}\hbox{\llap{$\raise2.3pt\hbox{\vrule
height.9pt width7pt}\phantom{\scriptstyle{#1}}\mkern-2mu$}}}

\newcommand{\ii}{\rm i}
\newcommand{\ee}{\rm e}
\newcommand{\divmeasp}{\mathcal{DM}^{p}}

\numberwithin{equation}{section}
\numberwithin{theorem}{section}
\numberwithin{lemma}{section}
\numberwithin{proposition}{section}
\numberwithin{corollary}{section}
\numberwithin{definition}{section}
\numberwithin{remark}{section}
\numberwithin{example}{section}
\numberwithin{figure}{section}

\begin{document}

\title{Cauchy Fluxes and Gauss-Green Formulas for Divergence-Measure Fields over General Open Sets}
	
	
	\titlerunning{Cauchy Fluxes and Gauss-Green Formulas for
$\protect \divmeasp$--Fields over Open Sets}
	
	\author{Gui-Qiang G. Chen  \and \\
		Giovanni E. Comi \and\\
		Monica Torres}
	
	\authorrunning{G.-Q.~Chen \and G.~E.~Comi \and M.~Torres} 
	
	\institute{Gui-Qiang G. Chen  \at
		Mathematical Institute, University of Oxford, Oxford, OX2 6GG, UK\\
		\email{chengq@maths.ox.ac.uk}           
		\and
	     Giovanni E. Comi \at
         Scuola Normale Superiore,
         56126 Pisa, Italy\\
           \email{giovanni.comi@sns.it}
         \and
      Monica Torres \at
      Department of Mathematics, Purdue University,
 West Lafayette, IN 47907-2067, USA\\
 \email{torres@math.purdue.edu}
}

	\date{{\bf \large In memoriam William P. Ziemer}\\ {} \\
      Received: September 4, 2018 / Accepted: December 8, 2018}
	 \maketitle
	
\begin{abstract}
\, We establish the interior and exterior Gauss-Green formulas for divergence-measure fields in $L^p$ over general
open sets,
motivated by the rigorous mathematical formulation of the physical principle of balance law via the Cauchy flux
in the axiomatic foundation, for continuum mechanics allowing discontinuities and singularities.
The method, based on a distance function,
allows to give a representation of the interior (resp. exterior) normal trace of the field on the boundary of any given open set
as the limit of classical normal traces over the boundaries of interior (resp. exterior) smooth approximations of the open set.
In the particular case of open sets with continuous boundary,
the approximating smooth sets
can explicitly be characterized by using a regularized distance.
We also show that any open set with Lipschitz boundary
has a regular Lipschitz deformable boundary from the interior.
In addition, some new product rules for divergence-measure fields and suitable scalar functions are presented,
and the connection between these product rules and the representation of the normal trace of the field as a Radon measure
is explored.
With these formulas at hand, we introduce the notion of Cauchy fluxes as functionals defined on the boundaries of general bounded
open sets for the rigorous mathematical formulation of the physical principle of balance law,
and show that the Cauchy fluxes can be represented by corresponding
divergence-measure fields.

	\end{abstract}
	
	\tableofcontents
\section{\, Introduction}

We are concerned with the interior and exterior Gauss-Green formulas for unbounded divergence-measure fields over general open sets,
motivated by the rigorous mathematical formulation of the physical principle of balance law via the Cauchy flux
in the axiomatic foundation, for continuum mechanics allowing discontinuities and singularities.
The divergence-measure fields are vector fields $\FF \in L^p$ for $1 \le p \le \infty$, whose distributional divergences are Radon measures.
These vector fields form a Banach space that is denoted by $\DM^{p}$.
Even though the definitions of normal traces for unbounded divergence-measure fields have been given in
Chen-Frid \cite{CF2} and \v{S}ilhav\'y \cite{Silhavy2} (see also \cite{Frid2}),
the objective of this paper is to give a representation of the {\it interior} (resp. {\it exterior}) normal trace
on the boundary of any given open set and to prove that these normal traces can be computed
as the limit of classical normal traces over the boundaries of {\it interior} (resp. {\it exterior}) smooth approximations of the open set.
In particular, this implies analogous results on general domains (that is,  open connected sets).

The approximation of domains is a fundamental problem that has many applications in
several fields of analysis.
The answer to this question depends on both the regularity of the domain
and the type of approximation that is needed.
Our interest in this problem is motivated
from the field of hyperbolic conservation laws.
It is important to approximate
the surface of a {\it discontinuity wave} (such as a shock wave, vortex sheet, and entropy wave)
by smooth surfaces from {\it one side} of the surface
so that the {\it interior} and {\it exterior} traces of the solutions
can be defined on such a {\it discontinuity wave} as the limit of classical traces
on the smooth approximating surfaces.
Furthermore, the physically meaningful notion of Cauchy fluxes
as functionals defined on the boundaries of general bounded
open sets requires the understanding of the flow behavior
in both the interior and exterior neighborhoods of each boundary.

In this paper, we consider arbitrary open sets,
which especially include domains with finite perimeter.
The sets of finite perimeter are relevant in the field of hyperbolic conservation laws,
since the reduced boundaries of sets of finite perimeter are rectifiable sets, while
the shock surfaces are often rectifiable, at least for multidimensional scalar conservation
laws ({\it cf.} De Lellis-Otto-Westickenberg \cite{Camillo}).
Moreover, one advantage for the sets of finite perimeter is that the normal to these sets
can be well defined almost everywhere on the boundaries.

A first natural approach to produce a smooth approximation of a domain is via the convolution with some mollifiers $\eta_{\eps}$.
Indeed, it is a classical result in geometric measure theory (see the classical monographs of Ambrosio-Fusco-Pallara \cite[Theorem 3.42]{afp} and Maggi \cite[Theorem 13.8]{Maggi}) that
any set of finite perimeter $E$ can be approximated with a suitable family of smooth sets $E_k$ such that
\begin{equation}
\label{main one sided}
\mathcal{L}^{n}(E_k \Delta E) \to 0, \quad  \Haus{n - 1}(\partial^* E_k) \to \Haus{n-1}(\partial^* E) \qquad\,\,\mbox{as $k\to \infty$},
\end{equation}
where $\mathcal{L}^{n}$ is the Lebesgue measure in $\R^n$,  $\partial^{*} E$ is the reduced boundary of $E$,
and $\Delta$ denotes the symmetric difference of sets (that is, $A\Delta B:=(A\setminus B)\cup (B\setminus A))$.

The approximating smooth sets $E_k$ are the superlevel sets $A_{k;t}:= \{u_k > t\}$, for {\it a.e.} $t \in (0,1)$,
of the convolutions $u_k:=\chi_E * \eta_{\eps_{k}}$, for some suitable subsequence $\eps_{k} \to 0$ as $k\to \infty$.
The main difficulty with the convolution approach is that the approximating surfaces $u_k^{-1}(t)$ do not
provide an {\it interior} approximation in general, since portions of $u_k^{-1}(t)$
might intersect the exterior of the set.
This problem was solved by Chen-Torres-Ziemer \cite{ctz} and Comi-Torres \cite{ComiTorres}
by improving the classical result and proving an almost {\it one-sided} approximation that distinguishes the superlevel
sets for {\it a.e.} $t \in (\frac{1}{2}, 1)$ from the ones corresponding to {\it a.e.} $t \in (0,\frac{1}{2})$,
thus providing an {\it interior} and an {\it exterior} approximation of the set with
\begin{align*}
& \Haus{n-1}(u_k^{-1}(t)\cap E^0) \to 0 \qquad\, \mbox{for {\it a.e.} $t \in (\frac{1}{2}, 1)$},\\
& \Haus{n-1}(u_k^{-1}(t)\cap E^1) \to 0 \qquad\, \mbox{for {\it a.e.} $t \in (0,\frac{1}{2})$},
\end{align*}
where $E^0$  and $E^1$ are the measure-theoretic exterior and interior of the set, respectively.
Moreover, for any measure $|\mu| \ll \Haus{n-1}$, the classical result \eqref{main one sided} was improved to
\begin{align*}
&|\mu|(A_{k;t} \Delta E^{1}) \to 0, \,\, \Haus{n-1}(\partial A_{k;t}) \to \Haus{n-1}(\partial^* E)
  \qquad\qquad\,\,\,\,\,\text{for {\it a.e.} $t \in (\frac{1}{2}, 1)$}, \\
&|\mu|(A_{k;t} \Delta (E^1 \cup \partial^{*} E)) \to 0, \,\, \Haus{n-1}(\partial A_{k;t}) \to \Haus{n-1}(\partial^* E)
 \quad \text{for {\it a.e.} $t \in (0, \frac{1}{2})$}.
\end{align*}

This new one-sided approximation for sets of finite perimeter is sufficient
to obtain the Gauss-Green formula for vector fields $\FF \in \DM^{\infty}_{\rm loc}$.
Indeed, we have
$$
|\div \FF| \ll \Haus{n-1},
$$
as first observed by Chen-Frid \cite{CF1} (also see \cite{ctz,Silhavy1}),
which implies
\begin{align*}
&\div \FF  (A_{k;t}) \to \div \FF
(E^{1})    \qquad\qquad\quad\,\mbox{for a.e. $t \in (\frac{1}{2}, 1)$},\\
&\div \FF  (A_{k;t}) \to \div \FF (E^{1} \cup \partial^{*} E) \qquad
  \mbox{for a.e. $t \in (0, \frac{1}{2})$}.
\end{align*}
This allows us to obtain the interior and exterior Gauss-Green formulas
over sets of finite perimeters
(see \cite[Theorem 5.2]{ctz}).

Our focus in this paper is on the Gauss-Green formulas for $\DM^p$ fields, {\it i.e.}, unbounded
weakly differentiable vector fields in $L^p$ whose distributional divergences are Radon measures.
It has been shown that, for $\FF \in \DM^{p}$ with $1 \le p < \infty$,  the Radon measure $\div \FF$ is no longer absolutely
continuous with respect to $\Haus{n-1}$ in general. Indeed, it is absolutely continuous with respect to the Sobolev
and relative $p'$--capacities if $p \ge \frac{n}{n - 1}$,
and can be even a Dirac measure if $1 \le p < \frac{n}{n - 1}$
(see \cite[Theorem 3.2, Example 3.3]{Silhavy1}, \cite[Lemma 2.25]{ctz}, and \cite[Theorem 2.8]{phuc2008characterizations}).
Thus, a new way of approximating the integration domains entirely from the interior and the exterior separately is required,
since we cannot rely anymore on the approximation described above, as in \cite{ctz}.

\smallskip
A second approach to approximate a domain $U$ is to employ the standard distance function and define
$$
U^{\ve} := \{x \in U: \dist(x, \partial U)> \ve \};
$$
see \cite[Theorem 2.4]{Silhavy2}.
In this case, since $\dist(x, \partial U)$ is only Lipschitz continuous
for the domains with less than the $C^2$--regularity,
the coarea formula implies that $\{x \in U: \dist(x, \partial U)= \ve \}$
is just a set of finite perimeter,
for almost every $\ve >0$; see \S 5.
In \S 7, we also use a regularized distance $\rho$, which is $C^\infty$,
introduced by Lieberman \cite{L} for the Lipschitz domains
and developed further
to the $C^0$--domains by Ball-Zarnescu \cite{BallZarnescu}.
For these domains, smooth approximations are obtained, since $\rho^{-1}(\ve)$ is smooth for any $\ve > 0$.
Thus, the use of the distance functions provides an interior smooth approximation
satisfying $\div \FF  (U^{\ve}) \to \div \FF (U)$ even for unbounded divergence-measure fields.

As for the exterior approximation, we consider the sets:
$$
U_{\eps} := \{ x \in \R^{n} : \dist(x, U) < \eps \},
$$
which clearly satisfy similar properties as $U^{\eps}$.
Indeed, we will unify the exposition by defining the signed distance $d$
from $\partial U$ and its regularized version analogously in \S 5.

Another motivation of this paper is from a result of Schuricht \cite[Theorem 5.20]{Sch},
where it is proved that, for any $\FF \in \DM^{1}_{\rm loc}(\Omega)$
and any compact set $K \Subset \Omega$,
the normal trace functional can be represented as an average on the one-sided tubular neighborhoods
of $\partial K$ in the sense that
\begin{equation} \label{Normal trace average Schuricht}
\div \FF(K) = \lim_{\ve \to 0} \frac{1}{\ve} \int_{K_{\ve} \setminus K} \FF \cdot \nu_{K}^{d} \, \dr x,
\end{equation}
where $K_{\eps} = \{ x \in \Omega : \dist(x, K) \le \eps \}$,
and $\nu_{K}^{d}(x) = \nabla_{x} \dist(x, K)$ is a unit vector for $\Leb{n}$--{\it a.e.} $x \in \Omega$
such that $\dist(x, K) > 0$.
This last property says that $\nu_{K}^{d}$ is a sort of generalization of the exterior normal.
It is clear that $K_{\eps}\subset K_{\eps '}$ if $\eps < \eps '$ and
that $\bigcap_{\eps > 0} K_{\eps} = K$, which implies that $\div \FF(K_{\eps}) \to \div \FF(K)$.
Therefore, this approach is similar to the one of the exterior approximation $U_{\eps}$
of a bounded open set $U$.
In \S 5, we
use this approach as a starting point
by differentiating under the integral sign before passing to the limit in $\eps$,
so that we can obtain a boundary integral on the right-hand side.

\smallskip
The classical Gauss-Green formula for Lipschitz vector fields $\FF$ over sets of finite perimeter
was proved first by De Giorgi  \cite{degiorgi1961complementi, degiorgi1961frontiere} and Federer \cite{Federer1, Federer2},
and by Burago-Maz'ya \cite{BM2,maz2013sobolev} and Vol'pert \cite{Volpert, VH} for $\FF$ in the class of functions of bounded variation ($BV$).
The Gauss-Green formula for vector fields $\FF \in L^{\infty}$ with $\div \FF \in \mathcal{M}$ was first investigated
by Anzellotti in \cite[Theorem 1.9]{Anzellotti_1983} and \cite{anzellotti1983traces} on bounded Lipschitz domains,
and his methods were then exploited by Ambrosio-Crippa-Maniglia \cite{ACM}, Kawohl-Schuricht \cite{kawohl2007dirichlet},
Leondardi-Saracco \cite{LeoSaracco}, and Scheven-Schmidt \cite{scheven2016dual, scheven2016bv, scheven2017anzellotti}.
Independently, motivated by the problems arising from the theory of hyperbolic conservation laws,
Chen-Frid \cite{CF1} first introduced the approach of defining the interior normal traces
on the boundary of a Lipschitz deformable set as the limits of the classical normal
traces over the boundaries of the interior approximations of the set, in which the Gauss-Green formulas hold.
One of the main objectives of this paper is to develop this approach further for unbounded vector fields to
understand the interior normal traces of divergence-measure fields on the boundary of general open sets,
and to show the existence of regular Lipschitz deformations introduced in \cite{CF1}.
Even though, locally, we always have the natural regular Lipschitz
deformation: $\Psi(\hat{y}, t)=(\hat{y}, \gamma(\hat{y})+t)$ for $\gamma$ as in Definition \ref{aquiaqui} and $\hat{y}=(y_1, \cdots, y_{n-1})$,
it may not be possible to extend this deformation globally to $\partial U$ in such
a way to satisfy Definition \ref{aquiaqui} in general.

Later, the Gauss-Green formulas over sets of finite perimeter for $\DM^{\infty}$--fields
were proved in Chen-Torres \cite{ChenTorres}, \v{S}ilhav\'y \cite{Silhavy1}, and Chen-Torres-Ziemer \cite{ctz}.
Subsequent generalizations of these formulas were given by Comi-Payne \cite{comi2017locally},
Comi-Magnani \cite{comimagnani}, and Crasta-De Cicco \cite{crasta2017anzellotti, crastadecicco2}.
We refer \cite{comi2017locally,dafermos2010hyperbolic} for a more detailed exposition of the history of Gauss-Green formulas.

The case of divergence-measure vector fields in $L^p$, $p \neq \infty$, has been studied in Chen-Frid \cite{CF2}
over Lipschitz deformable boundaries and in \v{S}ilhav\'y \cite{Silhavy2}
for open sets.
The main focus of this paper is to obtain the Gauss-Green formulas by using the limit of the classical traces
over appropriate approximations of the domain,
instead of representing it as the averaging over neighborhoods of the boundaries of the domain
as in Chen-Frid \cite{CF2} and \v{S}ilhav\'y \cite{Silhavy2}.
Even though a representation of the normal trace similar to the one in this paper can also be found in Frid \cite{Frid1},
it is required in \cite{Frid1} that the boundary of the domain is Lipschitz deformable.
In \S 8, we show that this last condition can actually be removed.

Degiovanni-Marzocchi-Musesti in \cite{degiovanni1999cauchy} and later Schuricht in \cite{Sch}
sought to prove the existence of normal traces under weak regularity hypotheses
in order to achieve a representation formula for Cauchy fluxes,
contact interactions, and forces in the context of the foundation of continuum physics.
The Gauss-Green formulas obtained in \cite{degiovanni1999cauchy,Sch} are valid for $\FF \in \DM^{p}(\Omega)$ for
any $p \ge 1$,
but are applicable only to sets $U\subset \Omega$ which lie in a suitable subalgebra of sets
of finite perimeter related to the particular representative of $\FF$.
One of our objectives in this paper is to use the representation of the normal traces
as the limits of classical normal traces on smooth boundaries to obtain an analogous representation
for the contact interactions and the Cauchy fluxes on the boundaries of any general open set.

This paper is organized in the following way:
In \S 2, some basic notions and facts on the $BV$ theory and $\DM^p$--fields are recalled.
In \S 3, we establish some product rules between $\DM^{p}$--fields and suitable scalar functions, including
continuous bounded scalar functions with gradient in $L^{p'}$ for any $1 \le p \le \infty$,
which has not been stated explicitly in the literature to the best of our knowledge.
In \S 4, we investigate the distributional definition of the normal trace functional
and its relation with the product rule between $\DM^{p}$--fields and characteristic functions of Borel measurable sets.
We also provide some necessary and sufficient conditions under which the normal trace
of a $\DM^{p}$--field can be represented by a Radon measure.
In \S 5, we describe the properties of the level sets of the signed distance function from a closed set
and their applications in the proof of the Gauss-Green formulas for general open sets.
As a byproduct, we obtain generalized Green's identities and other sufficient conditions under which the normal trace
of a divergence-measure field can be represented by a Radon measure on the boundary of an open set in \S 5--6.
In \S 7, we show the existence of interior and exterior smooth approximations
for $U$ and $\overline{U}$ respectively, where $U$ is a general open set,
together with their corresponding Gauss-Green formulas.
In the case of $C^0$ domains $U$,
we employ the results of Ball-Zarnescu \cite{BallZarnescu}
to find smooth interior and exterior
approximations of $U$ and $\overline{U}$
in an explicit way.
Indeed, we are able to write the interior and exterior normal traces as the limits of the classical normal traces
on the superlevel sets of a regularized distance introduced in \cite{BallZarnescu,L}.
In \S 8, we employ Ball-Zarnescu's theorem \cite[Theorem 5.1]{BallZarnescu} to show that
any Lipschitz domain $U$ is actually Lipschitz deformable in the sense of Chen-Frid ({\it cf.} Definition \ref{aquiaqui}).
In addition, we recall the previous approximation theory for open sets with Lipschitz boundary developed
by Ne\v{c}as \cite{nevcas1962, nevcas1964equations} and Verchota \cite{Verchota_1984,Verchota_2005}
to give a more explicit representation of a particular bi-Lipschitz deformation $\Psi(x, t)$,
which is also regular in the sense that
$$
\lim_{t \to 0^{+}} J^{\partial U} \Psi_{t} = 1 \qquad \text{in $L^{1}(\partial U; \Haus{n - 1})$},
$$
where $\Psi_{t}(x) = \Psi(x, t)$, and $J^{\partial U}$ denotes the tangential Jacobian.
Finally, in \S 9,
based on the theory of normal traces for $\DM^p$--fields obtained as the limit of classical normal traces on smooth
approximations or deformations,
we introduce the notion of Cauchy fluxes as functionals defined on the boundaries of general bounded
open sets for the rigorous mathematical formulation of the physical principle of balance law involving discontinuities and singularities,
and show that the Cauchy fluxes can be represented by corresponding
divergence-measure fields.

\section{\, Basic Notations and Divergence-Measure Fields}

In this section, for self-containedness, we first present some basic notations and known facts
in geometric measure theory and elementary properties of divergence-measure fields.

In what follows, $\Omega$ is an open set in $\R^{n}$, which is called a {\em domain} if it is also connected,
and  $\mathcal{M}(\Omega)$ is the space of all Radon measures in $\Omega$.
Unless otherwise stated, $\subset$ and $\subseteq$ are equivalent.
We denote by $E \Subset \Omega$ a set $E$ whose closure $\overline{E}$ is a compact set inside $\Omega$,
by $\mathring{E}$ the topological interior of $E$, and by $\partial E$ its topological boundary.

To establish the interior and exterior  normal traces in \S 5 later,
we need to use the following classical coarea formula ({\it cf.} \cite[\S 3.4, Theorem 1 and Proposition 3]{eg}):

\begin{theorem}
Let $u : \mathbb{R}^{n} \to \mathbb{R}$ be Lipschitz. Then
\begin{equation} \label{coarea}
\int_{A} |\nabla u| \, \dr x = \int_{\mathbb{R}} \Haus{n - 1}(A \cap u^{-1}(t)) \, \dr t
\qquad \mbox{for any $\mathcal{L}^{n}$--measurable set $A$}.
\end{equation}
In addition, if $\mathrm{ess inf}|\nabla u| > 0$, and $g : \R^{n} \to \R$ is $\Leb{n}$--summable,
then $g |_{u^{-1}(t)}$ is $\Haus{n-1}$--summable for $\Leb{1}$--{\it a.e.} $t \in \R$ and
\begin{equation*} \label{coarea superlevel}
\int_{\{u > t \}} g \, \dr x = \int_{t}^{\infty} \int_{\{u = s\}} \frac{g}{|\nabla u|} \, \dr \Haus{n - 1} \, \dr s
\qquad \mbox{for any $t \in \R$}.
\end{equation*}
In particular, for any $t \in \R$ and $h \ge 0$ such that set $\{ u = t + h \}$ is negligible with respect
to the measure $g \, \dr x$,
\begin{equation} \label{coarea super-sublevel}
\int_{\{ t < u < t + h \}} g \, \dr x = \int_{t}^{t + h} \int_{\{u = s\}} \frac{g}{|\nabla u|} \, \dr \Haus{n - 1} \, \dr s.
\end{equation}
In the case that $g : \R^{n} \to \R^{n}$ is $\Leb{n}$--summable,
the same results follow for each component $g_{i}$, $i = 1, \dots, n$.
\end{theorem}

The notions of functions of bounded variation ($BV$) and sets of finite perimeter will also be used.

\begin{definition} \label{def_BV}
A function $u\in L^{1}(\Omega)$ is a {\em function of bounded variation in $\Omega$}, written as $u \in BV(\Omega)$,
if its distributional gradient $Du$ is a finite $\R^{n}$--vector valued Radon measure on $\Omega$.
We say that $u$ is of {\em locally bounded variation} in $\Omega$, written as $u \in BV_{\rm loc}(\Omega)$,
if the restriction of $u$ to every open set $U \Subset \Omega$ is in $BV(U)$.
A measurable set $E \subset \Omega$ is said to be a {\em set of finite perimeter in $\Omega$}
if $\chi_{E} \in BV(\Omega)$ and said to be of {\em locally finite perimeter} in $\Omega$ if $\chi_{E} \in BV_{\rm loc}(\Omega)$.
\end{definition}

It is well known that the topological boundary of a set of finite perimeter $E$ can be very irregular,
since it may even have positive Lebesgue measure.
On the other hand, De Giorgi \cite{degiorgi1961frontiere} discovered a suitable subset of $\partial E$
of finite $\Haus{n - 1}$--measure on which $|D \chi_{E}|$ is concentrated.

\begin{definition} \label{reducedboundary}
Let $E$ be a set of locally finite perimeter in $\Omega$.
The {\em reduced boundary} of $E$, denoted by $\redb E$,
is defined as the set of all $x \in \mathrm{supp}(|D\chi_E|) \cap \Omega$ such that the limit{\rm :}
\begin{equation*}
	\nu_{E} (x):= \lim_{r \to 0} \frac{D\chi_{E}(B(x,r))}{|D\chi_{E}|(B(x,r))}
\end{equation*}
exists in $\R^n$ and satisfies
\begin{equation*}
|\nu_{E}(x)| = 1.
\end{equation*}
The function $\nu_E: \redb E \to \Sph^{n-1}$ is called the {\em measure-theoretic unit interior normal} to $E$.
\end{definition}

The reason for which $\nu_E$ is seen as a generalized interior normal lies in the approximate tangential properties
of the reduced boundary ({\it cf.}  \cite[Theorem 3.59]{afp}).
Indeed,
$E \cap B(x, \eps)$ is asymptotically close to the half ball
$\{ y : (y - x) \cdot \nu_{E}(x) \ge 0\} \cap B(x, \eps)$ as $\eps \to 0$,
and
\begin{equation}\label{redb_concentration}
    |D\chi_{E}| = \Haus{n - 1} \res \redb E.
\end{equation}

It is a well-known result from the $BV$ theory ({\it cf.}  \cite[Corollary 3.80]{afp})
that every function $u$ of bounded variation admits a representative that is the pointwise limit $\Haus{n - 1}$--{\it a.e.}
of any mollification of $u$ and coincides $\Haus{n - 1}$--{\it a.e.} with the precise representative $u^{*}$:
\[
u^{*}(x) :=
\begin{cases}
\displaystyle \lim_{r \to 0} \frac{1}{|B(x,r)|} \int_{B(x,r)} u(y) \, \dr y & \mbox{ if this limit exists, } \\
0 & \mbox{ otherwise.}
\end{cases}
\]
In particular, if $u = \chi_{E}$ for some set of finite perimeter $E$,
then $\chi_{E}^{*} = \frac{1}{2}$ on $\redb E$ $\Haus{n - 1}$--{\it a.e.}

We state now the generalization of the coarea formula for functions of bounded variation,
which indeed shows an important connection between $BV$ functions and sets of finite perimeter;
see \cite[Theorem 3.40]{afp} for a more detailed statement and proof.

\begin{theorem} \label{Federer-Fleming coarea}
If $u \in BV(\Omega)$, then, for $\Leb{1}$--{\it a.e.} $s \in \R$,
set $\{ u > s \}$ is of finite perimeter in $\Omega$ and
\[
|Du|(\Omega) = \int_{-\infty}^{\infty} |D\chi_{\{u > s\}}|(\Omega) \dr s.
\]
\end{theorem}

We recall now the definition of divergence-measure fields, the main object of study of this paper.

\begin{definition} \label{DMdef}
A vector field $\FF \in L^{p}(\Omega; \R^{n})$ for some $1 \le p \le \infty$ is called a {\em divergence-measure field}, denoted as $\FF \in \DM^{p}(\Omega)$,
if its distributional divergence $\div \FF$ is a real finite Radon measure on $\Omega$.
A vector field $\FF$ is a {\em locally divergence-measure field}, denoted as ${\FF \in \DM^{p}_{\rm loc}(\Omega)}$,
if the restriction of $\FF$ to $U$ is in $\DM^{p}(U)$ for any $U \Subset \Omega$ open.
\end{definition}

These vector fields have been widely studied in the last two decades; for a general theory,
we refer mainly to \cite{ACM,CF1,CF2,ChenTorres,ctz,comi2017locally,Frid1,Frid2,Sch,Silhavy1,Silhavy2}
and the references cited therein.

We recall that  Lipschitz functions with compact support can be used as test functions
in the definition of distributional divergence, since $C^{\infty}_{c}(\Omega)$ functions
are dense in $\Lip_{c}(\Omega)$, the space of Lipschitz functions with compact support
in $\Omega$.

Finally, we introduce two definitions, which are required in \S 8,
in order that the results on the smooth approximation of domains of class $C^{0}$ by Ball-Zarnescu \cite{BallZarnescu}
can be employed to show that the boundary of any bounded Lipschitz domain is Lipschitz deformable
in the sense of Chen-Frid \cite{CF1,CF2}.

\begin{definition}
\label{gooddirection}
Let $\Omega \subset \mathbb{R}^n$ be a domain of class $C^0$.
For a point $P \in \mathbb{R}^n$, define a good direction at $P$, with respect to a ball $B(P,\delta)$ with $\delta > 0$
and $B(P,\delta) \cap \partial \Omega \neq \emptyset$, to be a vector $\nu \in \Sph^{n-1}$
such that there is an orthonormal coordinate system $Y=(y',y_n)=(y_1,y_2,... y_{n-1},y_n)$ with origin at point $P$
so that $\nu = e_n$ is the unit vector in the $y_n$--direction which,
together with a continuous function $f : \mathbb{R}^{n-1} \to \mathbb{R}$ {\rm (}depending on $P$, $\nu$, and $\delta${\rm )},
satisfies
\begin{equation*}
\Omega \cap B(P,\delta) =\{y \in \mathbb{R}^n : y_n > f(y'), |y| < \delta \}.
\end{equation*}
We say that $\nu$ is a good direction at $P$ if it is a good direction with respect to some ball $B(P,\delta)$ with $B(P,\delta) \cap \partial \Omega \neq 0$.
If $P \in \partial \Omega$, then a good direction $\nu$ at $P$ is called a pseudonormal at $P$.
\end{definition}

\begin{definition}
\label{aquiaqui}
Let $\Omega$ be an open subset in $\mathbb{R}^n$.
We say that $\partial \Omega$ is a deformable Lipschitz boundary, provided that the following hold{\rm :}
\begin{itemize}
\item[(i).] For each $x \in \partial \Omega$,
there exist $r>0$ and a Lipschitz mapping $\gamma: \mathbb{R}^{n-1} \to \mathbb{R}$ such that,
upon rotating and relabeling the coordinate axis if necessary,
\[
\Omega \cap Q(x,r)=\{y \in \mathbb{R}^n : y_n> \gamma(y_1,...,y_{n-1}) \} \cap Q(x,r),
\]
where $Q(x,r)=\{y \in \mathbb{R}^n :|y_i -x_i|\leq r, i=1,...,n\}$.

\item[(ii).] There exists a map $\Psi: \partial \Omega \times [0,1] \to \overline{\Omega}$
such that $\Psi$ is a bi-Lipschitz homeomorphism over its image and
$\Psi(\cdot,0) \equiv {\rm Id}$,
where ${\rm Id}$ is the identity map over $\partial\Omega$.
Denote $\partial \Omega_{\tau}= \Psi ( \partial \Omega \times\{\tau\})$ for $\tau \in (0,1]$,
and denote $\Omega_{\tau}$ the open subset of $\Omega$ whose boundary is $\partial \Omega_{\tau}$.
We call $\Psi$ a Lipschitz deformation of $\partial \Omega$.
\end{itemize}
The Lipschitz deformation is {\em regular} if
\begin{equation} \label{regular_deformation}
\lim_{\tau \to 0^{+}} J^{\partial \Omega} \Psi_{\tau} = 1 \qquad \text{in $L^{1}(\partial \Omega; \Haus{n - 1})$},
\end{equation}
where $\Psi_{\tau}(x) = \Psi(x, \tau)$, and $J^{\partial \Omega}$ denotes the tangential Jacobian.
\end{definition}

\section{\, Product Rules between Divergence-Measure Fields and Suitable Scalar Functions}

In this section, we give some new product rules between $\DM^p$--fields and suitable scalar functions.
We start by proving a product rule for vector fields in $\DM^{p}$ for any $1 \le p \le \infty$,
which is the explicit formulation of a particular case of
the product rule for $\DM^p$--fields stated in \cite[Theorem 3.2]{CF2}.

From now on, as customary, we always use a standard mollifier:
\begin{equation}\label{mollifier-1}
\eta \in C^{\infty}_{c}(B(0, 1)) \,\, \mbox{radially symmetric, with $\,\eta \ge 0$ and $\int_{B(0, 1)} \eta(x) \, \dr x = 1$},
\end{equation}
and
\begin{equation}\label{mollifier-2}
\eta_{\eps}(x) := \frac{1}{\eps^{n}} \eta(\frac{x}{\eps}).
\end{equation}

\smallskip
\begin{proposition} \label{product rule p}
If $\FF \in \DM^{p}(\Omega)$ for $1 \le p \le \infty$, and $\phi \in C^{0}(\Omega) \cap L^{\infty}(\Omega)$
with $\nabla \phi \in L^{p'}(\Omega; \R^{n})$ for $p' = \frac{p}{p - 1}$, then
$$
\phi \FF \in \DM^{p}(\Omega),
$$
and
\begin{equation} \label{product rule}
\div (\phi \FF) = \phi\,\div \FF + \FF \cdot \nabla \phi.
\end{equation}
\end{proposition}

\begin{proof}
\, It is clear that $\phi \FF \in L^{p}(\Omega; \R^{n})$.

We first consider the case $1 < p \le \infty$.
Take $\phi_{\eps} := \phi \ast \eta_{\eps}$, where $\eta_{\eps}$ is defined in \eqref{mollifier-2}.
Then $\phi_{\eps} \to \phi$ uniformly
on compact subsets of $\Omega$, and $\nabla \phi_{\eps} \to \nabla \phi$ in $L^{p'}_{\rm loc}(\Omega; \R^{n})$.

For any test function $\psi \in C^{1}_{c}(\Omega)$,  we have
\begin{align} \label{product rule eq}
\int_{\Omega} \phi_{\eps} \FF \cdot \nabla \psi \, \dr x
& = \int_{\Omega} \FF \cdot \nabla (\phi_{\eps} \psi) \, \dr x - \int_{\Omega} \psi \FF \cdot \nabla \phi_{\eps} \, \dr x \\
& = - \int_{\Omega} \psi \phi_{\eps} \, \dd\div \FF - \int_{\Omega} \psi \FF \cdot \nabla \phi_{\eps} \, \dr x. \nonumber
\end{align}
We can now pass to the limit as $\eps \to 0$ to obtain \eqref{product rule} in the sense of distributions.
On the other hand, it follows that
\begin{equation*}
\left | \int_{\Omega} \phi \FF \cdot \nabla \psi \, \dr x \right |
\le \big(\|\phi\|_{L^\infty(\Omega)} |\div \FF|(\Omega)
    + \|\FF\|_{L^{p}(\Omega; \R^{n})} \|\nabla \phi\|_{L^{p'}(\Omega; \R^{n})}\big)  \|\psi\|_{L^\infty(\Omega)}.
\end{equation*}
This shows that $\div (\phi \FF)$ is a finite Radon measure on $\Omega$,
by the density of $C_{c}^{1}(\Omega)$ in $C_{c}(\Omega)$ with respect to the sup norm,
and that \eqref{product rule} holds in the sense of Radon measures.

\medskip
For the case $p = 1$,  we mollify $\FF$ instead, since $\phi \in W^{1, \infty}(\Omega) \subset \Lip_{\rm loc}(\Omega)$.
For any $\psi \in C^{1}_{c}(\Omega)$, we obtain
\begin{equation*}
\int_{\Omega} \phi \, \FF_{\eps} \cdot \nabla \psi \, \dr x
= \int_{\Omega} \FF_{\eps} \cdot \nabla (\phi \psi) \, \dr x
 - \int_{\Omega} \psi \, \FF_{\eps} \cdot \nabla \phi \, \dr x.
\end{equation*}
By passing to the limit as $\eps \to 0$, the $L^{1}$--convergence of $\FF_{\eps}$ to $\FF$ implies
\begin{align*}
\int_{\Omega} \phi \, \FF \cdot \nabla \psi \, \dr x
& = \int_{\Omega} \FF \cdot \nabla (\phi \psi) \, \dr x - \int_{\Omega} \psi \, \FF \cdot \nabla \phi \, \dr x \\
& = - \int_{\Omega} \psi \phi \, \dd \div \FF - \int_{\Omega} \psi \, \FF \cdot \nabla \phi \, \dr x.
\end{align*}
This shows \eqref{product rule} in the sense of distributions for $p = 1$. Then we can conclude by arguing as before.
\end{proof}

\begin{remark} \, Notice that, if $\phi \in L^{\infty}(\Omega)$ and $\nabla \phi \in L^{p'}(\Omega; \R^{n})$,
then $\phi \in W^{1, p'}_{\rm loc}(\Omega)$. Thus, if $p' > n$ {\rm (}that is, $\, 1 \le p < \frac{n}{n - 1}${\rm )},
we do not have to require that  $\phi \in C^{0}(\Omega)$ in Proposition {\rm \ref{product rule p}},
since this follows by Morrey's inequality {\rm (}see \cite[Theorem 3, \S 4.5.3]{eg}{\rm )}.
\end{remark}

Proposition \ref{product rule p} can be extended to the case $p = \infty$, by taking $g \in BV(\Omega) \cap L^{\infty}(\Omega)$.
Indeed, a product rule between essentially bounded divergence-measure fields and scalar functions
of bounded variation was first proved by Chen-Frid \cite[Theorem 3.1]{CF1}
(also see  \cite{Frid1}).

\begin{theorem}[Chen-Frid \cite{CF1}] \label{productruleinfty}
\, Let $g \in BV(\Omega) \cap L^{\infty}(\Omega)$ and $\FF \in \DM^{\infty}(\Omega)$. Then $g \FF \in \DM^{\infty}(\Omega)$ and
\begin{equation}\label{PRDMF}
\mathrm{div}(g \FF) = g^{*} \mathrm{div}\FF + \overline{\FF \cdot Dg}
\end{equation}
in the sense of Radon measures on $\Omega$,
where $g^{*}$ is the precise representative of $g$,
and $\overline{\FF \cdot Dg}$ is a Radon measure, which is the weak-star limit of $\FF \cdot \nabla g_{\eps}$
for the mollification $g_{\eps}:=g \ast \eta_{\eps}$,
and is absolutely continuous with respect to $|Dg|$. In addition,
$$
|\overline{\FF \cdot D g}| \le \|\FF\|_{L^{\infty}(\Omega; \R^{n})} |D g|.
$$
\end{theorem}

One could ask whether it would be possible to obtain a similar result also for $\FF \in \DM^{p}(\Omega)$, $1 \le p < \infty$,
by imposing some other assumptions on $\FF$ weaker than the essential boundedness.
It is obvious that $g \FF \in L^{p}(\Omega)$ and that $\div(g \FF)$ is a distribution of order $1$,
by definition; hence one should look for conditions under which it can be extended to
a linear continuous functional on $C_{c}(\Omega)$.

Our investigation is motivated by the following example,
where $g$ is the characteristic function of a set of finite perimeter $E$,
and $\FF$ is a vector field in $\DM^p_{\rm loc}, 1\le p< 2$,
which is unbounded on $\redb E$.

\begin{example} \label{prodruleWhitney}
\, Let $n = 2$, $g = \chi_{(0, 1)^{2}}$, and
$$
\FF(x_1, x_2) = \frac{1}{2 \pi} \frac{(x_1, x_2)}{x_1^{2} + x_2^{2}},
$$
which implies that $\div \FF = \delta_{(0, 0)}$.
Then $g \FF \in \DM^{p}_{\rm loc}(\R^{2})$ for any $1 \le p < 2$ with
\begin{equation} \label{divergence_prod_1}
\div (g \FF) = \frac{1}{4} \delta_{(0, 0)} + (\FF, Dg), \end{equation}
where
\begin{equation} \label{pairing_Whitney_boundary}
(\FF, Dg)(\phi) := - \frac{1}{2 \pi} \left ( \int_{0}^{1} \frac{\phi(x_{1}, 1)}{1 + x_{1}^{2}} \, \dr x_{1}
+ \int_{0}^{1} \frac{\phi(1, x_{2})}{1 + x_{2}^{2}} \, \dr x_{2} \right )
\,\,\,\mbox{for any $\phi \in C_{c}(\R^{2})$}.
\end{equation}
This pairing functional can also be regarded as a principal value in the sense that
\begin{equation*}
(\FF, Dg)(\phi) = \lim_{\eps \to 0} \int_{\partial (0, 1)^{2} \setminus B(0, \eps)} \phi \, \FF \cdot \nu_{(0, 1)^{2}} \, \dr \Haus{1},
\end{equation*}
so that $|(\FF, Dg)| \ll |D g| = \Haus{1} \res \redb (0, 1)^{2}$.
Moreover,
term $\frac{1}{4} \delta_{(0, 0)}$ comes from
the fact that $g^{*}(0)=\frac{1}{4}$ {\rm (}{\it i.e.}  the value of the precise representative of $g$ at the origin{\rm )}.

In order to prove these claims, we take $\phi \in C^{1}_{c}(\R^{2})$ to see
\begin{align*}
&\int_{\R^{2}} g \FF \cdot \nabla \phi \, \dr x
= \lim_{\eps \to 0} \int_{(0, 1)^{2} \setminus B(0, \eps)} \FF \cdot \nabla \phi \, \dr x \\
& = \lim_{\eps \to 0} \bigg( - \int_{\partial (0, 1)^{2} \setminus B(0, \eps)} \phi \FF \cdot \nu_{(0, 1)^{2}} \, \dr \Haus{1}\\
&\qquad\qquad
- \int_{\partial B(0, \eps) \cap \{ x_1 > 0, x_2> 0 \}}
  \phi(x_1, x_2) \frac{1}{2 \pi} \frac{(x_1, x_2)}{x_1^{2} + x_2^{2}}\cdot\frac{(x_1, x_2)}{\sqrt{x_1^{2} + x_2^{2}}} \, \dr \Haus{1} \bigg) \\
& = - \frac{1}{2 \pi} \lim_{\eps \to 0} \bigg( \int_{\eps}^{1} \phi(x_1, 0) \frac{(x_1, 0)}{x_1^{2}} \cdot (0, 1) \, \dr x_1
  + \int_{\eps}^{1} \phi(0, x_2) \frac{(0, x_2)}{x_2^{2}} \cdot (1, 0) \, \dr x_2 \\
  &\qquad\qquad\quad\,\,\,
  + \int_{0}^{1} \phi(1, x_2) \frac{(1, x_2)}{1 + x_2^{2}} \cdot (-1, 0) \, \dr  x_2
  + \int_{0}^{1} \phi(x_1, 1) \frac{(x_1, 1)}{x_1^{2} + 1} \cdot (0, - 1) \, \dr x_1 \\
&\qquad\qquad\quad\,\,\,
+ \int_{0}^{\frac{\pi}{2}} \phi(\eps \cos{\theta}, \eps \sin{\theta}) \, \dr \theta \bigg) \\
& = \frac{1}{2 \pi} \bigg( \int_{0}^{1} \frac{\phi(x_{1}, 1)}{1 + x_{1}^{2}} \, \dr x_{1}
+ \int_{0}^{1} \frac{\phi(1, x_{2})}{1 + x_{2}^{2}} \, dx_{2} \bigg) - \frac{1}{4} \phi(0, 0).
\end{align*}
This shows that $\div (g \FF)$ is a distribution of order $0$, so that it is a measure, since it can be uniquely extended
to a functional on $C_{c}(\R^{2})$ by density.

In addition, for any $\phi \in C^{0}([0, 1]^{2})$ with $\nabla \phi \in L^{p'}((0, 1)^{2})$ for some $p \in [1, 2)$,
the following integration by parts formula holds{\rm :}
\begin{align}
&\int_{(0, 1)^{2}} \frac{(x_1, x_2)}{x_1^{2} + x_2^{2}} \cdot \nabla \phi(x_1, x_2) \, \dr x_1 \, \dr x_2
+ \frac{\pi}{2} \phi(0, 0)\nonumber \\
&= \int_{0}^{1} \frac{\phi(x_{1}, 1)}{1 + x_{1}^{2}} \, \dr x_{1} + \int_{0}^{1} \frac{\phi(1, x_{2})}{1 + x_{2}^{2}} \, \dr x_{2}.
 \label{IBP_Whitney}
\end{align}
Indeed, since $\chi_{(0, 1)^{2}} \FF \in \DM^{p}(\R^{2})$ for any $p \in [1, 2)$, then \eqref{product rule} yields
\begin{equation*}
\div(\phi \chi_{(0, 1)^{2}} \FF) = \phi\, \div ( \chi_{(0, 1)^{2}} \FF) + \chi_{(0, 1)^{2}} \FF \cdot \nabla \phi,
\end{equation*}
which, by \eqref{divergence_prod_1}, implies
\begin{equation} \label{divergence_prod_phi}
\div(\phi \chi_{(0, 1)^{2}} \FF)
=  \frac{1}{4} \phi(0, 0) \delta_{(0, 0)} + \phi (\FF, D \chi_{(0, 1)^{2}}) + \chi_{(0, 1)^{2}} \FF \cdot \nabla \phi.
\end{equation}
Finally, \eqref{IBP_Whitney} follows by evaluating \eqref{divergence_prod_phi} over $\R^{2}$,
using the fact that $\phi \chi_{(0, 1)^{2}} \FF$ has compact support
to obtain $\div(\phi \chi_{(0, 1)^{2}} \FF)(\R^{2}) = 0$ by \cite[Lemma {\rm 3.1}]{comi2017locally},
and employing \eqref{pairing_Whitney_boundary}.
\end{example}

In this example,
the cancellations between $\FF$ and $\nu_{(0, 1)^{2}}$ play a crucial role in order to ensure
the existence of a measure given by the pairing of $\FF$ and $Dg$.

Indeed, we can impose the existence of such a measure in order to achieve a more general product rule.

\begin{theorem} \label{prodruleBVgeneral}
Let $\FF \in \DM^{p}(\Omega)$ for $1 \le p \le \infty$, and let $g \in L^{\infty}(\Omega) \cap BV(\Omega)$.
Assume that there exists a measure $(\FF, D g) \in \mathcal{M}(\Omega)$ such that $\FF_{\eps} \cdot Dg \weakto (\FF, D g)$
for any mollification $\FF_{\eps}$ of $\FF$.
Then
$$
g \FF \in \DM^{p}(\Omega),
$$
and
\begin{equation} \label{prod rule sigma}
\div(g \FF) = \tildef{g}\, \div \FF + (\FF, D g),
\end{equation}
where $\tildef{g} \in L^{\infty}(\Omega, |\div \FF|)$ is the weak$^{*}$--limit of a suitable subsequence of
mollified functions $g_{\eps}$ of $g$, which satisfies $\tildef{g}(x) = g^{*}(x)$ whenever $g^{*}$ is well defined.
In addition,
$$
|(\FF, D g)| \ll \Haus{n - 1} \,\qquad \mbox{if $p = \infty$},
$$
and, if $p \in \big [\frac{n}{n - 1}, \infty \big )$,
$$
|(\FF, D g)|(B) = 0
\qquad\mbox{for any Borel set $B$ with $\sigma$--finite $\Haus{n - p'}$ measure}.
$$
\end{theorem}

\begin{proof} \, It is clear that $g \FF \in L^{p}(\Omega; \R^{n})$.
We now divide the remaining proof into two steps.

\smallskip
1. In order to show \eqref{prod rule sigma}, we take any mollification $g_{\eps} = g \ast \eta_{\eps}$
with $\eta_{\eps}$ defined in \eqref{mollifier-2}.
Then we select $\phi \in \Lip_{c}(\Omega)$ to obtain
\begin{equation}\label{3.5.1}
\int_{\Omega} g_{\eps} \FF \cdot \nabla \phi \, \dr x
= - \int_{\Omega} \phi g_{\eps} \, \dd \div \FF - \int_{\Omega} \phi \FF \cdot \nabla g_{\eps} \, \dr x.
\end{equation}
Since $g_{\eps} \to g$ in $L^{p'}_{\rm loc}(\Omega)$, we have
\begin{equation}\label{3.5.2}
\int_{\Omega} g_{\eps} \FF \cdot \nabla \phi \, \dr x \to \int_{\Omega} g \FF \cdot \nabla \phi \, \dr x \qquad \mbox{as $\eps\to 0$}.
\end{equation}
Notice that $|g_{\eps}(x)| \le \|g\|_{L^{\infty}(\Omega)}$ for any $x \in \Omega$.
Then there exists a weak$^{*}$--limit $\tildef{g} \in L^{\infty}(\Omega, |\div \FF|)$ for a suitable subsequence $\{g_{\eps_{k}}\}$
so that $\tildef{g}$ coincides with the precise representative $g^{*}$ whenever this is well defined. Therefore, we obtain
\begin{equation}\label{3.5.3}
\int_{\Omega} \phi g_{\eps} \, \dd \div \FF  \to \int_{\Omega} \phi \tildef{g} \, \dd \div \FF
\end{equation}
up to
a subsequence.
As for the last term, we have
\begin{equation}\label{3.5.4}
\int_{\Omega} \phi(x) \FF(x) \cdot \nabla g_{\eps}(x) \, \dr x
= \int_{\Omega} (\phi \FF)_{\eps}(y) \cdot \, \dr  Dg(y).
\end{equation}
By the uniform continuity of $\phi$, for any $\delta> 0$ and $x \in \Omega$,
there exists $\eps_{0} > 0$ such that $|\phi(y) - \phi(x)| < \eta$
for any $y \in B(x, \eps)$ and $\eps \in (0, \eps_{0})$.
Since $\phi$ has compact support in $\Omega$, we can also assume that $B(x, \eps) \subset \Omega$ without loss of generality.
This implies
\begin{align*}
\big|(\phi \FF)_{\eps}(x) - \phi(x) \FF_{\eps}(x)\big|
& = \left |\int_{\Omega} \big(\phi(y) - \phi(x)\big) \FF(y) \eta_{\eps}(x - y) \, \dr y \right | \\
& \le \delta \int_{B(0, 1)} |\FF(x + \eps z)| \eta(z) \, \dr z \\
& \le \delta \| \eta\|_{L^{p'}(B(0, 1))} \| \FF \|_{L^{p}(\Omega; \R^{n})}.
\end{align*}
Hence, it follows that
\begin{equation}\label{3.5.5}
\int_{\Omega} (\phi \FF)_{\eps}(y) \cdot \, \dd Dg(y) = \int_{\Omega} \phi(y) \, \FF_{\eps}(y) \cdot \, \dd Dg(y) + o_{\eps}(1).
\end{equation}
Now we use our assumption on sequence $\FF_{\eps}$ to obtain
\begin{equation}\label{3.5.6}
\int_{\Omega} \phi(y) \, \FF_{\eps}(y) \cdot \, \dd Dg(y)\, \to \int_{\Omega} \phi(y) \, \dr (\FF, D g)(y).
\end{equation}
Combining \eqref{3.5.1}--\eqref{3.5.6}, we conclude that $\tildef{g}$ is actually unique and that \eqref{prod rule sigma} holds.
In particular, we see that $\div(g \FF) \in \mathcal{M}(\Omega)$, which implies that $g \FF \in \DM^{p}(\Omega)$.

\smallskip
2. As for the absolute continuity property of $(\FF, D g)$, we notice that
\begin{equation*}
(\FF, D g) = \div (g \FF) - \tildef{g}\, \div \FF
\end{equation*}
and  $\FF, g \FF \in \DM^{p}(\Omega)$.
We recall now that $|\div \FF|+|\div(g \FF)| \ll \Haus{n - 1}$ if $p = \infty$ (see \cite[Proposition 3.1]{CF1}
and \cite[Theorem 3.2]{Silhavy1}) and that, if $p \in [\frac{n}{n-1}, \infty)$, $|\div \FF|(B) = |\div (g \FF)|(B) = 0$
for any Borel set $B$ with $\sigma$--finite $\Haus{n - p'}$ measure, by \cite[Theorem 3.2]{Silhavy1}.
This concludes the proof.
\end{proof}

It seems to be delicate
to characterize the cases in which  measure $(\FF, D g)$ does exist
and the absolute continuity, {\it i.e.} $|(\FF, D g)| \ll |D g|$ holds as in Example \ref{prodruleWhitney}.
We give here a partial result.

\begin{corollary} \label{weak_star_sub_prod_rule_F}
Let $\FF \in \DM^{p}(\Omega)$ for $1 \le p \le \infty$, and let $g \in L^{\infty}(\Omega) \cap BV(\Omega)$.
Assume that there exists $\tildef{\FF} \in L^{\infty}_{\rm loc}(\Omega, |D g|; \R^{n})$ such that
$\FF_{\eps} \weakstarto \tildef{\FF}$ in $L^{\infty}_{\rm loc}(\Omega, |D g|; \R^{n})$,
where $\FF_{\eps} = \FF \ast \eta_{\eps}$ is the mollification of $\FF$.
Then  $g \FF \in \DM^{p}_{\rm loc}(\Omega)$ and
\begin{equation*}
\div(g \FF) = \tildef{g}\, \div \FF + \tildef{\FF} \cdot D g,
\end{equation*}
where $\tildef{g} \in L^{\infty}(\Omega; |\div \FF|)$ is the weak$^{*}$--limit of a subsequence of $g_{\eps}$ so that
$\tildef{g}(x) = g^{*}(x)$, whenever $g^{*}$ is well defined. In addition, for any open set $U \Subset \Omega$,
\begin{equation} \label{uniform_control_Dg}
|\tildef{\FF} \cdot D g| \res U
\le \inf_{\substack{U \Subset U' \Subset \Omega \\U' \ {\rm open}}}
 \|\FF\|_{L^{\infty}(U'; \R^{n})}\, |D g| \res U.
 \end{equation}
\end{corollary}

\begin{proof}
\, The first part of the result follows directly from Theorem \ref{prodruleBVgeneral},
since the assumptions imply that $\FF_{\eps} \cdot D g \weakto (\FF, D g) = \tildef{\FF} \cdot Dg$.
Moreover, since $\tildef{\FF} \in L^{\infty}_{\rm loc}(\Omega, |D g|; \R^{n})$, we have
$$
|\tildef{\FF} \cdot D g| \le \|\tildef{\FF}\|_{L^{\infty}(U, |D g|; \R^{n})} |D g|
\qquad\mbox{on any open set $U \Subset \Omega$.}
$$
Finally, since $|\FF_{\eps}(x)| \le \|\FF\|_{L^{\infty}(U + B(0, \eps); \R^{n})}$ for any $x \in U$, then,
for any open set $U'$ satisfying $U \Subset U' \Subset \Omega$,
the lower semicontinuity of the $L^{\infty}$--norm with respect to the weak$^{*}$--convergence implies
\begin{equation*}
\|\tildef{\FF}\|_{L^{\infty}(U, |D g|; \R^{n})}
\le \liminf_{\eps \to 0} \sup_{x \in U} |\FF_{\eps}(x)|
\le \|\FF\|_{L^{\infty}(U'; \R^{n})}.
\end{equation*}
By taking the infimum over $U'$, we obtain \eqref{uniform_control_Dg}.
\end{proof}

\begin{remark}
\, The assumptions on $\FF$ are satisfied in the case ${\FF \in C^{0}(\Omega; \R^{n})}$, for which $\tildef{\FF} = \FF$.

If $\FF \in L^{\infty}_{\rm loc}(\Omega; \R^{n})$, then, for any open set $\Omega' \Subset \Omega$,
$$
|\FF_{\eps}(x)| \le \| \FF \|_{L^{\infty}(\Omega' + B(0, \eps); \R^{n})} \qquad \mbox{for any $x \in \Omega'$}.
$$
Thus, by weak$^{*}$--compactness, there exists $\tildef{\FF} \in L^{\infty}_{\rm loc}(\Omega, |D g|; \R^{n})$
such that $\FF_{\eps} \weakstarto \tildef{\FF}$ in $L^{\infty}_{\rm loc}(\Omega, |D g|; \R^{n})$, up to a subsequence.
This implies the result of Corollary {\rm \ref{weak_star_sub_prod_rule_F}} again.
Moreover, since $|\div \FF| \ll \Haus{n - 1}$, by \cite[Theorem 3.2]{Silhavy1},
we can conclude
$$
\tildef{g} = g^{*}  \qquad \mbox{$|\div \FF|$--{\it a.e.}}
$$
since the precise representative of a $BV$ function $g$ exists $\Haus{n - 1}$--{\it a.e.}
In addition, by the product rule established in Theorem {\rm \ref{productruleinfty}},
we obtain the identity{\rm :}
$$
{\tildef{\FF} \cdot D g = \overline{\FF \cdot D g}}.
$$
Thus, if $\nu_{g}$ is the Borel vector field such that $D g = \nu_{g} |D g|$, then
\begin{equation*}
\overline{\FF \cdot D g} = (\tildef{\FF} \cdot \nu_{g}) |D g|.
\end{equation*}
That is, $\tildef{\FF} \cdot \nu_{g}$ is the density of measure $\overline{\FF \cdot D g}$
with respect to $|D g|$.

Finally,
the assumption
that $\FF \in L^{\infty}_{\rm loc}(\Omega; \R^{n})$ can be relaxed
to $\FF \in L^{\infty}_{\rm loc}(U; \R^{n})$, for some open set $U \supset \mathrm{supp}(|D g|)$.
Indeed, this implies that $\FF_{\eps}$ is uniformly bounded in $L^{\infty}(U', |D g|; \R^{n})$,
for any open set $U' \Subset U$ and $\eps$ small enough,
which ensures the existence of a weak$^{*}$--limit $\tildef{\FF} \in L^{\infty}_{\rm loc}(U, |D g|; \R^{n})$,
up to
a subsequence.
\end{remark}

On the other hand, we have seen that there are some examples of unbounded and discontinuous $\DM^p$--fields
which admit a product rule of this type, as in Example \ref{prodruleWhitney}.
Moreover, there exists an unbounded $\DM^p$--field $\GG$ and a set of finite perimeter $E$
for which a product rule holds,
but $|(\GG, D \chi_{E})|$ is not absolutely continuous with respect to $|D \chi_{E}|$,
as shown in the following example.

\begin{example}
Let $n = 2$, $E = (0, 1)^{2}$, and $\FF$ as in Example {\rm \ref{prodruleWhitney}}.
We have shown that $\chi_{E} \FF \in \DM^{p}_{\rm loc}(\R^{2})$ for any $p \in [1, 2)$ and that
\begin{equation} \label{divergence_prod_1_E}
\div(\chi_{E} \FF) = \frac{1}{4} \delta_{(0, 0)} + (\FF, D \chi_{E}),
\end{equation}
by \eqref{divergence_prod_1}.
Let now $\GG := \chi_{E} \FF$.
It is clear that $\chi_{E} \GG = \GG$ so that
$\div (\chi_{E} \GG) \in \mathcal{M}(\R^{2})$.
Let $\eta_{\eps}(x)$ be the mollifiers defined in \eqref{mollifier-1}--\eqref{mollifier-2},
and let $\phi \in C^{1}_{c}(\R^{2})$.
A simple calculation shows that
\begin{equation*}
\int_{\R^{2}} (\eta_{\eps} \ast \chi_{E}) \GG \cdot \nabla \phi \, \dr x
= - \int_{\R^{2}} \phi (\eta_{\eps} \ast \chi_{E}) \, \dd \div \,\GG
 - \int_{\R^{2}} \phi \GG \cdot \nabla (\eta_{\eps} \ast \chi_{E}) \, \dr x.
\end{equation*}
By Lebesgue's dominated convergence theorem, we have
\begin{equation*}
\int_{\R^{2}} (\eta_{\eps} \ast \chi_{E}) \GG \cdot \nabla \phi \, \dr x \to \int_{\R^{2}} \chi_{E} \GG \cdot \nabla \phi \, \dr x
= - \int_{\R^{2}} \phi \, \dd \div (\chi_{E} \GG),
\end{equation*}
and
\begin{align*}
\int_{\R^{2}} \phi (\eta_{\eps} \ast \chi_{E}) \, \dd \div\, \GG
&= \int_{\R^{2}} \phi (\eta_{\eps} \ast \chi_{E}) \, \dr \big(\frac{1}{4} \delta_{(0, 0)} + (\FF, D \chi_{E}) \big)\\
&\to \frac{1}{16} \phi(0, 0) + \int_{\redb E} \frac{1}{2} \phi \, \dr (\FF, D \chi_{E}),
\end{align*}
since $|(\FF, D \chi_{E})| \ll |D \chi_{E}|$ and $\chi_{E}^{*}(0, 0) = \frac{1}{4}$.
This and the density of $C^{1}_{c}(\R^{2})$ in $C^{0}_{c}(\R^{2})$ show
that $\GG \cdot \nabla (\eta_{\eps} \ast \chi_{E})$ is weakly converging to some measure $(\GG, D \chi_{E})$
that satisfies
\begin{equation} \label{divergence_prod_2}
\div(\chi_{E} \GG) = \frac{1}{16} \delta_{(0, 0)} + \frac{1}{2} (\FF, D \chi_{E}) + (\GG, D \chi_{E}).
\end{equation}
However, it is clear that $\div(\chi_{E} \GG) = \div\,\GG = \div(\chi_{E} \FF)$.
Therefore, \eqref{divergence_prod_1_E}--\eqref{divergence_prod_2} imply
\begin{equation} \label{pairing_E_G}
(\GG, D \chi_{E}) = \frac{3}{16} \delta_{(0, 0)} + \frac{1}{2} (\FF, D \chi_{E}).
\end{equation}
Therefore, $|(\GG, D \chi_{E})| \ll |D \chi_{E}| = \Haus{1} \res \redb E$ does not hold,
since there is a concentration at $(0, 0)$.
\end{example}

\section{\, Regularity of Normal Traces of Divergence-Measure Fields}

In this section, we investigate the connection between these product rules and
the representation of the normal trace
of the $\DM^p$--field as a Radon measure.

We first introduce the notion of generalized normal trace of a $\DM^p$--field $\FF$
on the boundary of a Borel set $E$,
which has indeed a close
relation with
the product rule between $\FF$ and $\chi_{E}$.

\begin{definition}
Given $\FF \in \DM^{p}(\Omega)$ for $1 \le p \le \infty$, and a bounded Borel set $E \subset \Omega$,
define the normal trace of $\FF$ on $\partial E$ as
\begin{equation} \label{normal trace def}
\ban{\FF \cdot \nu, \phi}_{\partial E} := \int_{E} \phi \, \dr \div \FF + \int_{E} \FF \cdot \nabla \phi \, \dr x
\qquad \mbox{for any $\phi \in \Lip_{c}(\R^{n})$}.
\end{equation}
\end{definition}

\begin{remark}
\, Since $\div \FF$ is a Radon measure, any Borel set $E$ is $|\div \FF|$--measurable.
Moreover, for any $|\div \FF|$--measurable set $E$, there is a Borel set $B \supset E$ such that $|\div \FF|(B \setminus E) = 0$,
so that there exists a $|\div \FF|$--negligible set $\mathcal{N}_{E}$ with $\mathcal{N}_{E} = B \setminus E$.
Therefore, if $\mathcal{N}_{E}$ is Lebesgue measurable, then $E$ is admissible for the definition of normal traces.

Furthermore, by the definition, the normal trace of $\FF \in \DM^{p}(\Omega)$ on the boundary of a bounded Borel
set $E \subset \Omega$ is a distribution of order $1$ on $\R^{n}$,
since
\begin{equation*}
|\ban{\FF \cdot \nu, \phi}_{\partial E}|
\le \|\phi\|_{L^{\infty}(\R^{n})} |\div \FF|(E)
 + \| \nabla \phi \|_{L^{\infty}(\R^{n}; \R^{n})} |E|^{1 - \frac{1}{p}} \| \FF \|_{L^{p}(E; \R^{n})}
\end{equation*}
for any $\phi \in C^{1}_{c}(\R^{n})$.
Moreover, the normal trace is not stable {\it a priori} under the modifications of $E$
by Lebesgue negligible sets.
Indeed, if $\tildef{E}$ is any Borel set such that $|E \Delta \tildef{E}| = 0$,
then, unless $|\div \FF| \ll \Leb{n}$,
we may obtain that $|\div \FF|(E \Delta \tildef{E}) \neq 0$, even though the second terms in \eqref{normal trace def} are equal.

Therefore, the normal trace depends on the particular Borel representative of set $E$,
not even only on $\partial E$.
Indeed, if $U \subset \Omega$ is an open set with smooth boundary,
then $\partial U = \partial \overline{U}${\rm ;} however, when $|\div \FF|(\partial U) \neq 0$, the normal traces of $\FF$
on the boundary of $U$ and $\overline{U}$ are different in general.
\end{remark}

\begin{remark} \label{normal_trace_functional_fract_Sobolev}
\, By the definition of normal traces, we have
\begin{equation*}
\ban{\FF \cdot \nu, \phi}_{\partial E} = \div (\phi \FF)(E).
\end{equation*}
Therefore, Theorem {\rm \ref{product rule p}} implies that,
if $\FF \in \DM^{p}(\Omega)$ for $1 \le p \le \infty$,
$\ban{\FF \cdot \nu, \cdot}_{\partial E}$ as a functional can be extended to the space of
test functions $\phi \in C^{0}(\Omega) \cap L^{\infty}(\Omega)$ such that
$\nabla \phi \in L^{p'}(\Omega; \R^{n})$.
Under such conditions, we can also take any Borel set $E \subset \Omega$, since
\begin{equation*}
|\ban{\FF \cdot \nu, \phi}_{\partial E}|
\le \|\phi\|_{L^{\infty}(\Omega)} |\div \FF|(E)
 + \|\nabla \phi\|_{L^{p'}(\Omega; \R^{n})} \|\FF\|_{L^{p}(E; \R^{n})}.
\end{equation*}
Therefore, if $\FF \in \DM^{p}(\Omega)$ for $1 \le p \le \infty$,
and $E$ is a Borel set in $\Omega$,
then the normal trace $\ban{\FF \cdot \nu, \cdot}_{\partial E}$ can be extended to a functional
in the dual of $$\{ \phi \in C^{0}(\Omega) \cap L^{\infty}(\Omega)\, :\; \nabla \phi \in L^{p'}(\Omega; \R^{n}) \}.$$
\end{remark}

\begin{proposition} \label{support trace}
Let $\FF \in \DM^{p}(\Omega)$ for $1 \le p \le \infty$.
Then the normal trace of $\FF$ on the boundary of a bounded Borel set $E \subset \Omega$
is a distribution of order $1$ supported on $\partial E$.
\end{proposition}

\begin{proof}
\, Let $V \Subset \Omega \setminus \partial E$ and $\phi \in C^{1}_{c}(V)$.
We need to show that $\ban{\FF \cdot \nu, \phi}_{\partial E} = 0$.

Since $\phi \FF \in \DM^{p}(\Omega)$ and $\mathrm{supp}(\phi \FF) \subset V$, then $\mathrm{supp}(\div(\phi \FF)) \subset V$.
From this, it follows that
\begin{equation*}
\ban{\FF \cdot \nu, \phi}_{\partial E} = \div(\phi \FF)(E) = \div(\phi \FF)(V \cap \mathring{E}).
\end{equation*}
We may assume that $\mathring{E} \neq \emptyset$ (otherwise, there is nothing to prove)
and $V \subset \mathring{E}$, without loss of generality.
Then  \cite[Lemma 3.1]{comi2017locally}
implies that $\div(\phi \FF)(V) = 0$ so that
$\ban{\FF \cdot \nu, \phi}_{\partial E} = \div(\phi \FF)(V) = 0$.
\end{proof}

\begin{remark}
\, Given $\FF \in \DM^{p}(\Omega)$ for $1 \le p \le \infty$,
then, for any Borel set $E$ in $\Omega$,
the following locality property for the normal trace functional holds{\rm :}
\begin{equation*}
\ban{\FF \cdot \nu, \cdot}_{\partial E} = - \ban{\FF \cdot \nu, \cdot}_{\partial (\Omega \setminus E)}
\end{equation*}
in the sense of distributions on $\Omega$.

Indeed, given any $\phi \in C^{1}_{c}(\Omega)$, $\phi \FF \in \DM^{p}(\Omega)$ by Proposition {\rm \ref{product rule p}}, and
\begin{equation*}
\int_{\Omega} \phi \, \dd \div \FF + \int_{\Omega} \FF \cdot \nabla \phi \, \dr x
= \div(\phi \FF)(\Omega) = 0
\end{equation*}
by \cite[Lemma 3.1]{comi2017locally}, since $\mathrm{supp}(\phi \FF)$ is compact in $\Omega$.
Then
\begin{equation*}
\int_{E} \phi \, \dd \div \FF + \int_{E} \FF \cdot \nabla \phi \, \dr x
= - \int_{\Omega \setminus E} \phi \, \dd \div \FF - \int_{\Omega \setminus E} \FF \cdot \nabla \phi \, \dr x.
\end{equation*}
\end{remark}

\begin{theorem} \label{equivalence trace prod}
Let $\FF \in \DM^{p}(\Omega)$ for $1 \le p \le \infty$,
and let $E \subset \Omega$ be a bounded Borel set.
Then
\begin{equation} \label{trace_div_meas_repr}
\ban{\FF \cdot \nu, \cdot}_{\partial E} = \chi_{E} \div \FF - \div (\chi_{E} \FF)
\end{equation}
in the sense of distributions on $\Omega$.
Thus, $\ban{\FF \cdot \nu, \cdot}_{\partial E} \in \mathcal{M}(\partial E)$ if and only
if $\div (\chi_{E} \FF) \in \mathcal{M}(\Omega)${\rm ;} that is, $\chi_{E} \FF \in \DM^{p}(\Omega)$.
In addition, if $\ban{\FF \cdot \nu, \cdot}_{\partial E}$ is a measure,
then
\begin{enumerate}
\item[\rm (i)] $|\ban{\FF \cdot \nu, \cdot}_{\partial E}| \ll \Haus{n - 1} \res \partial E$, if $p = \infty${\rm ;}

\item[\rm (ii)] $|\ban{\FF \cdot \nu, \cdot}_{\partial E}|(B) = 0$ for any Borel set $B \subset \partial E$ with
$\sigma$--finite $\Haus{n - p'}$ measure, if $\frac{n}{n - 1} \le p < \infty$.
\end{enumerate}
\end{theorem}

\begin{proof}
\, By Proposition \ref{support trace}, the support of distribution $\ban{\FF \cdot \nu, \cdot}_{\partial E}$ is $\partial E$.
As for the equivalence, we notice that
\begin{equation*}
\ban{\FF \cdot \nu, \phi}_{\partial E} - \int_{E} \phi \, \dd \div \FF = \int_{E} \FF \cdot \nabla \phi \, \dr x
= \int_{\Omega} \chi_{E} \FF \cdot \nabla \phi \, \dr x
\end{equation*}
for any $\phi \in \Lip_{c}(\Omega)$.
This implies \eqref{trace_div_meas_repr} in the sense of distributions.
Since $\div \FF \in \mathcal{M}(\Omega)$,
it follows that $\ban{\FF \cdot \nu, \cdot}_{\partial E} \in \mathcal{M}(\partial E)$ if and only
if $\div(\chi_{E} \FF) \in \mathcal{M}(\Omega)$,
by the density of $\Lip_{c}(\Omega)$ in $C_{c}(\Omega)$ with respect to the supremum norm.
Since $\chi_{E} \FF \in L^{p}(\Omega; \R^{n})$,
then $\div(\chi_{E} \FF) \in \mathcal{M}(\Omega)$ implies that $\chi_{E} \FF \in \DM^{p}(\Omega)$.
As for the absolute continuity properties of the normal trace measure,
we argue as those in the end of the proof of Theorem \ref{prodruleBVgeneral},
by employing \eqref{trace_div_meas_repr} and \cite[Theorem 3.2]{Silhavy1}.
\end{proof}

We now employ \eqref{trace_div_meas_repr} to show the relation
between $\ban{\FF \cdot \nu, \cdot}_{\partial E}$
and $\ban{\FF \cdot \nu, \cdot}_{\partial \tildef{E}}$
for any another Borel representative $\tildef{E}$,
with respect to the Lebesgue measure, of a given bounded Borel set $E$.

\begin{proposition}
Let $\FF \in \DM^{p}(\Omega)$ for $1 \le p \le \infty$, and let $E, \tildef{E} \subset \Omega$ be bounded Borel
sets such that $|E \Delta \tildef{E}| = 0$. Then
\begin{equation} \label{difference_normal_trace_E_tilde_E}
\ban{\FF \cdot \nu, \cdot}_{\partial E} - \ban{\FF \cdot \nu, \cdot}_{\partial \tildef{E}}
= (\chi_{E \setminus \tildef{E}} - \chi_{\tildef{E} \setminus E} ) \div \FF,
\end{equation}
which means that $\ban{\FF \cdot \nu, \cdot}_{\partial E} - \ban{\FF \cdot \nu, \cdot}_{\partial \tildef{E}} \in \mathcal{M}(\Omega)$,
and
\begin{equation} \label{tot_var_difference_normal_trace_E_tilde_E}
|\ban{\FF \cdot \nu, \cdot}_{\partial E} - \ban{\FF \cdot \nu, \cdot}_{\partial \tildef{E}}|
= \chi_{E \Delta \tildef{E}} |\div \FF|.
\end{equation}
In particular, if $U$ is an open bounded set in $\Omega$ with $|\partial U| = 0$, then
\begin{equation} \label{difference_normal_trace_U_closure_U}
\ban{\FF \cdot \nu, \cdot}_{\partial \overline{U}} - \ban{\FF \cdot \nu, \cdot}_{\partial U}
= \chi_{\partial U}\, \div \FF.
\end{equation}
\end{proposition}

\begin{proof}
\, Since $|E \Delta \tildef{E}| = 0$,
$\div(\chi_{E} \FF) = \div(\chi_{\tildef{E}} \FF)$ in the sense of distributions.
Thus, by subtracting \eqref{trace_div_meas_repr} for $\tildef{E}$ from the same identity with $E$,
we obtain \eqref{difference_normal_trace_E_tilde_E}.
Then we see that $\ban{\FF \cdot \nu, \cdot}_{\partial E} - \ban{\FF \cdot \nu, \cdot}_{\partial \tildef{E}} \in \mathcal{M}(\Omega)$
and \eqref{tot_var_difference_normal_trace_E_tilde_E}.
Finally, if $U$ is open bounded set with $|\partial U| = 0$,
\eqref{difference_normal_trace_U_closure_U} follows from \eqref{difference_normal_trace_E_tilde_E}
with $E = \overline{U}$ and $\tildef{E} = U$.
\end{proof}

\begin{remark} \, While $\div(\chi_{E} \FF)$ is not a Radon measure in general,
we can employ \eqref{trace_div_meas_repr} to obtain some information on its restriction to some particular sets.
Indeed, since $\ban{\FF \cdot \nu, \cdot}_{\partial E}$ is supported on $\partial E$,
by Proposition {\rm \ref{support trace}}, it suffices to restrict  \eqref{trace_div_meas_repr}
to $\partial E$ and $\mathring{E}$ to obtain
\begin{align*}
\ban{\FF \cdot \nu, \cdot}_{\partial E}
 = \chi_{E \cap \partial E} \div \FF - \div(\chi_{E} \FF) \res \partial E,
 \quad \div \FF \res \mathring{E} - \div(\chi_{E} \FF) \res \mathring{E}=0.
\end{align*}
In particular, this means that $\div(\chi_{E} \FF) \res \mathring{E} = \div \FF \res \mathring{E}$,
so that this restriction is a Radon measure for any $\FF \in \DM^{p}(\Omega)$ and bounded Borel set $E$ in $\Omega$.
In addition, if $U$ is an open bounded set in $\Omega$, then
\begin{equation*}
\ban{\FF \cdot \nu, \cdot}_{\partial U} = - \div(\chi_{U} \FF) \res \partial U.
\end{equation*}
\end{remark}

We now state a particular result concerning the sets of finite perimeter and the case $p = \infty$,
which gathers much of the known
theory (see \cite{ctz,comi2017locally}).
It also provides a generalization of the Gauss-Green formulas by allowing for scalar functions
$\phi \in C^{0}(\Omega)$ with $\nabla \phi \in L^{1}_{\rm loc}(\Omega; \R^{n})$.
Such a result can be seen as a particular case of \cite[Theorem 5.1]{crasta2017anzellotti},
when $\Omega = \R^{n}$.

First, we need to recall the definitions of both {\em measure-theoretic interior} and {\em measure-theoretic boundary}
of a measurable set $E$:
\begin{align*}
E^{1}  := \Big\{ x \in \R^{n} : \lim_{r \to 0} \frac{|B(x,r) \cap E|}{|B(x,r)|} = 1 \Big\}, \quad
\partial^{m} E  := \R^{n} \setminus (E^{1} \cup (\R^{n} \setminus E)^{1}).
\end{align*}
By Lebesgue's differentiation theorem, it follows that $|E \Delta E^{1}| = 0$ and $|\partial^{m} E| = 0$.
By \cite[Lemma 5.9, \S 5.11]{eg}, $E^{1}$ and $\partial^{m} E$ are Borel measurable sets.

We notice that, if $\FF \in \DM^{\infty}_{\rm loc}(\Omega)$, and $E \subset \Omega$ is a set of locally finite perimeter,
then $\redb E$ is a $|\div \FF|$--measurable set.
Indeed, $\redb E \subset \partial^{m} E$ and $\Haus{n - 1}(\partial^{m} E \setminus \redb E) = 0$ by \cite[Lemma 5.5, \S 5.8]{eg}.
This means that $\partial^{m} E = \redb E \cup \mathcal{N}_{E}$ for some set $\mathcal{N}_{E}$ satisfying $\Haus{n - 1}(\mathcal{N}_{E}) = 0$.
Since $|\div \FF| \ll \Haus{n - 1}$ by \cite[Proposition 3.1]{CF1},
$\redb E$ is $|\div \FF|$--measurable, because it is the difference
between the Borel set $\partial^{m} E$ and the $|\div \FF|$--negligible set $\mathcal{N}_{E}$.
This means that, if $\FF \in \DM^{\infty}_{\rm loc}(\Omega)$, and $E \subset \Omega$ is a set of locally finite perimeter,
then $\ban{\FF \cdot \nu, \cdot}_{\partial E^{1}}$ and $\ban{\FF \cdot \nu, \cdot}_{\partial (E^{1} \cup \redb E)}$
are well defined.

\begin{proposition} \label{normal trace p infty}
Let $\FF \in \DM^{\infty}_{\rm loc}(\Omega)$, and let $E \Subset \Omega$ be a set of finite perimeter.
Then the normal trace of $\FF$ on the boundary of any Borel representative $\tildef{E}$ of set $E$
is a Radon measure supported on $\redb E \cup (\tildef{E} \Delta E^{1})\subset \partial\tilde{E}$.
In particular, if $\tildef{E} = E^{1}$ or $\tildef{E} = E^{1} \cup \redb E$ up to $\Haus{n - 1}$--negligible sets, then
$$
|\ban{\FF \cdot \nu, \cdot}_{\partial \tildef{E}}| \ll \Haus{n - 1} \res \redb E
$$
with density in $L^{\infty}(\redb E; \Haus{n - 1})$.
More precisely, for any set $E$ of locally finite perimeter in $\Omega$ and $\phi \in C^{0}(\Omega)$
such that $\nabla \phi \in L^{1}_{\rm loc}(\Omega; \R^n)$ and $\chi_{E} \phi$ has compact support in $\Omega$,
then
\begin{align}
&\int_{E^{1}} \phi \, \dd \div \FF + \int_{E} \FF \cdot \nabla \phi \, \dr x
 = - \int_{\redb E} \phi \, (\mathfrak{F}_{\ii} \cdot \nu_{E}) \, \dr \Haus{n - 1},  \label{G-G phi Sobolev int}\\
&\int_{E^{1} \cup \redb E} \phi \, \dd \div \FF + \int_{E} \FF \cdot \nabla \phi \, \dr x
 = - \int_{\redb E} \phi \, (\mathfrak{F}_{\rm e} \cdot \nu_{E}) \, \dr \Haus{n - 1}, \label{G-G phi Sobolev ext}
\end{align}
where $(\mathfrak{F}_{\ii} \cdot \nu_{E}), (\mathfrak{F}_{\rm e} \cdot \nu_{E}) \in L^{\infty}_{\rm loc}(\redb E; \Haus{n - 1})$
are the interior and exterior normal traces of $\FF$, respectively, as introduced in {\rm \cite[Theorem 5.3]{ctz}}.
\end{proposition}

\begin{proof}
\, Assume first that $E \Subset \Omega$.
By \cite[Theorem 5.3]{ctz} and \cite[Theorem 4.2]{comi2017locally},
it follows that the normal traces on the boundaries of $E^{1}$ and $E^{1} \cup \redb E$ are Radon measures.
They are indeed absolutely continuous with respect to $\Haus{n - 1} \res \redb E$
and with densities given by essentially bounded interior and exterior normal traces: For any $\phi \in \Lip_{c}(\Omega)$,
\begin{align*}
&\ban{\FF \cdot \nu, \phi}_{\partial E^{1}}
 = - \int_{\redb E} \phi \, (\mathfrak{F}_{\ii} \cdot \nu_{E}) \, \dr \Haus{n - 1}, \\
&\ban{\FF \cdot \nu, \phi}_{\partial (E^{1} \cup \redb E)}
 = - \int_{\redb E} \phi \, (\mathfrak{F}_{\ee} \cdot \nu_{E}) \, \dr \Haus{n - 1}.
\end{align*}

These two formulas hold also for any set $\tildef{E}$ with $\Haus{n - 1}(\tildef{E} \Delta E^{1}) = 0$
or $\Haus{n - 1}(\tildef{E} \Delta (E^{1} \cup \redb E)) = 0$, respectively,
since $|\div \FF| \ll \Haus{n - 1}$ if $\FF \in \DM^{\infty}(\Omega)$, by \cite[Theorem 3.2]{Silhavy1}.

Let $\tildef{E}$ be any Borel representative of $E$ with respect to the Lebesgue measure so that
$|E \Delta \tildef{E}| = 0$,
which implies that $\tildef{E}^{1} = E^{1}$.
By \eqref{difference_normal_trace_E_tilde_E}, we have
\begin{equation*}
\ban{\FF \cdot \nu, \phi}_{\partial \tildef{E}}
= \ban{\FF \cdot \nu, \phi}_{\partial E^{1}} + \int_{\Omega} \phi \, (\chi_{\tildef{E} \setminus E^{1}}
- \chi_{E^{1} \setminus \tildef{E}}) \, \dd \div \FF \,\,\, \mbox{for any $\phi \in \Lip_{c}(\Omega)$}.
\end{equation*}
This shows that $\ban{\FF \cdot \nu, \cdot}_{\partial \tildef{E}}$ is a Radon measure on $\redb E \cup (\tildef{E} \Delta E^{1})$,
while this set is contained in $\partial \tildef{E}$, since $\mathring{\tildef{E}} \subset E^{1} \subset \overline{\tildef{E}}$,
coherently with Proposition \ref{support trace}.

Finally, let $E$ be a set of locally finite perimeter in $\Omega$, and let $\phi \in C^{0}(\Omega)$
such that $\nabla \phi \in L^{1}_{\rm loc}(\Omega; \R^n)$ and $\mathrm{supp}(\chi_{E} \phi)\Subset \Omega$.
Then \eqref{G-G phi Sobolev int}--\eqref{G-G phi Sobolev ext} follow from \cite[Theorem 4.2]{comi2017locally}.
Indeed, such equations hold for $\phi \in \Lip_{\rm loc}(\Omega)$ such that $\mathrm{supp}(\chi_{E} \phi) \subset V \Subset \Omega$
for some open set $V$.
Thus, we can take any mollification $\phi_{\eps}$ of $\phi$, with $\eps > 0$ small enough,
such that $\mathrm{supp}(\chi_{E} \phi_{\eps}) \subset V$.
Then we pass to the limit as $\eps \to 0$ by employing the fact that $\phi_{\eps} \to \phi$ uniformly
on $V$ and $\nabla \phi_{\eps} \to \nabla \phi$ in $L^{1}(V; \R^{n})$.
This completes the proof.
\end{proof}

\begin{remark}
\, Given $\FF \in \DM^{\infty}_{\rm loc}(\Omega)$ and a set of locally finite perimeter $E \subset \Omega$,
\eqref{G-G phi Sobolev int}--\eqref{G-G phi Sobolev ext} hold for any $\phi \in \Lip_{c}(\Omega)$.
This shows that the normal traces of $\FF$ on the portion of the boundaries $\partial E^{1} \cap \Omega$
and $\partial(E^{1} \cup \redb E) \cap \Omega$ are locally represented by  measures
$(\mathfrak{F}_{i} \cdot \nu_{E}) \, \Haus{n - 1} \res \redb E$
and $(\mathfrak{F}_{e} \cdot \nu_{E})\,  \Haus{n - 1} \res \redb E$, respectively.
\end{remark}

\begin{remark}
\, Proposition {\rm \ref{normal trace p infty}} can be seen as a special case
of Theorem {\rm \ref{equivalence trace prod}}, because of Theorem {\rm \ref{productruleinfty}}.
In addition, it shows that the normal trace measures of $\FF \in \DM^{\infty}(\Omega)$ on $\partial E^{1}$
and  $\partial (E^{1} \cup \redb E)$ are actually concentrated on $\redb E = \redb E^{1} = \redb (E^{1} \cup \redb E)$,
for any set of finite perimeter $E \Subset \Omega$.

Moreover,
if $\FF \in \DM^{p}(\Omega)$ for $1 \le p < \infty$,
the normal trace on $\partial E$ is not
a measure that is absolutely continuous
with respect to $\Haus{n - 1}$ in general, as shown in \cite[Example 6.1]{comi2017locally}.
However, as we will see in \S 7,
the normal trace on the boundary of open and closed sets can still be represented
as the limit of the classical normal traces on an approximating family of smooth sets.
\end{remark}

\begin{remark}
\, Theorem {\rm \ref{equivalence trace prod}} shows that,
in the case of Example {\rm \ref{prodruleWhitney}},
the normal trace is a Radon measure on $\partial E$,
since a product rule holds between
$$
\FF(x_1, x_2) = \frac{1}{2 \pi} \frac{(x_1, x_2)}{x_1^{2} + x_2^{2}}
\quad\,\,\mbox{and} \quad\,\, \chi_{E} \,\,\,\mbox{for $E = (0, 1)^{2}$}.
$$
Indeed, we have
\begin{equation*}
\div(\chi_{E} \FF) = \frac{1}{4} \delta_{(0,0)} + (\FF, D \chi_{E})
\end{equation*}
with
\begin{equation*}
(\FF, D \chi_{E})(\phi) := - \frac{1}{2 \pi} \left ( \int_{0}^{1} \frac{\phi(x_{1}, 1)}{1 + x_{1}^{2}} \, \dr x_{1}
+ \int_{0}^{1} \frac{\phi(1, x_{2})}{1 + x_{2}^{2}} \, \dr x_{2} \right ).
\end{equation*}
Using \eqref{trace_div_meas_repr} and $(0, 0) \notin E$,
it follows that, for any $\phi \in \Lip_{c}(\R^{2})$,
\begin{align*}
\ban{\FF \cdot \nu, \phi}_{\partial E}
& = \int_{\R^{2}} \phi \chi_{E} \, \dd \div \FF - \int_{\R^{2}} \phi \, \dd \div(\chi_{E} \FF)
= - \int_{\R^{2}} \phi \, \dd \div(\chi_{E} \FF) \\
&  = - \frac{1}{4} \phi(0, 0) - (\FF, D \chi_{E})(\phi).
\end{align*}
Therefore, $\ban{\FF \cdot \nu, \cdot}_{\partial E}$ is a Radon measure on $\partial E$.

In this example, $E$ is also a set of finite perimeter with $E = E^{1}$,
but the normal trace is supported on $\partial E$,
not only on $\redb E$, since $(0, 0) \notin \redb E$.
\end{remark}

\begin{remark} \label{trace_no_meas_ex}
\, Theorem {\rm \ref{equivalence trace prod}} implies that,
if $\FF \in \DM^{p}(\Omega)$ does not admit a normal trace on $\partial E$ representable by a Radon measure,
then $\chi_{E} \FF \notin \DM^{p}(\Omega)$,
even for a set of locally finite perimeter $E$.

An example of such a vector field was provided by
\cite[Example 2.5]{Silhavy2} and \cite[Remark 2.2]{Frid2} as follows{\rm :}
$$
\FF(x_1, x_2) := \frac{(-x_2, x_1)}{x_1^{2} + x_2^{2}}.
$$
Then  $\FF \in \DM^{p}_{\rm loc}(\R^{2})$ for any $1 \le p < 2$,
$\div \FF = 0$ on  $E = (- 1, 1) \times (- 1, 0)$.
For any $\phi \in \Lip_{c}((-1, 1)^{2})$, we have
\begin{align*}
&\int_{(-1, 1)^{2}} \chi_{E} \FF \cdot \nabla \phi \, \dr x_1 \, \dr x_2\\[1mm]
& = \int_{-1}^{1} \int_{-1}^{0} \frac{1}{x_1^{2} + x_2^{2}} \left ( - x_2 \frac{\partial \phi}{\partial x_1}
   + x_1 \frac{\partial \phi}{\partial x_2} \right ) \, \dr x_2 \, \dr x_1 \\[1mm]
& = \lim_{\eps \to 0} \left ( \int_{-1}^{- \eps} + \int_{\eps}^{1} \right ) \int_{- 1}^{0} \frac{1}{x_1^{2} + x_2^{2}}
  \left ( - x_2 \frac{\partial \phi}{\partial x_1} + x_1 \frac{\partial \phi}{\partial x_2} \right ) \, \dr x_2 \, \dr x_1 \\[1mm]
& = \lim_{\eps \to 0}\left\{\int_{-1}^{0} \frac{x_{2}}{\eps^{2} + x_2^{2}} \big(- \phi(- \eps, x_2) +  \phi(\eps, x_2)\big) \, \dr x_2
+\left ( \int_{-1}^{- \eps} + \int_{\eps}^{1} \right ) \frac{\phi(x_1, 0)}{x_1} \, \dr x_1\right\}\\
& = {\rm P.V.} \int_{-1}^{1} \frac{\phi(x_1, 0)}{x_1} \, \dr x_1,
\end{align*}
since the two area integrals are simplified and
\begin{align*}
&\left | \lim_{\eps \to 0} \int_{-1}^{0} \frac{x_2}{\eps^{2} + x_2^{2}} \big(- \phi(- \eps, x_2) +  \phi(\eps, x_2)\big) \, \dr x_2 \right |\\
&\le 2L \lim_{\eps \to 0} \int_{0}^{1} \frac{\eps x_2}{\eps^{2} + x_2^{2}} \, \dr x_2
= L \lim_{\eps \to 0} \eps  \log ( 1 + \frac{1}{\eps^{2}})= 0,
\end{align*}
where $L$ is the Lipschitz constant of $\phi$.
This shows
\begin{equation*}
\div(\chi_{E} \FF) = {\rm P.V.}(\frac{1}{x_1})\res (- 1, 1) \otimes \delta_{0},
\end{equation*}
so that $\div( \chi_{E} \FF) \notin \mathcal{M}((-1, 1)^{2})$,
which means that $\chi_{E} \FF \notin \DM^{p}((-1, 1)^{2})$ for any $1 \le p < 2$.

The argument can be generalized to
\begin{equation*}
\FF(x_1, x_2) = \frac{(- x_2, x_1)}{(x_1^{2} + x_2^{2})^{\frac{\alpha}{2}}} \qquad\mbox{for $2 \le \alpha < 3$}
\end{equation*}
to obtain
\begin{equation*}
\div(\chi_{E} \FF) = ({\rm P. V.} \, \sgn{(x_1)} \, |x_1|^{1 - \alpha}) \res (- 1, 1) \otimes \delta_{0}.
\end{equation*}
\end{remark}

\begin{remark}\label{rem:3.22}
\,\, By Theorem {\rm \ref{equivalence trace prod}},
$\FF \in \DM^{p}(\Omega)$ admits a normal trace on the boundary of a Borel set $E \Subset \Omega$
representable by a Radon measure if and only if $\chi_{E} \FF \in \DM^{p}(\Omega)$.
This condition is generally weaker than the requirement of $E$ to be a set of locally finite perimeter in $\Omega$.
Indeed,
there exist a set $E \subset \R^{2}$
with $\chi_{E} \notin BV_{\rm loc}(\R^{2})$
and a field $\FF \in \DM^{p}(\R^{2})$ for any $p \in [1, \infty]$
with $\ban{\FF \cdot \nu, \cdot}_{\partial E} \in \mathcal{M}(\partial E)$.
The key observation in the construction of such a set $E$ is that, given a constant vector field $\FF \equiv \vv \in \R^{n}$,
\eqref{trace_div_meas_repr} implies that the normal trace is given by
\begin{equation*}
\ban{\vv \cdot \nu, \cdot}_{\partial E}
= - \div(\chi_{E} \vv) = - \sum_{j = 1}^{n} v_{j} D_{x_{j}} \chi_{E}.
\end{equation*}
Clearly, the requirement that $\sum_{j = 1}^{n} v_{j} D_{x_{j}} \chi_{E} \in \mathcal{M}(\Omega)$
is weaker than the requirement that $\chi_{E} \in BV(\Omega)$, since there may be some cancellations.

We choose $E$ as the open bounded set whose boundary is given by
\begin{equation*}
\partial E =\big( \{0\} \times [0, 1]\big) \cup \big([0, 1] \times \{0\}\big) \cup \big([0, 1 + \log{2}] \times \{1\}\big) \cup S,
\end{equation*}
as shown in Figure {\rm \ref{fig-4.1}}, where
\begin{align*}
S & = \Big(\{1\} \times [0, \frac{1}{2}]\Big) \bigcup \Big([1, 2] \times \{ \frac{1}{2}\}\Big)\\
&\quad  \bigcup \Big(\bigcup_{n \ge 1} \{ 1 + \sum_{k = 1}^{n} \frac{(-1)^{k - 1}}{k}\} \times [ 1 - \frac{1}{2^{n}}, 1 - \frac{1}{2^{n + 1}}]\Big)  \\
&\quad  \bigcup \Big(\bigcup_{n \ge 1} [ 1 + \sum_{k = 1}^{2n} \frac{(-1)^{k - 1}}{k} , 1 + \sum_{k = 1}^{2n + 1} \frac{(-1)^{k - 1}}{k}]\times \{ 1 - \frac{1}{2^{2n + 1}}\}\Big) \\
&\quad \bigcup \Big(\bigcup_{n \ge 1} [ 1 + \sum_{k = 1}^{2n} \frac{(-1)^{k - 1}}{k} , 1 + \sum_{k = 1}^{2n - 1} \frac{(-1)^{k - 1}}{k}] \times \{ 1 - \frac{1}{2^{2n}} \}\Big).
\end{align*}
\begin{figure} \label{fig-4.1}
  \centering
      \includegraphics[width=0.7\textwidth]{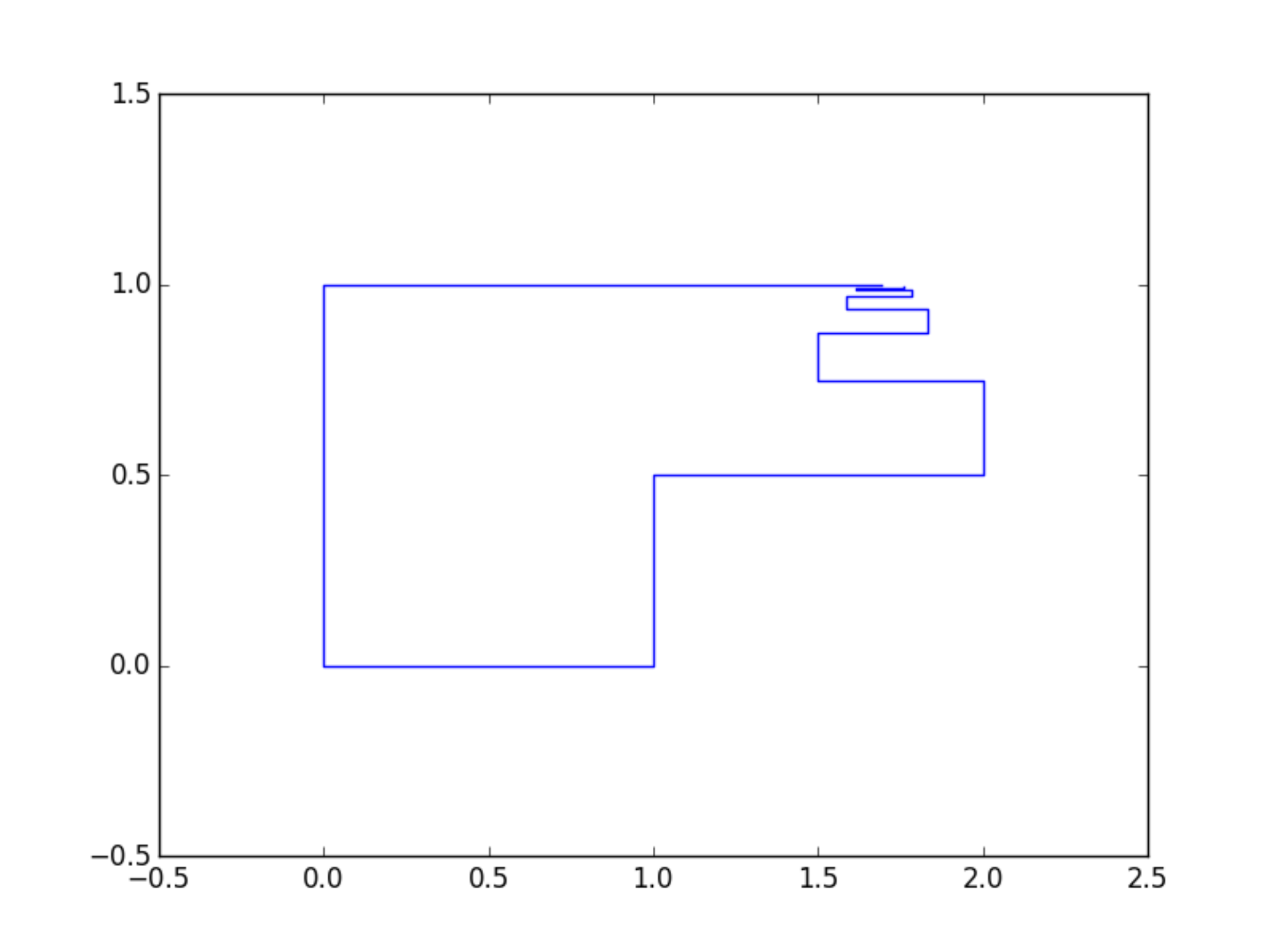}
  \caption{The open bounded set $E$}
\end{figure}
Then $\chi_{E} \notin BV_{\rm loc}(\R^{2})$, since $\Haus{1}(S) = \infty$.
However, we can show that $D_{x_1} \chi_{E} \in \mathcal{M}(\R^{2})$.

Indeed, given any $\phi \in C^{1}_{c}(\R^{2})$, we have
\begin{align*}
&\int_{E} \frac{\partial \phi}{\partial x_1} \, \dr x_1 \, \dr x_2\\
& = \int_{0}^{\frac{1}{2}} \int_{0}^{1} \frac{\partial \phi}{\partial x_1} \, \dr x_1 \, \dr x_2
 + \sum_{n = 1}^{\infty} \int_{1 - \frac{1}{2^{n}}}^{1 - \frac{1}{2^{n + 1}}}
  \int_{0}^{1 + \sum_{k = 1}^{n} \frac{(-1)^{k - 1}}{k}} \frac{\partial \phi}{\partial x_1} \, \dr x_1 \, \dr x_2 \\
& = \int_{0}^{\frac{1}{2}} \big(\phi(1, x_2) - \phi(0, x_2)\big) \, \dr x_2 \\
&\quad  + \sum_{n = 1}^{\infty} \int_{1 - \frac{1}{2^{n}}}^{1 - \frac{1}{2^{n + 1}}}
  \big(\phi (1 + \sum_{k = 1}^{n} \frac{(-1)^{k - 1}}{k}, x_2 ) - \phi(0, x_2)\big) \, \dr x_2 \\
& = - \int_{0}^{1} \phi(0, x_2) \, dx_2 + \int_{0}^{\frac{1}{2}} \phi(1, x_2) \, \dr x_2
 + \sum_{n = 1}^{\infty}
   \int_{1 - \frac{1}{2^{n}}}^{1 - \frac{1}{2^{n + 1}}} \phi(1 + \sum_{k = 1}^{n} \frac{(-1)^{k - 1}}{k}, x_2) \, \dr x_2.
\end{align*}
This implies
\begin{align*} \label{derivative_x_E}
&D_{x_{1}} \chi_{E} \\
&=\, \Haus{1} \res \big(\{0\} \times (0, 1)\big) \\
&\quad - \Haus{1} \res \Big(\big(\{1\} \times (0, \frac{1}{2})\big) \bigcup \big (\bigcup_{n \ge 1}\{ 1 + \sum_{k = 1}^{n} \frac{(-1)^{k - 1}}{k} \}
        \times ( 1 - \frac{1}{2^{n}}, 1 - \frac{1}{2^{n + 1}}) \big)\Big),
\end{align*}
which is clearly a finite Radon measure on $\R^{2}$.

Now we observe that, if $\FF(x_1, x_2) = f(x_2) g(x_1) (1, 0)$ for some $f \in L^{p}(\R)$ and $g \in C^{1}_{c}(\R)$,
then $\FF \in \DM^{p}(\R^{2})$,
$$
\div \FF = f(x_2) g'(x_1) \Leb{2},
$$
and
\begin{equation} \label{divergence_chi_E_F_example}
\div(\chi_{E}\FF) = f(x_2) g(x_1) D_{x_1} \chi_{E} + \chi_{E}(x_1,x_2) f(x_2) g'(x_1) \Leb{2}.
\end{equation}
Indeed, for any $\phi \in C^{1}_{c}(\R^{2})$, we have
\begin{align*}
\int_{\R^{2}} \chi_{E} \FF \cdot \nabla \phi \, \dr x_1 \dr x_2
& = \int_{\R^{2}} \chi_{E}(x_1, x_2) f(x_2) g(x_1) \frac{\partial \phi(x_1, x_2)}{\partial x_1} \, \dr x_1 \dr x_2 \\
& = \int_{\R^{2}} \chi_{E}(x_1,x_2) f(x_2) \frac{\partial (g(x_1) \phi(x_1, x_2))}{\partial x_1} \, \dr x_1 \dr x_2  \\
&\quad - \int_{\R^{2}} \chi_{E}(x_1, x_2) f(x_2) \phi(x_1, x_2) g'(x_1) \, \dr x_1 \dr x_2 \\
& = - \int_{\R^{2}} f(x_2) g(x_1) \phi(x_1, x_2) \, \dr  D_{x_1} \chi_{E} \\
&\quad - \int_{\R^{2}} \chi_{E}(x_1, x_2) f(x_2) \phi(x_1, x_2) g'(x_1) \, \dr x_1 \dr x_2.
\end{align*}
Thus, by \eqref{divergence_chi_E_F_example}, $\div(\chi_{E} \FF) \in \mathcal{M}(\R^{2})$
so that $\ban{\FF \cdot \nu, \cdot}_{\partial E} \in \mathcal{M}(\partial E)$, by Theorem {\rm \ref{equivalence trace prod}},
even if $E$ is not a set of locally finite perimeter in $\R^{2}$.
In addition, by \eqref{trace_div_meas_repr}, we have
\begin{equation*}
\ban{\FF \cdot \nu, \cdot}_{\partial E} = \chi_{E} \div \FF - \div(\chi_{E} \FF) = - f(x_2) g(x_1) D_{x_1} \chi_{E},
\end{equation*}
from which the following is deduced{\rm :}
\begin{align*}
|\ban{\FF \cdot \nu, \cdot}_{\partial E}| \ll \Haus{1} \res & \Big( \big(\{0\} \times (0, 1)\big) \bigcup \big(\{1\} \times (0, \frac{1}{2})\big)
   \\
  &\quad \bigcup \big(\bigcup_{n \ge 1} \{ 1 + \sum_{k = 1}^{n} \frac{(-1)^{k - 1}}{k}\}
    \times ( 1 - \frac{1}{2^{n}}, 1 - \frac{1}{2^{n + 1}})\big)\Big).
\end{align*}
\end{remark}

On the other hand, as we will show,
whether
$\ban{\FF \cdot \nu, \cdot}_{\partial E}$ is a Radon measure on $\partial E$ or not
does not play any role in the representation of the normal trace of $\FF$
on the boundary of an open or closed set as the limit of classical normal traces
on the boundaries of a sequence of approximating smooth sets.

We provide now a necessary condition for the normal trace to be a Radon measure.
\begin{proposition} \label{necessary_condition_trace_measure}
Let $\FF \in \DM^{p}(\Omega)$ for $1 \le p \le \infty$,
and let $E \subset \Omega$ be a Borel set
such that there exists $\sigma \in \mathcal{M}(\partial E)$ satisfying
\begin{equation*}
\ban{\FF \cdot \nu, \phi}_{\partial E} = \int_{\partial E} \phi \, \dr \sigma
\qquad \mbox{for any $\phi \in \Lip_{\rm c}(\R^n)$}.
\end{equation*}
Then, if $1 \le p < \frac{n}{n - 1}$, for any $x \in \partial E$ and $r > 0$,
there exists a constant $C >0$ such that $|\sigma|(\partial E) + |\div \FF|(E)\ge C$ and
\begin{equation} \label{necess_cond_trace_meas_sub_critical}
\Big| \int_{B(x, r) \cap E} \FF(y) \cdot \frac{(y - x)}{|y - x|} \, \dr y \Big| \le C r.
\end{equation}
If $p \ge \frac{n}{n - 1}$, for any $x \in \partial E$ and $r > 0$,
\begin{equation} \label{necess_cond_trace_meas_super_critical}
\Big| \int_{B(x, r) \cap E} \FF(y) \cdot \frac{(y - x)}{|y - x|} \, \dr y \Big| = o(r).
\end{equation}
Moreover, given any $\alpha \in (0, n]$, for $\Haus{\alpha}$--{\it a.e.} $x \in \partial E$ and $r > 0$,
there exists a constant $C = C_{E, \FF, x} > 0$ such that
\begin{equation} \label{necess_cond_trace_meas_Haus_a_e}
\Big| \int_{B(x, r) \cap E} \FF(y) \cdot \frac{(y - x)}{|y - x|} \, \dr y \Big| \le C r^{\alpha + 1}.
\end{equation}
\end{proposition}

\begin{proof}
\, We just need to choose $\phi(y) := (r - |y - x|) \chi_{B(x, r)}(y)$ so that,
by \eqref{normal trace def},
\begin{align*}
&\int_{B(x, r)\cap \partial E} (r - |y - x|) \, \dr \sigma(y)\\
&= \int_{B(x, r)\cap E} (r - |y - x|) \, \dr \div \FF
  - \int_{B(x, r)\cap E} \FF(y) \cdot \frac{(y - x)}{|y - x|} \, \dr y.
\end{align*}
Then
we obtain
\begin{equation} \label{necess_cond_inter}
\left | \int_{B(x, r) \cap E} \FF(y) \cdot \frac{(y - x)}{|y - x|} \, \dr y \right |
\le r \Big (|\sigma|( B(x, r)\cap \partial E) + |\div \FF|(B(x, r)\cap E) \Big ).
\end{equation}

Now, if $1 \le p < \frac{n}{n - 1}$,
then $\div F$ and $\sigma = \chi_{E} \div F - \div(\chi_{E} F)$ do not enjoy any absolute continuity property in general,
by \cite[Example 3.3, Proposition 6.1]{Silhavy1}, so that \eqref{necess_cond_trace_meas_sub_critical} holds from \eqref{necess_cond_inter}.

If $p \ge \frac{n}{n - 1}$,
then $|\div F|(\{x\}) = |\sigma|(\{x\}) = 0$, by \cite[Theorem 3.2]{Silhavy1} and Theorem \ref{equivalence trace prod}.
Therefore, \eqref{necess_cond_inter} implies \eqref{necess_cond_trace_meas_super_critical}.
Finally, a consequence of \cite[Theorem 2.56]{afp} is that, given a positive Radon measure $\mu$ on $\Omega$,
its $\alpha$--dimensional upper density $\Theta^{*}_{\alpha}(\mu, x)$ satisfies the property:
\[
\Theta^{*}_{\alpha}(\mu, x) < \infty \qquad\,\, \text{for $\Haus{\alpha}$--{\it a.e.} $x \in \Omega$}.
\]
This means that, for $\Haus{\alpha}$--{\it a.e.} $x \in \Omega$, there exists a constant $C = C_{\mu, x}$ such that
$$
\mu(B(x, r)) \le C r^{\alpha}.
$$
Therefore, this argument holds for both measures $|\div F| \res E$ and $|\sigma| \res \partial E$.
Then,  from \eqref{necess_cond_inter}, we achieve \eqref{necess_cond_trace_meas_Haus_a_e}.
\end{proof}

\begin{remark}
\, The result of Proposition {\rm \ref{necessary_condition_trace_measure}} does not seem
to be very restrictive, since
the example in Remark {\rm \ref{trace_no_meas_ex}} satisfies
all the three conditions at any point on $(-1, 1) \times \{0\}$.

Indeed, consider points $(t, 0)$ for some $t \in (-1, 1)$, and $r > 0$ small enough so that
$B((t, 0), r) \cap \{ x_{2} < 0 \} \subset E = (-1, 1) \times (-1, 0)$.
Since $(- x_{2}, x_{1}) \cdot (x_{1} - t, x_{2}) = x_{2} t$, we have
\begin{align*}
&\int_{B((t, 0), r) \cap E} \FF(x_{1}, x_{2}) \cdot \frac{(x_{1} - t, x_{2})}{|(x_{1} - t, x_{2})|}  \, \dr x_{1} \, \dr x_{2}  \\
& = \int_{B((t, 0), r) \cap \{ x_{2} < 0 \}} \frac{x_{2} t}{(x_{1}^{2} + x_{2}^{2}) \sqrt{(x_{1} - t)^{2} + x_{2}^{2}}} \, \dr x_{1} \, \dr x_{2} \\
& = \int_{B((0, 0), 1) \cap \{ u < 0 \}} \frac{t u}{((t + rv)^{2} + r^{2} u^{2}) \sqrt{v^{2} + u^{2}}} r^{2} \, \dr v \, \dr u.
\end{align*}
Therefore, for $t \neq 0$, we have
\begin{align*}
&\bigg| \int_{B((t, 0), r) \cap E} \FF(x_{1}, x_{2}) \cdot \frac{(x_{1} - t, x_{2})}{|(x_{1} - t, x_{2})|} \, \dr x_{1} \, \dr x_{2} \bigg|\\
& = \int_{0}^{1} \int_{- \frac{\pi}{2}}^{\frac{\pi}{2}} \frac{\rho |t| \cos{\theta}}{r^{2} \rho^{2} + t^{2} + 2t r \rho \sin{\theta}} r^{2} \, \dr \theta \, \dr \rho \\
& = \int_{0}^{1} r\, {\rm sgn}(t) \log{\left | \frac{r \rho + t}{r \rho - t} \right |} \, \dr \rho  = \frac{r^{2}}{|t|} + o(r^2)
\end{align*}
for any sufficiently small $r${\rm ;} while, if $t = 0$, we just have
\begin{equation*}
\int_{B((0, 0), r) \cap E} \FF(x_{1}, x_{2}) \cdot \frac{(x_{1}, x_{2})}{|(x_{1}, x_{2})|} \, \dr x_{1} \, \dr x_{2} = 0 \qquad \mbox{for any $r > 0$}.
\end{equation*}
These calculations also show
that this $\FF$ satisfies \eqref{necess_cond_trace_meas_Haus_a_e} for any $\alpha \in (0, 1]$,
which is sufficient, since the Hausdorff dimension of $\partial E$ is $1$.

Moreover, for any $\FF \in L^{p}(\Omega; \R^{n})$,
\begin{equation*}
\Big| \int_{B(x, r) \cap E} \FF(y) \cdot \frac{(y - x)}{|y - x|} \, \dr y \Big|
\le \Big( \int_{B(x, r)} |\FF|^{p} \, \dr y \Big)^{\frac{1}{p}} (\omega_{n} r^{n})^{1 - \frac{1}{p}}.
\end{equation*}
Then condition \eqref{necess_cond_trace_meas_super_critical} is satisfied for any $r \in (0, 1]$
if $\displaystyle n - \frac{n}{p} > 1${\rm ,} that is, $p > \frac{n}{n - 1}$.

On the other hand, we obtain a better decay estimate for $\Haus{\alpha}$--{\it a.e.} $x \in \partial E$,
for any $\alpha \in (0, n]$.
Indeed, $\FF \in L^{1}_{\rm loc}(\Omega; \R^{n})$ so that measure $\mu = |\FF| \Leb{n}$
satisfies $\mu(B(x, r)) \le C r^{\alpha}$ for $\Haus{\alpha}$--{\it a.e.} $x \in \Omega$.
This implies
\begin{equation*}
\Big| \int_{B(x, r) \cap E} \FF(y) \cdot \frac{(y - x)}{|y - x|} \, \dr y \Big|  \le C r^{\alpha}
\qquad\mbox{for $\Haus{\alpha}$--{\it a.e.} $x \in \partial E$},
\end{equation*}
while we obtain the higher exponent $\alpha + 1$  in \eqref{necess_cond_trace_meas_Haus_a_e}.
\end{remark}

\section{\, The Gauss-Green Formula on General Open Sets}

We now consider a general open set $U\subset \R^{n}$ and provide a way to construct its interior and
exterior approximations  via the signed distance function,
suitable for the derivation of the Gauss-Green formula
for $\FF \in \DM^{p}$ for $1 \le p \le \infty$.

For the given  open set $U$ in $\R^{n}$, we consider the signed distance from $\partial U$:
\begin{equation}
d(x)=
\begin{cases}
\dist(x, \partial U) &\quad \mbox{for} \ \ x \in U,\\
-\dist(x, \partial U) &\quad \mbox{for}  \ \ x \notin U.
\end{cases}
\end{equation}
We summarize some known results on the signed distance function in the following lemma.

\begin{lemma} \label{properties distance}
The distance function $d(x)$ is Lipschitz in $\R^n$ with Lipschitz constant equal to $1$ and satisfies
$$
|\nabla d(x)| = 1\qquad \mbox{for $\Leb{n}$--{\it a.e.} $x \notin \partial U$}.
$$
In addition, $\nabla d = 0$ $\Leb{n}$--{\it a.e.} on sets $\{ d = t \}$ for any $t \in \R$.
\end{lemma}

\begin{proof}
\, The elementary properties of the distance show that $d$ is Lipschitz with Lipschitz constant $L \le 1$,
and hence differentiable $\Leb{n}$--{\it a.e.}.
Then it is clear that $|\nabla d(x)| \le 1$.

Let now $x \in U$ such that $d$ is differentiable at $x$ and $|\nabla d(x)|<1$.
Then there exists a point $y \in \partial U$, depending on $x$, such that $d(x) = |x - y|$.
Indeed, given $z \in \partial U$ such that $d(x) \le |x - z|$,
then we can look for $y \in \overline{B}(x, |x - z|) \cap \partial U$, which is a compact set.

Setting $x_{r}:= x + r (y - x)$, we see that $d(x_{r}) = (1 - r)|x - y|$ for any $r \in [0, 1]$.
Otherwise, if there would exist $z \in \partial U$ such that $|x_{r} - z| < |x_{r} - y|$, then we would obtain
\begin{equation*}
|x - z| \le |x - x_{r}| + |x_{r} - z| < r |y - x| + |x_{r} - y| = |x - y|,
\end{equation*}
which contradicts the assumption that $y$ realizes the minimum distance from $x$.

Since $d$ is differentiable at $x$, then
\begin{equation*}
d(x_{r}) - d(x) = \nabla d(x) \cdot (y - x)\,r + o(r),
\end{equation*}
that is,
\begin{equation*}
|y - x| = \nabla d(x) \cdot (x - y) + o(1),
\end{equation*}
which yields a contradiction with the assumption that $|\nabla d(x)| < 1$.

Similarly, we also obtain a contradiction, provided that $d$ is differentiable
at $x\in\R^b\setminus\overline{U}$ and $|\nabla d(x)|<1$.
Since the Lipschitz constant $L$
satisfies $L \ge \| \nabla d \|_{L^{\infty}(\R^{n}, \R^{n})}$, we conclude that $L = 1$.

As for the second part of the statement, we refer to \cite[Theorem 3.2.3]{Ambrosio_Tilli}.
\end{proof}

For any $\eps > 0$, denote
\begin{equation} \label{U^eps}
U^{\eps} := \{ x \in \R^{n} : d(x) > \eps \},
\end{equation}
and
\begin{equation} \label{U_eps}
U_{\eps} := \{ x \in \R^{n} : d(x) > -\eps \}.
\end{equation}
Then $U^{\eps '} \subset U^{\eps}$ when $\eps ' > \eps$, and
$$
\bigcup_{\eps > 0} U^{\eps} = U.
$$
Similarly, $U_{\ve'}\subset U_{\ve}$ when $\eps ' < \eps$, and
$$
\bigcap_{\eps > 0} U_{\eps} = \overline{U}.
$$
It is clear that, for $K:=\overline{U}$ and $K_{\eps} := \overline{U_{\eps}}$,
we recover the same setting of Schuricht's result as in \eqref{Normal trace average Schuricht}.

\begin{remark} \label{boundary negl}
\, By Lemma {\rm \ref{properties distance}}, we can integrate indifferently on $\{d \ge t \}$ and $\{d > t \}$ for any $t \in \R$
with respect to $\nabla d \, \dr x$ {\rm (}or, analogously, on $\{ d \le t \}$ and $\{ d < t \}${\rm )}.
This means that  $\partial U^{\eps}$ and $\partial U_{\eps}$ are negligible
for measure $\nabla d \, dx$ for any $\eps \ge 0$ {\rm (}with $U^{0} = U_{0} = U${\rm )}.
In particular, it follows that \eqref{coarea super-sublevel} holds for any $t \in \R$ and $h \ge 0$, if $u = d$ and $g = f \, \nabla d$
for some $f : \R^{n} \to \R$ $\Leb{n}$--summable.
\end{remark}

We can say more on the regularity of sets $U^{\eps}$ and $U_{\ve}$.
Indeed, since $d$ is a Lipschitz function, which is particularly in $BV_{\rm loc}(\R^{n})$,
the coarea formula (Theorem \ref{Federer-Fleming coarea}) implies that the superlevel and sublevel sets of $d$ are almost all sets of locally finite perimeter.
Thus, we can conclude that $U^{\eps}$ and $U_{\ve}$ are sets of locally finite perimeter for $\Leb{1}$--{\it a.e.} $\eps >0$.
In fact, we can show the following slightly stronger result.

\begin{lemma} \label{boundary tubular ngb}
For any open set $U$ in $\R^{n}$, for $\Leb{1}$--{\it a.e.} $\eps > 0$,
$$
\Haus{n - 1}(\partial U^{\eps} \setminus \redb U^{\eps}) = 0,
\qquad \,
\nabla d(x) = \nu_{U^{\eps}}(x) \quad \mbox{for $\Haus{n - 1}$--{\it a.e.} $x \in \partial U^{\eps}$},
$$
where $\nu_{U^{\eps}}$ is the measure-theoretic interior normal to $U^{\eps}$.
Analogously, for $\Leb{1}$--{\it a.e.} $\eps > 0$,
$$
\Haus{n - 1}(\partial U_{\eps} \setminus \redb U_{\eps}) = 0, \qquad \,
\nabla d(x) = \nu_{U_{\eps}}(x) \quad\mbox{for $\Haus{n - 1}$--{\it a.e.} $x \in \partial U_{\eps}$}.
$$
\end{lemma}

\begin{proof}
\, By the previous remarks, $U^{\eps}$ is a set of locally finite perimeter for $\Leb{1}$--{\it a.e.} $\eps > 0$.
Then, for any smooth vector field $\varphi \in C^{1}_{c}(\R^{n}; \R^{n})$,
\begin{equation}\label{G-G smooth 1}
\int_{U^{\eps}} \div\,\varphi \, \dr x
= - \int_{\redb U^{\eps}} \varphi \cdot \nu_{U^{\eps}} \, \dr \Haus{n - 1}
\qquad \mbox{for $\Leb{1}$--{\it a.e.} $\eps > 0$}.
\end{equation}
Consider now the functions:
\begin{equation*}
\psi^{U}_{\eps}(x)
:= \begin{cases}
        \eps & \mbox{if} \ x \in U^{\eps}, \\
        d(x) & \mbox{if} \ x \in U \setminus U^{\eps}, \\
        0 & \mbox{if} \ x \notin U.
 \end{cases}
\end{equation*}
Then
\begin{align*}
\int_{U} \psi_{\eps}^U \div\,\varphi \, \dr x
& = \eps \int_{U^{\eps}} \div\,\varphi \, \dr x
  + \int_{U \setminus U^{\eps}} d(x)\, \div\,\varphi \, \dr x \quad \\
& = - \int_{U \setminus U^{\eps}} \varphi \cdot \nabla d \, \dr x \\
& = - \int_{U \setminus \overline{U^{\eps}}} \varphi \cdot \nabla d \, |\nabla d| \, \dr x \\
& = - \int_{0}^{\eps} \int_{\partial U^{t}} \varphi \cdot \nabla d \, \dr \Haus{n - 1} \, \dr t,
\end{align*}
since $|\nabla d(x)| = 1$ for $\Leb{n}$--{\it a.e.} $x \notin \partial U$ and $\nabla d(x) = 0$
for $\Leb{n}$--{\it a.e.} $x \in \partial U^{\eps}$ (Lemma \ref{properties distance} and Remark \ref{boundary negl}),
and by the coarea formula \eqref{coarea super-sublevel} with $u = d$ and $g = \chi_{U} \varphi \cdot \nabla d \, |\nabla d|$.
Indeed, using \eqref{coarea superlevel}, we have
\begin{align*}
\int_{\R^{n} \setminus \overline{U^{\eps}}} \chi_{U} \varphi \cdot \nabla d \, |\nabla d| \, \dr x
& = \int_{\{d < \eps\}} \chi_{U} \varphi \cdot \nabla d \, |\nabla d| \, \dr x \\
& = \int_{\{- d > - \eps\}} \chi_{U} \varphi \cdot \nabla d \, |\nabla d| \, \dr x \\
& = \int_{- \eps}^{\infty} \int_{\{- d = t\}} \chi_{U} \varphi \cdot \nabla d \, \dr \Haus{n - 1} \, \dr t \\
& = \int_{- \infty}^{ \eps} \int_{\{d = t\}} \chi_{U} \varphi \cdot \nabla d \, \dr \Haus{n - 1} \, \dr t \\
& = \int_{0}^{\eps} \int_{\partial U^{t}} \varphi \cdot \nabla d \, \dr \Haus{n - 1} \, \dr t,
\end{align*}
since $d > 0$ in $U$ and $|\nabla d(x)|=1$ for $\Leb{n}$--{\it a.e.} $x\in \partial U$.

We can repeat the same calculation with $\psi_{\eps + h}$ for $h > 0$, and subtract the two resultant equations to obtain
\begin{equation*}
h \int_{U^{\eps + h}} \dd \div \varphi + \int_{U^{\eps } \setminus U^{\eps + h}} (d(x) - \ve ) \, \dd \div \varphi
= - \int_{\eps}^{\eps + h} \int_{\partial U^{t}} \varphi \cdot \nabla d \, \dr\Haus{n - 1} \, \dr t.
\end{equation*}

We now divide by $h$ and use the fact that $0 \le d(x) - \eps \le h$ in $U^{\eps} \setminus U^{\eps + h}$
and $|U^{\eps} \setminus U^{\eps + h}| \to 0$ as $h \to 0$ to conclude
\begin{equation}
\label{G-G smooth 2}
 \int_{U^{\eps}} \div \varphi \, \dr x
 = - \int_{\partial U^{\eps}} \varphi \cdot \nabla d \, \dr \Haus{n - 1}\qquad\mbox{for $\Leb{1}$--{\it a.e.} $\eps > 0$}.
 \end{equation}

Notice that,  for any $R > 0$,
\begin{align}
& \Haus{n - 1}(B(0, R) \cap \redb U^{\eps})\nonumber\\
&= \sup \Big\{ \int_{U^{\eps}} \div (-\varphi) \, \dr x\, :\, \varphi \in C^{1}_{c}(B(0, R); \R^{n}), \|\varphi\|_{\infty} \le 1 \Big\}
\label{yamero}.
\end{align}
Now, we can take a double index sequence of fields $\varphi_{k, m}$ in $C^{1}_{c}(B(0, R); \R^{n})$ such that
$$
\varphi_{k, m} \to \chi_{B(0, R - \frac{1}{m})} \nabla d \qquad \mbox{in $L^{1}(\R^{n}; \R^{n})$ as $k \to \infty$, for any fixed $m \in \mathbb{N}$}.
$$
For each $k$ and $m$, there is a set $\mathcal{N}_{k, m} \subset \R$ with $\Leb{1}(\mathcal{N}_{k, m}) = 0$
such that \eqref {G-G smooth 2} holds for any $\eps \notin \mathcal{N}_{k, m}$.
Set $\mathcal{N} := \bigcup_{(k, m) \in \N^{2} } \mathcal{N}_{k, m}$.
Then, for any $\eps \notin \mathcal{N}$, we obtain
\begin{align*}
&\sup \Big\{ \int_{U^{\eps}} \div (-\varphi) \, \dr x \,: \, \varphi \in C^{1}_{c}(B(0, R); \R^{n}), \|\varphi\|_{\infty} \le 1 \Big\}\\
& \ge \int_{U^{\eps}} \div ( - \varphi_{k, m}) \, \dr x = \int_{\partial U^{\eps}} \varphi_{k, m} \cdot \nabla d \, \dr \Haus{n - 1}.
\end{align*}
Now we let $k \to \infty$ and employ \eqref{yamero} to obtain
\begin{equation*}
\Haus{n - 1}(B(0, R) \cap \redb U^{\eps}) \ge  \Haus{n - 1}(B(0, R - \frac{1}{m}) \cap \partial U^{\eps}),
\end{equation*}
and the arbitrariness of $m \in \N$ yields
\begin{equation}
 \label{alla}
  \Haus{n - 1}(B(0, R) \cap \redb U^{\eps}) \ge  \Haus{n - 1}(B(0, R) \cap \partial U^{\eps}).
\end{equation}
Combining \eqref{alla} with the well-known fact that $\redb U^{\eps} \subset \partial U^{\eps}$, we obtain
$$
\Haus{n - 1}(B(0, R) \cap (\partial U^{\eps} \setminus \redb U^{\eps})) = 0,
$$
which implies that $\Haus{n - 1}(\partial U^{\eps} \setminus \redb U^{\eps}) = 0$, by the arbitrariness of $R > 0$.

Therefore, from \eqref{G-G smooth 1}--\eqref{G-G smooth 2}, we have
\begin{equation*}
\int_{\redb U^{\eps}} \varphi \cdot (\nu_{U^{\eps}} - \nabla d ) \, \dr \Haus{n - 1} = 0
\qquad\,\, \mbox{for any $\varphi \in C^{1}_{c}(\R^{n}; \R^{n})$},
\end{equation*}
which implies our assertion.

The second part of the statement is proved in a similar way  by considering the following functions instead:
\begin{equation*}
\xi^{U}_{\eps}(x) :=
\begin{cases}
\ve & \mbox{if} \ x \in U, \\
d(x)+\ve  & \mbox{if} \ x \in U_{\eps} \setminus U, \\
0 & \mbox{if} \ x \notin U_{\eps}.
\end{cases}
\end{equation*}
\end{proof}

Using similar techniques as in the proof of Lemma \ref{boundary tubular ngb}, we are able to show the following Gauss-Green formulas.

\begin{theorem}[Interior normal trace] \label{interior normal trace}
Let $U \subset \Omega$ be a bounded open set, and let $\FF \in \DM^p(\Omega)$ for $1 \le p \le \infty$.
Then, for any $\phi \in C^{0}(\Omega) \cap L^{\infty}(\Omega)$ with $\nabla \phi \in L^{p'}(\Omega; \R^{n})$,
there exists a set $\mathcal{N} \subset \R$ with $\mathcal{L}^1(\mathcal{N})=0$ such that,
for every nonnegative sequence $\{\ve_k\}$ satisfying $\ve_k \notin \mathcal{N}$ for any $k$ and $\ve_k \to 0$,
the following representation for the interior normal trace on $\partial U$ holds{\rm :}
\begin{equation}
\label{main1}
\ban{\FF \cdot \nu, \phi}_{\partial U} = \int_{U} \phi \, \dd \div \FF + \int_{U} \FF \cdot \nabla \phi \, \dr x
= - \lim_{k \to \infty} \int_{\partial^* U^{\eps_{k}}} \phi \FF \cdot \nu_{U^{\eps_{k}}} \,\, \dr \Haus{n - 1},
\end{equation}
where $\nu_{U^{\eps_{k}}}$ is the inner unit normal to $U^{\ve_k}$ on $\redb U^{\ve_k}$.
In addition, \eqref{main1} holds also for any open set $U \subset \Omega$,
provided that $\mathrm{supp}(\phi) \cap U^{\delta} \Subset \Omega$ for any $\delta > 0$.
\end{theorem}

\begin{proof} \, We divide the proof into three steps.

\smallskip
1. Suppose first that $U \Subset \Omega$.
Then  $U^{\eps} \Subset \Omega$ for any small $\eps > 0$. Recall that $U^{\eps '} \subset U^{\eps}$ if $\eps ' > \eps$ and $\bigcup_{\eps > 0} U^{\eps} = U$.
Define
\begin{equation*}
\psi^{U}_{\eps}(x)
:=
\begin{cases} \eps &\quad \mbox{if} \ x \in U^{\eps}, \\
d(x) &\quad \mbox{if} \ x \in U \setminus U^{\eps}, \\
0 &\quad \mbox{if} \ x \notin U.
 \end{cases}
\end{equation*}
Since $\psi^{U}_{\eps} \in \mathrm{Lip}_{c}(\Omega)$,
we can use it as a test function.
In addition, for any $\phi \in C^{0}(\Omega) \cap L^{\infty}(\Omega)$ with $\nabla \phi \in L^{p'}(\Omega; \R^{n})$,
$\phi \FF \in \DM^{p}(\Omega)$ by Proposition \ref{product rule p}.
Then
\begin{align}
\int_{U} \psi^{U}_{\eps} \, \dd \div(\phi \FF)
&= - \int_{U \setminus U^{\eps}} \phi \FF \cdot \nabla d \, \dr x
= - \int_{U \setminus U^{\eps}} \phi \FF \cdot \nabla d  \, |\nabla d| \, \dr x\nonumber \\
&= - \int_{0}^{\eps} \int_{\partial U^{t}} \phi \FF \cdot  \nabla d  \, \dr \Haus{n - 1} \, \dr t, \label{test psi}
\end{align}
by the coarea formula \eqref{coarea super-sublevel} with $u = d$ and $g = \chi_{U} \phi \FF \cdot \nabla d \, |\nabla d|$,
by Lemma \ref{properties distance} and Remark \ref{boundary negl}.
Thus, we use  test functions $\psi^{U}_{\eps}$ and $\psi^{U}_{\eps + h}$ with $h > 0$ to obtain
\begin{equation*}
\int_{U^{\ve}} \ve \, \dr \div(\phi \FF) + \int_{U \setminus U^{\ve}} d(x)  \, \dd \div(\phi \FF)
= - \int_{0}^{\eps} \int_{\partial U^{t}} \phi \FF \cdot\nabla d  \, \dr \Haus{n - 1} \, \dr t,
\end{equation*}
and
\begin{align*}
& \int_{U^{\ve+h}} (\ve+h) \, \dd \div( \phi \FF) +\int_{U \setminus U^{\ve +h}} d(x) \, \dd \div( \phi \FF)\\
&= - \int_{0}^{\eps + h} \int_{\partial U^{t}} \phi \FF \cdot\nabla d  \, \dr \Haus{n - 1} \, \dr t.
\end{align*}
Subtracting the first equation from the second one, we have
\begin{align*}
&h \int_{U^{\eps + h}} \dd \div(\FF \phi) + \int_{U^{\eps } \setminus U^{\eps + h}} (d(x) - \ve ) \, \dd \div(\FF \phi)\\
& =- \int_{\eps}^{\eps + h} \int_{\partial U^{t}} \phi \FF \cdot \nabla d \, \dr\Haus{n - 1} \, \dr t.
\end{align*}

We now divide by $h$ and use the fact that
$$
0 \le d(x) - \eps \le h \,\,\,\,\, \mbox{in $U^{\eps} \setminus U^{\eps + h}$},
\qquad\,\,
|\div(\FF \phi)|(U^{\eps} \setminus U^{\eps + h}) \to 0 \,\,\,\,\,\mbox{as $h \to 0$}
$$
to conclude
\begin{equation}
\label{aqui}
\int_{U^{\eps}} \dd \div(\FF \phi)
= - \int_{\partial U^{\eps}} \phi \FF \cdot \nabla d  \,\, \dr \Haus{n - 1} \qquad \mbox{for $\Leb{1}$--{\it a.e.} $\eps > 0$}.
\end{equation}
We can take any sequence $\eps_{k} \to 0$ of such good values to obtain
\begin{equation}
\label{almost}
\int_{U} \phi \, \dd \div \FF + \int_{U} \FF \cdot \nabla \phi \, \dr x
= - \lim_{k \to \infty} \int_{\partial U^{\eps_{k}}} \phi \FF \cdot \nabla d \,\, \dr \Haus{n - 1}.
 \end{equation}
By Lemma \ref{boundary tubular ngb}, such a sequence can be chosen so that $\Haus{n - 1}(\partial U^{\eps_{k}} \setminus \partial^{*} U^{\eps_{k}}) = 0$
and $\nabla d$ is the inner normal to $U^{\eps_{k}}$ at $\Haus{n - 1}$--almost every point of $\partial^{*} U^{\eps_{k}}$. Then the result follows.

\smallskip
2. Let now $U \subset \Omega$ be bounded.
Since $U^{\delta} \Subset \Omega$ for any $\delta > 0$, we can consider the test functions:
\begin{equation*}
\psi^{U^{\delta}}_{\eps}(x)  := \begin{cases} \eps & \mbox{if} \ x \in U^{\eps + \delta}, \\
d(x) - \delta & \mbox{if} \ x \in U^{\delta} \setminus U^{\eps + \delta}, \\
0 & \mbox{if} \ x \notin U^{\delta}.
 \end{cases}
\end{equation*}
Clearly, $\psi^{U^{\delta}}_{\eps} \in \Lip_{c}(\Omega)$ for any $\delta, \eps > 0$.
Arguing as before, identity \eqref{test psi} becomes
\begin{equation*}
\int_{U^{\delta}} \psi^{U^{\delta}}_{\eps} \, \dd \div( \phi \FF)
= - \int_{\delta}^{\eps + \delta} \int_{\partial U^{t}} \phi \FF \cdot \nabla d \, \dr \Haus{n - 1} \, \dr t.
\end{equation*}
We use the test functions $\psi^{U^{\delta}}_{\eps}$ and $\psi^{U^{\delta}}_{\eps + h}$ for any $h > 0$,
and then subtract the equation
involving $\psi^{U^{\delta}}_{\eps + h}$
from the one involving $\psi^{U^{\delta}}_{\eps}$
to obtain
\begin{align*}
& h \int_{U^{\eps + h + \delta}} \dd \div( \phi \FF)
+ \int_{U^{\eps + \delta} \setminus U^{\eps + h + \delta}} \big(d(x) - \delta - \eps\big) \, \dd \div( \phi \FF)\\
& = - \int_{\eps + \delta}^{\eps + h + \delta} \int_{\partial U^{t}} \phi \FF \cdot \nabla d \, \dr \Haus{n - 1} \, \dr t.
\end{align*}
We can divide by $h$ and send $h \to 0$  to obtain
\begin{equation*}
\int_{U^{\eps + \delta}} \dd \div( \phi \FF) = - \int_{\partial U^{\eps + \delta}} \phi \FF \cdot \nabla d \, \dr \Haus{n - 1}
\qquad\mbox{for $\Leb{1}$--{\it a.e.} $\eps, \delta > 0$}.
\end{equation*}
Now set $\eps' := \eps + \delta$. We choose a suitable sequence $\eps'_{k} \to 0$ for which Lemma \ref{boundary tubular ngb}
applies so that \eqref{main1} holds  by \eqref{product rule}.

\smallskip
3. Consider the case that $U \subset \Omega$ is not bounded.
Then we take $\phi$ with bounded support in $\Omega$. Thus, we can choose test functions $\eta \psi^{U^{\delta}}_{\eps}$
for some $\eps, \delta > 0$ and $\eta \in C^{\infty}_{c}(\Omega)$ satisfying $\eta \equiv 1$ on an open set $V$
such that $\mathrm{supp}(\phi) \cap U^{\delta} \subset V \Subset \Omega$.
Indeed, $\eta \psi^{U^{\delta}}_{\eps} \in \Lip_{c}(\Omega)$ and $\phi \eta \psi^{U^\delta}_{\eps} = \phi \psi^{U^\delta}_{\eps}$.
By the product rule \eqref{product rule}, we have
\begin{equation*}
\int_{\Omega} \phi \psi^{U^{\delta}}_{\eps} \, \dd \div \FF
= \int_{\Omega} \eta \psi^{U^{\delta}}_{\eps} \, \dd \div( \phi \FF)
- \int_{\Omega} \psi^{U^{\delta}}_{\eps} \FF \cdot \nabla \phi \, \dr x.
\end{equation*}
Again, by the product rule, we obtain that $\mathrm{supp}(\div(\phi \FF)) \subset \mathrm{\supp}(\phi)$, which implies
\begin{equation*}
\int_{\Omega} \eta \psi^{U^{\delta}}_{\eps} \, \dd \div( \phi \FF)
= \int_{\Omega} \psi^{U^{\delta}}_{\eps} \, \dd \div( \phi \FF),
\end{equation*}
since $\eta \equiv 1$ on $\mathrm{supp}(\phi) \cap U^{\delta} \supset \mathrm{supp}(\div(\phi \FF)) \cap \mathrm{supp}(\psi^{U^{\delta}}_{\eps})$.
Therefore, from this point, we can repeat the same steps as before to conclude the proof.
\end{proof}

\begin{remark} \,Theorem {\rm \ref{interior normal trace}} implies that
we may take $U = \Omega$ in \eqref{main1} to obtain the Gauss-Green formula up to the boundary of the open set where $\FF$ is defined.
\end{remark}

As an immediate consequence of Theorem \ref{interior normal trace},
we obtain approximations of the classical Green's identities for scalar functions with gradients in $\DM^{p}(\Omega)$.

\begin{theorem}[First Green's identity] \label{first Green's identity}
Let $u \in W^{1, p}(\Omega)$ for $1 \le p \le \infty$ be such that $\Delta u \in \mathcal{M}(\Omega)$,
and let $U \subset \Omega$ be a bounded open set.
Then, for any $\phi \in C^{0}(\Omega) \cap L^{\infty}(\Omega)$ with $\nabla \phi \in L^{p'}(\Omega; \R^{n})$,
there exists a set $\mathcal{N} \subset \R$ with $\mathcal{L}^1(\mathcal{N})=0$ such that,
for every nonnegative sequence $\{\ve_k\}$ satisfying $\ve_k \notin \mathcal{N}$ for any $k$ and $\ve_k \to 0$,
\begin{equation}
\label{first_Green_id}
\int_{U} \phi \, \dr \Delta u + \int_{U} \nabla u \cdot \nabla \phi \, \dr x
= - \lim_{k \to \infty} \int_{\partial^* U^{\eps_{k}}} \phi \nabla u \cdot \nu_{U^{\eps_{k}}} \,\, \dr \Haus{n - 1},
\end{equation}
where $\nu_{U^{\eps_{k}}}$ is the inner unit normal to $U^{\ve_k}$ on $\redb U^{\ve_k}$.

In particular, if $u \in W^{1, 2}(\Omega) \cap C^{0}(\Omega) \cap L^{\infty}(\Omega)$ with $\Delta u \in \mathcal{M}(\Omega)$,
\begin{equation}
\label{first_Green_id_bis}
\int_{U} u \, \dr \Delta u + \int_{U} |\nabla u|^{2} \, \dr x
= - \lim_{k \to \infty} \int_{\partial^* U^{\eps_{k}}} u \nabla u \cdot \nu_{U^{\eps_{k}}} \,\, \dr \Haus{n - 1}.
\end{equation}

In addition, \eqref{first_Green_id} holds also for any open set $U \subset \Omega$,
provided that $\mathrm{supp}(\phi) \cap U^{\delta} \Subset \Omega$ for any small $\delta > 0$.
Analogously, \eqref{first_Green_id_bis} holds for any open set $U \subset \Omega$,
provided that $\mathrm{supp}(u) \cap U^{\delta} \Subset \Omega$ for any $\delta >0$.
\end{theorem}

\begin{proof} \, In order to obtain \eqref{first_Green_id}, it suffices to apply Theorem \ref{interior normal trace} to the vector field $\FF = \nabla u$,
which clearly belongs to $\DM^{p}(\Omega)$. Then, if $u \in W^{1, 2}(\Omega) \cap C^{0}(\Omega) \cap L^{\infty}(\Omega)$,
we can take $\phi = u$ to  obtain \eqref{first_Green_id_bis}.
\end{proof}

\begin{corollary}[Second Green's identity] \label{second Green's identity}
Let $u \in W^{1, p}(\Omega) \cap C^{0}(\Omega) \cap L^{\infty}(\Omega)$ and $v \in W^{1, p'}(\Omega) \cap C^{0}(\Omega) \cap L^{\infty}(\Omega)$ for $1 \le p \le \infty$
be such that $\Delta u, \Delta v \in \mathcal{M}(\Omega)$,
and let $U \subset \Omega$ be a bounded open set.
Then there exists a set $\mathcal{N} \subset \R$ with $\mathcal{L}^1(\mathcal{N})=0$ such that,
for every nonnegative sequence $\{\ve_k\}$ satisfying $\ve_k \notin \mathcal{N}$ for any $k$ and $\ve_k \to 0$,
\begin{equation}
\label{second_Green_id}
\int_{U} v \, \dr \Delta u - u \, \dr \Delta v
= - \lim_{k \to \infty} \int_{\partial^* U^{\eps_{k}}} (v \nabla u - u \nabla v) \cdot \nu_{U^{\eps_{k}}} \,\, \dr \Haus{n - 1},
\end{equation}
where $\nu_{U^{\eps_{k}}}$ is the inner unit normal to $U^{\ve_k}$ on $\redb U^{\ve_k}$.
In addition, \eqref{second_Green_id} holds also for any open set $U \subset \Omega$,
provided that $\mathrm{supp}(u), \mathrm{supp}(v) \cap U^{\delta} \Subset \Omega$ for any small $\delta > 0$.
\end{corollary}

\begin{proof} \, We just need to apply Theorem \ref{first Green's identity} to the vector field $\nabla u$, by using $v$ as scalar function, and vice versa.
Then we can obtain \eqref{first_Green_id} for the vector fields $\nabla u$ and $\nabla v$ with the same sequence $U^{\eps_{k}}$,
since it is enough to select one sequence suitable for $\nabla u$ and then extract a subsequence for $\nabla v$.
Thus, we have
\begin{align*}
\int_{U} v \, \dr \Delta u + \int_{U} \nabla u \cdot \nabla v \, \dr x
& = - \lim_{k \to \infty} \int_{\partial^* U^{\eps_{k}}} v \nabla u \cdot \nu_{U^{\eps_{k}}} \,\, \dr \Haus{n - 1}, \\
\int_{U} u \, \dr \Delta v + \int_{U} \nabla u \cdot \nabla v \, \dr x
& = - \lim_{k \to \infty} \int_{\partial^* U^{\eps_{k}}} u \nabla v \cdot \nu_{U^{\eps_{k}}} \,\, \dr \Haus{n - 1},
\end{align*}
and subtracting the second equation from the first yields \eqref{second_Green_id}.
\end{proof}

\begin{theorem}[Exterior normal trace] \label{exterior normal trace}
Let $U \Subset \Omega$ be an open set, and let $\FF \in \DM^p(\Omega)$ for $1 \le p \le \infty$.
Then, for any $\phi \in C^{0}(\Omega)$ with $\nabla \phi \in L^{p'}(\Omega; \R^{n})$,
there exists a set $\mathcal{N} \subset \R$ with $\mathcal{L}^1(\mathcal{N})=0$ such that,
for every nonnegative sequence $\{\ve_k\}$ satisfying $\ve_k \notin \mathcal{N}$ for any $k$ and $\ve_k \to 0$,
the following representation for the exterior normal trace on $\partial U$ holds{\rm :}
\begin{equation}
\label{main2}
\ban{\FF \cdot \nu, \phi}_{\partial \overline{U}}
= \int_{\overline{U}} \phi \, \dd \div \FF + \int_{\overline{U}} \FF \cdot \nabla \phi \, \dr x
= - \lim_{k \to \infty} \int_{\partial^* U_{\eps_{k}}} \phi \FF \cdot \nu_{U_{\eps_{k}}} \,\, \dr \Haus{n - 1},
\end{equation}
where $\nu_{U_{\eps_{k}}}$ is the inner unit normal to $U_{\ve_k}$ on $\partial^*U_{\ve_k}$.
In addition, \eqref{main2} holds also for any open set $U$ satisfying $\overline{U} \subset \Omega$,
provided that $\mathrm{supp}(\phi)$ is compact in $\Omega$.
\end{theorem}

\begin{proof}
\,\, We start with the case $U \Subset \Omega$.  Then  $U_{\eps} \Subset \Omega$ for any $\eps > 0$ small enough.
We consider the Lipschitz functions:
\begin{equation*}
\xi^{U}_{\eps}(x) :=
\begin{cases}
\ve & \mbox{if} \ x \in U, \\
d(x)+ \ve & \mbox{if} \ x \in U_{\eps} \setminus U, \\
0 & \mbox{if} \ x \notin U_{\eps}.
\end{cases}
\end{equation*}

By Proposition \ref{product rule p},
$\phi \FF \in DM^{p}(\Omega')$ for any open set $\Omega'$ satisfying $U_{\eps} \Subset \Omega' \Subset \Omega$
for any $\eps > 0$ small enough.
Thus, we can use $\xi^{U}_{\eps}$ as test functions to obtain
\begin{align} \label{test xi}
\int_{\Omega'} \xi^{U}_{\eps} \, \dd \div (\phi \FF)
&= - \int_{\Omega'} \phi \FF \cdot \nabla \xi^{U}_{\eps} \, \dr x
= - \int_{U_{\eps} \setminus U} \phi \FF \cdot \nabla d
\, \dr x \nonumber\\
&= - \int_{0}^{\eps} \int_{\partial U_{t}} \phi \FF \cdot \nabla d \, \dr \Haus{n - 1} \dr t,
\end{align}
by the coarea formula \eqref{coarea super-sublevel} with $u = d$ and $g = \chi_{\Omega' \setminus U} \phi \FF \cdot \nabla d |\nabla d|$,
by Lemma \ref{properties distance} and Remark \ref{boundary negl}.

Now we proceed as in the proof of Theorem \ref{interior normal trace}:
Take $\xi^{U}_{\eps}$ and $\xi^{U}_{\eps + h}$ for some $h > 0$ small enough as test functions so that
\begin{equation*}
\int_{U_{\eps} \setminus \overline{U}} \big(d(x)+\eps \big) \, \dd \div( \phi \FF) + \int_{\overline{U}} \eps \, \dd\div( \phi \FF)
=- \int_{0}^{\eps} \int_{\partial U_{t}} \phi \FF \cdot \nabla d \, \dr \Haus{n - 1} \, dt,
\end{equation*}
and
\begin{align*}
& \int_{U_{\eps + h} \setminus \overline{U}} \big(d(x)+\eps + h \big) \, \dd \div( \phi \FF) + \int_{\overline{U}} (\eps + h) \, \dd\div( \phi \FF)\\
& =- \int_{0}^{\eps + h} \int_{\partial U_{t}} \phi \FF \cdot \nabla d \, \dr \Haus{n - 1} \, \dr t.
\end{align*}
By subtracting the first equation from the second one, we have
\begin{align*}
&\int_{U_{\eps + h} \setminus \overline{U_{\eps}}} \big(d(x)+\eps + h \big) \, \dd \div( \phi \FF)
+ \int_{\overline{U_{\eps}}} h \, \dd \div( \phi \FF)\\
& = -\int_{\eps}^{\eps + h} \int_{\partial U_{t}} \phi \FF \cdot \nabla d \, \dr \Haus{n - 1} \, \dr t.
\end{align*}
It is clear that $ 0 \le d(x)+\ve \le h$ on $U_{\eps + h} \setminus \overline{U_{\eps}}$
and that $\bigcap_{h > 0} U_{\eps + h} = \overline{U_{\eps}}$ for any $\eps > 0$
implies
$$
|\div (\phi \FF)|(U_{\eps + h} \setminus \overline{U_{\eps}}) \to 0 \qquad \mbox{as $h \to 0$}.
$$
Then we can divide by $h$ and let $h \to 0$, by applying the Lebesgue theorem, to obtain that, for $\Leb{1}$--{\it a.e.} $\eps > 0$,
\begin{equation}
\label{Gauss-Green-compact}
\int_{\overline{U_{\eps}}} \phi \, \dd \div \FF + \int_{\overline{U_{\eps}}} \FF \cdot \nabla \phi \, \dr x
= -\int_{\partial U_{\eps}} \phi \FF \cdot \nabla d \, \dr \Haus{n - 1},
\end{equation}
by the product rule \eqref{product rule}.
We now choose a sequence $\eps_{k} \to 0$ such that \eqref{Gauss-Green-compact} holds
and pass to the limit to obtain
\begin{equation} \label{Gauss-Green p compact eps}
\int_{\overline{U}} \phi \, \dd \div \FF + \int_{\overline{U}} \FF \cdot \nabla \phi \, \dr x
= - \lim_{k \to \infty} \int_{\partial U_{\eps_{k}}} \phi \FF \cdot \nabla d \, \dr \Haus{n - 1}.
\end{equation}

As in the proof of Theorem \ref{interior normal trace},
we can choose the sequence in such a way that the assertion
in Lemma \ref{boundary tubular ngb} also holds. Thus, we obtain the result.

\smallskip
In the general case, $\overline{U} \subset \Omega$, and $\mathrm{supp}(\phi)$ is a compact subset of $\Omega$.
Hence, we can argue as in the last part of the proof of Theorem \ref{interior normal trace},
by taking a smooth cutoff function $\eta \in C^{\infty}_{c}(\Omega)$ such that $\eta \equiv 1$ on an open neighborhood $V$
of $\mathrm{supp}(\phi)$. Following the same steps and replacing $\xi^{U}_{\eps}$ as test functions,
we obtain the desired result.
\end{proof}

\begin{remark}
\, The previous results apply in particular to the case when $U$ is an open set of finite perimeter.
\end{remark}

\begin{remark} \label{G-G_eps_open_closed}
\, As a byproduct of the proofs of Theorems {\rm \ref{interior normal trace}} and {\rm \ref{exterior normal trace}},
we obtain the Gauss-Green formulas for almost every set that is approximating a given open set $U$ from the interior and the exterior.
More precisely, if $U \Subset \Omega$ is an open set, $\FF \in \DM^p(\Omega)$ for $1 \le p \le \infty$,
and $\phi \in C^{0}(\Omega)$ with $\nabla \phi \in L^{p'}(\Omega; \R^{n})$,
then, for $\Leb{1}$--{\it a.e.} $\eps > 0$,
\begin{align}
\int_{U^{\eps}} \phi \, \dd \div \FF + \int_{U^{\eps}} \FF \cdot \nabla \phi \, \dr x
& = -\int_{\redb U^{\eps}} \phi \FF \cdot \nu_{U^{\eps}} \, \dr \Haus{n - 1},\label{G-G_eps_open} \\
\int_{\overline{U_{\eps}}} \phi \, d \div \FF + \int_{U_{\eps}} \FF \cdot \nabla \phi \, \dr x
& = -\int_{\redb U_{\eps}} \phi \FF \cdot \nu_{U_{\eps}} \, \dr \Haus{n - 1}. \label{G-G_eps_closed}
\end{align}
This follows from \eqref{aqui} and \eqref{Gauss-Green-compact}
and by taking $\eps > 0$ {\rm (}up to another negligible set{\rm )} such that Lemma {\rm \ref{boundary tubular ngb}} holds,
so that $\Haus{n - 1}(\partial U^{\eps} \setminus \redb U^{\eps}) = 0$,
$\Haus{n - 1}(\partial U_{\eps} \setminus \redb U_{\eps}) = 0$ {\rm (}which implies that $|\partial U^{\eps}| = 0${\rm )},
$\nabla d = \nu_{U^{\eps}}$ $\Haus{n - 1}$--{\it a.e.} on $\redb U^{\eps}$,
and $\nabla d = \nu_{U_{\eps}}$ $\Haus{n - 1}$--{\it a.e.} on $\redb U_{\eps}$.

In addition, \eqref{G-G_eps_open}--\eqref{G-G_eps_closed} also hold
for any open set $U$ satisfying $\overline{U} \subset \Omega$, provided that $\mathrm{supp}(\phi)$ is compact in $\Omega$.
In general, this statement is valid for $\Leb{1}$--{\it a.e.} $\eps > 0$ because we need to apply
Lemma {\rm \ref{boundary tubular ngb}} and to derive the integrals in $h > 0${\rm :}
\begin{equation*}
\int_{\eps}^{\eps + h} \int_{\partial U^{t}} \phi \FF \cdot \nabla d \, \dr \Haus{n - 1} \, \dr t,  \qquad
\int_{\eps}^{\eps + h} \int_{\partial U_{t}} \phi \FF \cdot \nabla d \, \dr \Haus{n - 1} \, \dr t.
\end{equation*}
Therefore, such a condition may be removed as long as the conclusions of Lemma {\rm \ref{boundary tubular ngb}} hold
for any $\eps > 0$, and $\int_{\partial U^{t}} \phi \FF \cdot \nabla d \, \dr \Haus{n - 1}$ and $\int_{\partial U_{t}} \phi \FF \cdot \nabla d \, \dr \Haus{n - 1}$
are continuous functions of $t > 0$.
\end{remark}

\begin{remark}
\, It is not necessary to use the signed distance function to construct a family of approximating sets suitable
for Theorems {\rm \ref{interior normal trace}} and  {\rm \ref{exterior normal trace}}.
Such an argument is
related to the one in
\cite[Theorem 2.4]{Silhavy2}.

If, for a given open set $U \Subset \Omega$, there exists a function $m \in \Lip(\Omega)$ satisfying $m > 0$ in $U$,
$m = 0$ on $\partial U$, and $\mathrm{ess inf} (|\nabla m|) > 0$ in $U$,
then  sets $\{ m > \eps \}, \eps \in \R$, can be used for the approximation.
In fact, sets $\{m > \eps \}$ are of finite perimeter for $\Leb{1}$--{\it a.e.} $\eps>0$ and,
for such good values of $\eps$, the measure-theoretic unit interior normals satisfy
\begin{equation*}
\nu_{\{m > \eps \}} = \frac{\nabla m}{|\nabla m|} \qquad \,\, \Haus{n - 1}\text{--{\it a.e.} on} \, \, \redb \{ m > \eps \}.
\end{equation*}
In addition, if there exists such a function $m$ with $C^{k}, k \ge 2$, or $C^{\infty}$ regularity,
then $\{m > \eps \}$ has a $C^{k}$ or smooth boundary.

As we will see in \S 7, if $U$ is an open bounded set with $C^{0}$ boundary,
then there exists a smooth regularized distance $\rho$ satisfying the previously mentioned properties.
For a general open set $U$, this may be false.
\end{remark}

We can extend Theorem \ref{exterior normal trace} to any compact set $K \subset \Omega$, in the spirit of \cite[Theorem 5.20]{Sch}.
Indeed, we just need to choose the following Lipschitz functions as test functions:
\begin{equation*}
\varphi_{K}^{\eps}(x) :=
\begin{cases} \eps &\,\, \mbox{if} \ \dist(x, K) = 0, \\
\eps - \dist(x, K) &\,\, \mbox{if} \ 0 < \dist(x, K) < \eps, \\
0 &\,\, \mbox{if} \ \dist(x, K) \ge \eps,
\end{cases}
\end{equation*}
and then argue as in the proof of Theorem \ref{exterior normal trace} to achieve the following result.

\begin{corollary} \label{K_compact_GG_eps}
Let $K \subset \Omega$ be a compact set, and let $\FF \in \DM^p(\Omega)$.
Then, for any
$\phi \in C^{0}(\Omega)$ with $\nabla \phi \in L^{p'}(\Omega; \R^{n})$,
there exists a set $\mathcal{N} \subset \R$ with $\mathcal{L}^1(\mathcal{N})=0$ such that,
for every nonnegative sequence $\{\ve_k\}$ satisfying $\ve_k \notin \mathcal{N}$ for any $k$ and $\ve_k \to 0$,
\begin{equation} \label{Gauss-Green p compact}
\int_{K} \phi \, \dd \div \FF + \int_{K} \FF \cdot \nabla \phi \, \dr x
= \lim_{k \to \infty} \int_{\partial K_{\eps_{k}}} \phi \FF \cdot \nabla \dist(x, K) \, \dr \Haus{n - 1}, \end{equation}
where $K_{\eps}{\rm :}= \{ x \in \Omega : \dist(x, K) \le \eps \}$.
In addition, \eqref{Gauss-Green p compact} holds also for any closed set $C \subset \Omega$,
provided that $\mathrm{supp}(\phi)$ is compact in $\Omega$.
\end{corollary}

The right-hand side of \eqref{Gauss-Green p compact} can be seen as the definition of the generalized normal trace
functional related to $\FF$ on $\partial K$, where $\nabla \dist(x, K)$ plays the role of a generalized unit exterior normal,
even in the case $\mathring{K} = \emptyset$.

\begin{remark}\label{Traceneighborhood}
\, The results of Schuricht \cite[Theorem 5.20]{Sch} and \v{S}ilhav\'y \cite[Theorem 2.4]{Silhavy2} can be recovered by \eqref{test xi} and \eqref{test psi}, respectively.

Indeed, under the same assumptions of Theorem {\rm \ref{interior normal trace}},
we divide by $\eps$ in \eqref{test psi}, use the product rule \eqref{product rule}, and send $\eps \to 0$ to obtain
\begin{equation} \label{Gauss-Green open average}
\int_{U} \phi \, \dd \div \FF + \int_{U} \FF \cdot \nabla \phi \, \dr x
= - \lim_{\eps \to 0} \frac{1}{\eps} \int_{U \setminus U^{\eps}} \phi \FF \cdot \nabla d \, \dr x,
\end{equation}
since $0 \le \frac{\psi^{U}_{\eps}}{\eps} \le 1$ on $U \setminus U^{\eps}$
and $|\div(\FF \phi)|(U \setminus U^{\eps}) \to 0$ as $\eps \to 0$.
On the other hand, applying the same steps to \eqref{test xi} yields
\begin{equation} \label{Gauss-Green compact average}
\int_{\overline{U}} \phi \, \dd \div \FF + \int_{\overline{U}} \FF \cdot \nabla \phi \, \dr x
= -\lim_{\eps \to 0} \frac{1}{\eps} \int_{U_{\eps} \setminus U} \phi \FF \cdot \nabla d \, \dr x,
\end{equation}
if $\FF, \phi$, and $U$ satisfy the conditions of Theorem {\rm \ref{exterior normal trace}}.
In particular, this works for any compact set $K \subset \Omega$, as in Corollary {\rm \ref{K_compact_GG_eps}}{\rm :}
\begin{equation} \label{Gauss-Green real compact average}
\int_{K} \phi \, \dd \div \FF + \int_{K} \FF \cdot \nabla \phi \, \dr x
= \lim_{\eps \to 0} \frac{1}{\eps} \int_{K_{\eps} \setminus K} \phi \FF \cdot \nabla \dist(x, K) \, \dr x.
\end{equation}
\end{remark}

\begin{remark}
\, Formulas \eqref{main1} and \eqref{main2} can be used to obtain the Gauss-Green formula on the boundary of $U \Subset \Omega${\rm :}
\begin{align} \label{Gauss-Green p boundary}
\div(\FF \phi)(\partial U) & = \div(\FF \phi)(\overline{U}) - \div(\FF \phi)(U)\nonumber \\
& = \lim_{k \to \infty} \Big(\int_{\redb U^{\eps_{k}}} \phi \FF \cdot \nu_{U^{\eps_{k}}} \, \dr \Haus{n - 1}
     - \int_{\redb U_{\eps_{k}}} \phi \FF \cdot \nu_{U_{\eps_{k}}} \, \dr\Haus{n - 1} \Big),
\end{align}
since we can extract the same subsequence $\eps_{k}$ for $\overline{U}$ and $U$.
The same result holds for $U$ such that $\overline{U} \subset \Omega$ if $\phi$ has compact support in $\Omega$.
\end{remark}

\begin{remark}
\, If $U = B(x_{0}, r)$, we obtain the Gauss-Green formula for $\Leb{1}$--{\it a.e.} $r > 0$.

Indeed,
$\dist(x, \partial B(x_{0}, r)) = r - |x - x_{0}|$ for any $x \in B(x_{0}, r)$ so that
\eqref{G-G_eps_open}
implies that, for $\Leb{1}$--{\it a.e.} $\eps \in (0, r)$,
\begin{align*}
\int_{B(x_{0}, r - \eps)} \dd \div(\phi \FF)
& = \int_{\partial B(x_{0}, r - \eps)} \phi(x) \FF(x) \cdot \frac{(x - x_{0})}{|x - x_{0}|} \, \dr \Haus{n - 1}(x) \\
& = - \int_{\partial B(x_{0}, r - \eps)} \phi \FF \cdot \nu_{B(x_{0}, r - \eps)} \, \dr \Haus{n - 1}.
\end{align*}
Since the initial choice of $r$ is arbitrary, we conclude
\begin{equation} \label{G-G open balls}
\int_{B(x_{0}, r)} \dd \div(\phi \FF) = - \int_{\partial B(x_{0}, r)} \phi \FF \cdot \nu_{B(x_{0}, r)} \, \dr \Haus{n - 1}
\qquad \mbox{for $\Leb{1}$--{\it a.e.} $r > 0$}.
\end{equation}
Moreover, the same argument works with closed balls so that, by \eqref{G-G_eps_closed},
\begin{equation} \label{G-G closed balls}
\int_{\overline{B(x_{0}, r)}} \dd \div(\phi \FF)
= - \int_{\partial B(x_{0}, r)} \phi \FF \cdot \nu_{B(x_{0}, r)} \, \dr \Haus{n - 1}\qquad \mbox{for $\Leb{1}$--{\it a.e.} $r > 0$},
\end{equation}
since $\dist(x, \partial B(x_{0}, r)) = |x - x_{0}| - r$ for any $x \notin B(x_{0}, r)$.

This can also be seen as a consequence of the fact:
$$
|\div(\phi \FF)|(\partial B(x_{0}, r)) = 0 \qquad \mbox{for $\Leb{1}$--{\it a.e.} $r > 0$},
$$
since $\div(\phi \FF)$ is a Radon measure.
\end{remark}

We now present a concrete example of applications of \eqref{main1} and \eqref{main2} to a $\DM^p$--field
whose norm blows up on the boundary of the integration domains.

\begin{example} \label{Whitney_field_example}
Let $\FF : \R^{2} \setminus \{ (0, 0) \} \to \R^{2}$ be the vector field{\rm :}
\begin{equation} \label{Whitney_field}
\FF(x_1, x_2) := \frac{(x_1,x_2)}{x_1^{2} + x_2^{2}}.
\end{equation}
This is the particular case for $n = 2$ of the vector field $\FF : \R^{n} \setminus \{\mathbf{0}\} \to \R^{n}$ given by
\begin{equation*}
\FF(x) := \frac{x}{|x|^{n}}   \qquad \mbox{for $x\ne \mathbf{0}$}.
\end{equation*}
Then $\FF \in \DM^{p}_{\rm loc}(\R^{n})$ for $1 \le p < \frac{n}{n - 1}$ and
\begin{equation} \label{divergence_distr}
\div \FF = n \omega_{n} \delta_{\mathbf{0}},
\end{equation}
where $\omega_{n} = |B(\mathbf{0}, 1)|$.
In particular, if $n = 2$,  $\FF \in \DM^{p}_{\rm loc}(\R^{2})$ for $1 \le p < 2$,
and $\div \FF = 2 \pi \delta_{(0,0)}$.

Consider $U = (0, 1)^{2}$.
Chen-Frid \cite[Example 1.1]{CF2} observed that
\begin{equation*}
0 = \div \FF (U) \neq - \int_{\partial U} \FF \cdot \nu_{U} \, \dr \Haus{1} = \frac{\pi}{2},
\end{equation*}
since $\FF \cdot \nu_{U} = 0$ on $\big(\{0\} \times (0, 1)\big) \cup \big((0, 1) \times \{0\}\big)$
and $\displaystyle  \int_{0}^{1} \frac{1}{1 + x_1^{2}} \, \dr x_1 = \frac{\pi}{4}$.

The approach employed in the proof of Theorems {\rm \ref{interior normal trace}} and {\rm \ref{exterior normal trace}}
enable us to solve this apparent contradiction, by showing that
\begin{align*}
&0 = \div \FF (U) = - \lim_{\eps \to 0} \int_{\partial U^{\eps}} \FF \cdot \nu_{U^{\eps}} \, \dr \Haus{1}, \\
&2 \pi = \div \FF (\overline{U})  = - \lim_{\eps \to 0} \int_{\partial U_{\eps}} \FF \cdot \nu_{U_{\eps}} \, \dr \Haus{1},
\end{align*}
where $U^{\eps}$ and $U_{\eps}$ are given by \eqref{U^eps} and \eqref{U_eps}, respectively.

In this case, we do not have to select a suitable sequence $\eps_{k} \to 0$.
Indeed, $\FF$ is smooth away from the origin, and  $U^{\eps}$ and $U_{\eps}$ are sets of finite perimeter for any $\eps > 0$.
Moreover, for this choice of $U$, Lemma {\rm \ref{boundary tubular ngb}} is valid for any $\eps \in (0, 1)$.
Also,
the continuity condition mentioned in Remark {\rm \ref{G-G_eps_open_closed}} can be checked.
Therefore, by \eqref{G-G_eps_open}--\eqref{G-G_eps_closed}, we obtain that, for any $\eps > 0$,
\begin{align}
0 = \div \FF (U^{\eps}) = - \int_{\partial U^{\eps}} \FF \cdot \nu_{U^{\eps}} \, \dr \Haus{1}, \label{G-G_Whit_eps_int}\\
2 \pi = \div \FF (\overline{U_{\eps}}) = - \int_{\partial U_{\eps}} \FF \cdot \nu_{U_{\eps}} \, \dr \Haus{1}.
\label{G-G_Whit_eps_ext}
\end{align}
Passing to the limit verifies our assertion.

We may also verify this statement by hand.
Observe that $U^{\eps} = (\eps, 1 - \eps)^{2}$ for any $\eps \in (0, 1)$.
Therefore, we have
\begin{align*}
\int_{\partial U^{\eps}} \FF \cdot \nu_{U^{\eps}} \, \dr \Haus{1}
 =&\, \int_{\eps}^{1 - \eps} \frac{\eps}{\eps^{2} + x_1^{2}} \, \dr x_1
- \int_{\eps}^{1 - \eps} \frac{1 - \eps}{(1 - \eps)^{2} + x_1^{2}} \, \dr x_1 \\
& + \int_{\eps}^{1 - \eps} \frac{\eps}{\eps^{2} + x_2^{2}} \, \dr x_2
- \int_{\eps}^{1 - \eps} \frac{1 - \eps}{(1 - \eps)^{2} + x_2^{2}} \, \dr x_2 \\[2mm]
=&\, 2 \Big( \arctan{( \frac{1 - \eps}{\eps})} - \frac{\pi}{4} - \frac{\pi}{4}
  + \arctan{(\frac{\eps}{1 - \eps})} \Big) \\
=&\,  0 = - \div \FF(U^{\eps})
\end{align*}
for any $\eps > 0$, which is \eqref{G-G_Whit_eps_int}.
As for \eqref{G-G_Whit_eps_ext},
$\partial U_{\eps}$ is the union of four segments{\rm :}
$$
(0, 1) \times \{- \eps \}, \quad \{ 1 + \eps \} \times (0, 1),
\quad (0, 1) \times \{ 1 + \eps \}, \quad \{ - \eps \} \times (0, 1),
$$
and of four circumference arcs of angle $\frac{\pi}{2}$ and radius $\eps$
centered at the corners of square $U$.
Therefore, these terms give
\begin{align*}
&\int_{\partial U_{\eps}} \FF \cdot \nu_{U_{\eps}} \, \dr \Haus{1}\\
& = -\int_{0}^{1} \frac{\eps}{\eps^{2} + x_1^{2}} \, \dr x_1
  - \int_{0}^{1} \frac{1 + \eps}{(1 + \eps)^{2} + x_2^{2}} \, \dr x_2
  - \int_{0}^{1} \frac{1 + \eps}{(1 + \eps)^{2} + x_1^{2}} \, \dr x_1\\
&\quad - \int_{0}^{1} \frac{\eps}{\eps^{2} + x_2^{2}} \, \dr x_2
  + \int_{\pi}^{\frac{3\pi}{2}} \big( - \frac{1}{\eps} \big) \eps \, \dr \theta
  - \int_{\frac{\pi}{2}}^{\pi} \frac{\eps(\eps + \sin{\theta})}{1 + \eps^{2} + 2 \eps \sin{\theta}} \, \dr \theta \\
&\quad -  \int_{\frac{3\pi}{2}}^{2\pi} \frac{\eps(\eps + \cos{\theta})}{1 + \eps^{2} + 2 \eps \cos{\theta}} \, \dr \theta
   - \int_{0}^{\frac{\pi}{2}}  \frac{\eps(\eps + \cos{\theta} + \sin{\theta})}{2 + \eps^{2} + 2 \eps (\cos{\theta} + \sin{\theta})} \, \dr \theta \\
& = - 2 \arctan{( \frac{1}{\eps})} - 2 \arctan{( \frac{1}{1 + \eps})} - \frac{\pi}{2} - \frac{\pi}{4}
   + \arctan{(\frac{1 - \eps}{1 + \eps})}  \\
&\quad - \frac{\pi}{2} + \arctan{(\frac{1}{\eps})} - \frac{\pi}{4}
  - \arctan{(\frac{\eps}{\eps + 2})} + \arctan{( \frac{1}{1 + \eps})} \\
& = - \frac{3}{2}\pi - \arctan{(\frac{1}{\eps})}
   - \arctan{(\frac{1}{1 + \eps})}
     + \frac{\pi}{2} - \arctan{\eps} - \arctan{(1 + \eps)} \\
& = - \frac{3}{2}\pi - \frac{\pi}{2} = - 2 \pi = - \div \FF(U_{\eps})
\qquad\,\, \mbox{for any $\eps > 0$}.
\end{align*}
\end{example}

\section{\, Other Classes of Divergence-Measure Fields with Normal Trace Measures}

In this section, as a result of the construction in \S 5, we characterize a class of $\DM^p$--fields
whose normal traces on $\partial U$ are represented by Radon measures.

Remark \ref{Traceneighborhood} allows us to find a new sufficient condition
under which the normal trace functional on an open or closed set can be represented
by a Radon measure on the boundary.
Such a condition requires a particular representation for the vector field $\FF$,
first introduced by \v{S}ilhav\'y \cite[Proposition 6.1]{Silhavy1}.
We also need to recall the notion of lower $(n - 1)$--dimensional Minkowski content.

\begin{definition} Given a closed set $K$ in $\R^{n}$, the $(n - 1)$--dimensional Minkowski content is defined as
\begin{equation*}
\mathscr{M}_{*}^{n - 1}(K) := \liminf_{\eps \to 0} \frac{|K + B(0, \eps)|}{2 \eps}.
\end{equation*}
\end{definition}

\begin{proposition} \label{trace_Silhavy_repr}
Let $\FF \in \DM^{p}(\Omega)$ for $1 \le p \le \infty$,
and let $U \subset \Omega$ be a bounded open set such that $\mathscr{M}_{*}^{n - 1}(\partial U) < \infty$.
Assume that $\div \FF$ has compact support in $U$ and that
\begin{equation} \label{repr_F_Silhavy}
\FF(x) = \frac{1}{n \omega_{n}} \int_{\Omega} \frac{(x - y)}{|x - y|^{n}} \, \dr\div \FF(y)
\qquad\mbox{for $\Leb{n}$--a.e. $x \in \Omega$}.
\end{equation}
Then
$$
\ban{\FF \cdot \nu, \cdot}_{\partial U} \in \mathcal{M}(\partial U).
$$
Similarly, if $K \subset \Omega$ is a compact set such that $\mathscr{M}_{*}^{n - 1}(\partial K) < \infty$,
then
$$
\ban{\FF \cdot \nu, \cdot}_{\partial K} \in \mathcal{M}(\partial K),
$$
which is in particular true for $K = \overline{U}$ if $U \Subset \Omega$.
\end{proposition}

\begin{proof}
\, The normal trace $\ban{\FF \cdot \nu, \cdot}_{\partial U}$ has the following representation:
\begin{equation}
\ban{\FF \cdot \nu, \varphi}_{\partial U}= -\lim_{\ve \to 0} \frac{1}{\eps} \int_{U \setminus U^{\ve}} \varphi(x) \FF (x) \cdot \nabla d (x)\, \dr x
\qquad \mbox{for any $\varphi \in \Lip_{c}(\R^{n})$};
\end{equation}
see \eqref{Gauss-Green open average} and the observations in Remark \ref{Traceneighborhood}.
Thus, in order to prove that $\ban{\FF \cdot \nu, \cdot}_{\partial U} \in \mathcal{M}(\partial U)$,
 it suffices to show
 \begin{equation} \label{measure_cond_content}
|\ban{\FF \cdot \nu, \varphi}_{\partial U}| \leq C \norm{\varphi}_{L^{\infty}(\partial U)}
\end{equation}
for a constant $C$ independent of $\varphi$. Let $V \subset U$ be a compact set such that $\supp (|\div \FF|) \subset V$.
Then it follows that
\begin{align}
\big|\ban{\FF \cdot \nu, \varphi}_{\partial U}\big|
& = \lim_{\ve \to 0} \Big|\frac{1}{\ve}\int_{U \setminus U^{\ve}} \varphi(x) \FF (x) \cdot \nabla d (x)\, \dr x \Big|   \nonumber\\
 &=  \lim_{\ve \to 0} \Big| \frac{1}{\ve}\int_{U \setminus U^{\ve}} \varphi(x)
    \Big(\int_{\Omega} \frac{(x-y)}{|x-y|^n} \, \dd \div \FF(y) \cdot \nabla d (x) \Big)\, \dr x \Big| \nonumber \\
 &\leq  \liminf_{\ve \to 0} \frac{1}{\ve} \norm{\varphi}_{L^{\infty}(U \setminus U^{\eps})}
    \int_{U \setminus U^{\ve}} \int_{V} \frac{1}{|x-y|^{n- 1}} \, \dr |\div \FF|(y)\, \dr x \nonumber \\
  &=  \liminf_{\ve \to 0} \frac{1}{\ve} \norm{\varphi}_{L^{\infty}(U \setminus U^{\eps})} \int_{V} \int_{U \setminus U^{\ve}} \frac{1}{|x-y|^{n- 1}} \, \dr x \,
   \dr |\div \FF|(y),
\end{align}
where we have used Fubini's theorem and the fact that $|\nabla d| = 1$ $\Leb{n}$--{\it a.e.} (Lemma \ref{properties distance}).
Moreover,
$\lim_{\eps \to 0} \sup_{U \setminus U^{\eps}} |\varphi|
   =  \norm{\varphi}_{L^{\infty}(\partial U)}$, by the continuity of $\varphi$.

Since $V \subset U$ is a compact set, there exists $k=k(V)$ such that, for small enough $\ve$,
$$
|x-y| \geq k \qquad \mbox{for any $x \in U \setminus U^{\eps}$ and $y \in V$}.
$$
Then it follows that
\begin{align*}
|\ban{\FF \cdot \nu, \varphi}_{\partial U}|
&\leq \norm{\varphi}_{L^{\infty}(\partial U)} |\div \FF|(V) k^{1 - n} \liminf_{\ve \to 0} \frac{1}{\ve}  |U \setminus U^{\ve}| \\
&\leq 2 k^{1-n} |\div \FF|(V) \norm{\varphi}_{L^{\infty}(\partial U)}  \liminf_{\ve \to 0} \frac{1}{2\ve} |\partial U + B(0,\ve)| \\
&\leq  2 k^{1-n}  |\div \FF|(V) \mathscr{M}_{*}^{n - 1}(\partial U) \norm{\varphi}_{L^{\infty}(\partial U)}.
\end{align*}
This proves \eqref{measure_cond_content}, since $\partial U$ is of finite lower $(n - 1)$--dimensional Minkowski content.

In the same way, using \eqref{Gauss-Green real compact average}, we can show that
$\ban{\FF \cdot \nu, \cdot}_{\partial K} \in \mathcal{M}(\partial K)$ if $K \subset \Omega$ is a compact set,
especially when $K = \overline{U}$ for some bounded open set $U$.
\end{proof}

\begin{remark}
\,\, Condition \eqref{repr_F_Silhavy} is not strongly restrictive in the sense that $\FF$ may not be compactly supported and unbounded.
Indeed, let $\FF(x) = \frac{x}{|x|^{n}}$ as in Example {\rm \ref{Whitney_field_example}}.
Then, by \eqref{divergence_distr},
$\FF$ satisfies \eqref{repr_F_Silhavy}, even though $\FF$ is  unbounded and supported on the whole $\R^{n}$.

Moreover, \eqref{repr_F_Silhavy} is satisfied by a large class of vector fields $\FF$,
as
shown in
\cite[Proposition 6.1]{Silhavy1}.
Indeed, given any $\mu \in \mathcal{M}(\Omega)$ with compact support in $\Omega$, the vector field
\begin{equation*}
\FF(x) = \frac{1}{n \omega_{n}} \int_{\Omega} \frac{(x - y)}{|x - y|^{n}} \, \dr \mu(y)
\end{equation*}
satisfies $\div \FF = \mu$ in $\mathcal{M}(\Omega)$,
and $\FF \in L^{p}_{\rm loc}(\Omega; \R^{n})$ for any $1 \le p < \frac{n}{n - 1}$.
In addition, if $\frac{n}{n - 1} \le p \le \infty$,
then $\FF \in L^{p}_{\rm loc}(\Omega; \R^{n})$
if $|\mu|(B(x, r)) \le c r^{m}$ for any $x \in \R^{n}$ and $r \in (0, a)$,
for some $m > n - \frac{p}{p-1}$, $a > 0$, and $c > 0$.
\end{remark}

\begin{remark}
\,\, Proposition {\rm \ref{trace_Silhavy_repr}} applies to a particular subfamily of sets of finite perimeter.
Indeed, any bounded open set $U$ with $\mathscr{M}_{*}^{n - 1}(\partial U) < \infty$ is a set of finite perimeter in $\R^{n}$.

Even though the result is well known, we give here a short proof for the ease of the reader.
Let
\begin{equation*}
g_{\eps}(x) := \max\{0, 1 - \frac{\mathrm{dist}(x, U^{\eps})}{\eps}\}.
\end{equation*}
Then $g_{\eps} \to \chi_{U}$ in $L^{1}(\R^{n})$, and
$
|\nabla g_{\eps}| = \frac{1}{\eps} \chi_{U \setminus U^{\eps}}.
$
Thus, for any $\phi \in C^{1}_{c}(\R^{n}; \R^{n})$, we have
\begin{align*}
\Big| \int_{\R^{n}} \chi_{U} \div \phi \, \dr x \Big|
& = \lim_{\eps \to 0} \Big| \int_{\R^{n}} g_{\eps} \div \phi \, \dr x \Big|
= \lim_{\eps \to 0} \Big| \int_{\R^{n}} \phi \cdot \nabla g_{\eps} \, \dr x \Big|\\
&\le \|\phi\|_{L^{\infty}(\R^{n}; \R^{n})} \liminf_{\eps \to 0} \frac{|U \setminus U^{\eps}|}{\eps}
 \le 2 \mathscr{M}_{*}^{n - 1}(\partial U) \|\phi\|_{L^{\infty}(\R^{n}; \R^{n})}.
\end{align*}
This implies that $U$ is a set of finite perimeter with $|D \chi_{U}|(\R^{n}) \le 2 \mathscr{M}_{*}^{n - 1}(\partial U)$.

Arguing analogously, we can also show that any compact set $K$ with $\mathscr{M}_{*}^{n - 1}(\partial K) < \infty$
is a set of finite perimeter in $\R^{n}$, with $|D \chi_{K}|(\R^{n}) \le 2 \mathscr{M}_{*}^{n - 1}(\partial K)$.
This can be shown by considering
the functions{\rm :}
\begin{equation*}
f_{\eps}(x) := \max\{0, 1 - \frac{d(x, K)}{\eps}  \},
\end{equation*}
which satisfy that $f_{\eps} \to \chi_{K}$ in $L^{1}(\R^{n})$ and
$
|\nabla f_{\eps}| = \frac{1}{\eps} \chi_{K_{\eps} \setminus K}.
$
\end{remark}

In the case $p = \infty$,  assumption \eqref{repr_F_Silhavy} is superfluous,
as shown in the following proposition,
which can be seen as a particular case of \cite[Theorem 2.4]{Silhavy2} and \cite[Theorem 2.4]{Frid2}.

\begin{proposition} \label{trace_infty_open_average}
Let $\FF \in \DM^{p}(\Omega)$ for $1 \le p \le \infty$,
and let $U \subset \Omega$ be a bounded open set such that $\mathscr{M}_{*}^{n - 1}(\partial U) < \infty$.
If $p = \infty$, or $1 \le p < \infty$ and $\FF$ satisfies
\begin{equation} \label{boundedness_average_integral_open_p}
\limsup_{\eps \to 0} \frac{1}{\eps} \int_{U \setminus U^{\eps}} |\FF|^{p} \, \dr x < \infty,
\end{equation}
then
$$
\ban{\FF \cdot \nu, \cdot}_{\partial U} \in \mathcal{M}(\partial U).
$$
Analogously, let $K \subset \Omega$ be a compact set such that $\mathscr{M}_{*}^{n - 1}(\partial K) < \infty$.
If $p = \infty$, or $1 \le p < \infty$ and $\FF$ satisfies
\begin{equation} \label{boundedness_average_integral_compact_p}
\limsup_{\eps \to 0} \frac{1}{\eps} \int_{K_{\eps} \setminus K} |\FF|^{p} \, \dr x < \infty,
\end{equation}
then
$$
\ban{\FF \cdot \nu, \cdot}_{\partial K} \in \mathcal{M}(\partial K).
$$
In particular, this implies that, if $U \Subset \Omega$, $\mathscr{M}_{*}^{n - 1}(\partial U) < \infty$,
and the same assumption on $\FF$ is made with $K = \overline{U}$,
then
$$
\ban{\FF \cdot \nu, \cdot}_{\partial \overline{U}} \in \mathcal{M}(\partial \overline{U}).
$$
\end{proposition}

\begin{proof}
\, Arguing as in the proof of Proposition \ref{trace_Silhavy_repr}, we see that,
for any $\varphi \in \Lip_{c}(\R^{n})$,
\begin{align*}
\big|\ban{\FF \cdot \nu, \varphi}_{\partial U}\big| & = \lim_{\ve \to 0} \Big|\frac{1}{\ve}\int_{U \setminus U^{\ve}} \varphi(x) \FF (x) \cdot \nabla d (x)\, \dr x \Big| \\
& \le \liminf_{\eps \to 0} \norm{\varphi}_{L^{\infty}(U \setminus U^{\eps})} \frac{1}{\ve}  |U \setminus U^{\ve}|^{\frac{1}{p'}} \|\FF\|_{L^{p}(U \setminus U^{\eps}; \R^{n})}.
\end{align*}
If $p = \infty$, we have
\begin{equation*}
|\ban{\FF \cdot \nu, \varphi}_{\partial U}| \le 2 \|\FF\|_{L^{\infty}(U; \R^{n})} \mathscr{M}_{*}^{n - 1}(\partial U) \norm{\varphi}_{L^{\infty}(\partial U)}.
\end{equation*}
If $1 \le p < \infty$, then
\begin{equation*}
|\ban{\FF \cdot \nu, \varphi}_{\partial U}|
\le \big(2 \mathscr{M}_{*}^{n - 1}(\partial U)\big)^{\frac{1}{p'}}
\limsup_{\eps \to 0} \eps^{- \frac{1}{p}} \|\FF\|_{L^{p}(U \setminus U^{\eps}; \R^{n})}\,\norm{\varphi}_{L^{\infty}(\partial U)},
\end{equation*}
from which  $\ban{\FF \cdot \nu, \cdot}_{\partial U} \in \mathcal{M}(\partial U)$, because of \eqref{boundedness_average_integral_open_p}.
The case of the compact set follows analogously from \eqref{Gauss-Green real compact average},
by employing \eqref{boundedness_average_integral_compact_p}, if $\FF \notin \DM^{\infty}(\Omega)$.
\end{proof}

This proposition may also be seen as an alternative way of obtaining a part of the results of Proposition \ref{normal trace p infty}
in the case that $E$ is an open or compact set whose boundary has finite Minkowski content.

\begin{remark} In particular,
\, Propositions {\rm \ref{trace_Silhavy_repr}} and {\rm \ref{trace_infty_open_average}}
hold also for $U = \Omega$, when $\Omega$ is an open bounded set
such that $\mathscr{M}_{*}^{n - 1}(\partial \Omega) < \infty$.
\end{remark}

\begin{remark} \label{weak_conv_traces_sets}
\, We can reinterpret Theorems {\rm \ref{interior normal trace}} and {\rm \ref{exterior normal trace}} in the distributional sense.
Indeed, given $\FF \in \DM^{p}(\Omega)$ for $1 \le p \le \infty$,
and an open bounded set $U \subset \Omega$, \eqref {main1} is equivalent to the following{\rm :}
There exists a set $\mathcal{N} \subset \R$ with $\mathcal{L}^1(\mathcal{N})=0$ such that,
for every nonnegative sequence $\{\ve_k\}$ satisfying $\ve_k \notin \mathcal{N}$ for any $k$ and $\ve_k \to 0$,
\begin{equation} \label{weak_conv_tr_int}
\FF \cdot \nu_{U^{\eps_{k}}} \, \Haus{n - 1} \res \partial U^{\eps_{k}}
= \FF \cdot D \chi_{U^{\eps_{k}}} \weakstarto -\ban{F \cdot \nu, \cdot}_{\partial U}
\end{equation}
in the distributional sense on $\R^{n}${\rm ;} that is, testing the traces against $\phi \in \Lip_{c}(\R^{n})$.

Analogously, if $U \Subset \Omega$, \eqref{main2} implies that there exists a set $\mathcal{N}' \subset \R$
with $\mathcal{L}^1(\mathcal{N}')=0$ such that,
for every nonnegative sequence $\{\ve_k\}$ satisfying $\ve_k \notin \mathcal{N}$ for any $k$ and $\ve_k \to 0$,
\begin{equation} \label{weak_conv_tr_ext}
\FF \cdot \nu_{U_{\eps_{k}}} \, \Haus{n - 1} \res \partial U_{\eps_{k}}
= \FF \cdot D \chi_{U_{\eps_{k}}} \weakstarto -\ban{F \cdot \nu, \cdot}_{\partial \overline{U}}
\end{equation}
in the distributional sense  on $\R^{n}$.
In particular, this means that, if $\ban{\FF \cdot \nu, \cdot}_{\partial U} \in \mathcal{M}(\Omega)$,
then, by the uniform boundedness principle, we have
\begin{equation*}
\limsup_{k \to \infty} \|\FF \cdot \nu_{U^{\eps_{k}}}\|_{L^{1}(\redb U^{\eps_{k}}; \Haus{n - 1})} < \infty.
\end{equation*}
Analogously, if $\ban{\FF \cdot \nu, \cdot}_{\partial \overline{U}} \in \mathcal{M}(\Omega)$, then
\begin{equation*}
\limsup_{k \to \infty} \|\FF \cdot \nu_{U_{\eps_{k}}}\|_{L^{1}(\redb U_{\eps_{k}}; \Haus{n - 1})} < \infty.
\end{equation*}

Furthermore,
if $U$ is an open set of finite perimeter in $\Omega$,
then
$$
D \chi_{U^{\eps_{k}}} \weakto D \chi_{U}
$$
in the sense of Radon measures,
where $\eps_{k}$ is a vanishing sequence for which the conclusions of Lemma {\rm \ref{boundary tubular ngb}} hold.
Indeed, for any $\phi \in C^{1}_{c}(\Omega; \R^{n})$,
\begin{equation*}
- \int_{\Omega} \phi \cdot \, \dd D \chi_{U^{\eps_{k}}}
= \int_{\Omega} \chi_{U^{\eps_{k}}} \, \div\,\phi \, \dr x
 \,\longrightarrow \,\int_{\Omega} \chi_{U} \, \div\,\phi \, \dr x
= - \int_{\Omega} \phi \cdot \, \dd D \chi_{U},
\end{equation*}
and the assertion follows by the density of $C^{1}_{c}(\Omega; \R^{n})$ in $C_{c}(\Omega; \R^{n})$
with respect to the supremum norm.
If $\overline{U}$ is also a set of finite perimeter in $\Omega$,
then $D \chi_{U_{\eps_{k}}} \weakto D \chi_{\overline{U}}$ analogously.
\end{remark}

Thanks to Remark \ref{weak_conv_traces_sets},
we can show that the normal traces on open and closed sets of finite perimeter agree with the classical dot product,
provided that $\FF$ is continuous.
It is true that $\FF \in C^{0}(\Omega; \R^{n}) \cap \DM^{p}(\Omega)$ for $1 \le p \le \infty$
implies that $\FF \in \DM^{\infty}_{\rm loc}(\Omega)$.
Thus, we may expect the existence of normal traces as locally bounded functions
by the known
theory (\cite{ctz,comi2017locally}).
Through \eqref{weak_conv_tr_int}--\eqref{weak_conv_tr_ext}, we now give a more direct proof.

\begin{proposition} \label{normal_trace_classical_repr_cont}
Let $\FF \in C^{0}(\Omega; \R^{n}) \cap \DM^{p}(\Omega)$ for $1 \le p \le \infty$,
and let $U \subset \Omega$ be an open set of finite perimeter.
Then
$$
\ban{\FF \cdot \nu, \cdot}_{\partial U} = -\FF \cdot \nu_{U} \Haus{n - 1} \res \redb U \qquad \mbox{in $\mathcal{M}_{\rm loc}(\Omega)$}.
$$
Similarly, if $\overline{U} \subset \Omega$ is a set of finite perimeter,
then
$$
\ban{\FF \cdot \nu, \cdot}_{\partial \overline{U}} = -\FF \cdot \nu_{\overline{U}} \Haus{n - 1} \res \redb \overline{U} \qquad \mbox{in $\mathcal{M}_{\rm loc}(\Omega)$}.
$$
In addition, if $U \Subset \Omega$, the previous identities hold in $\mathcal{M}(\Omega)$.
\end{proposition}

\begin{proof} \, Let $\eps_{k} \to 0$ be a sequence such that both \eqref{weak_conv_tr_int} and the conclusions of Lemma \ref{boundary tubular ngb} hold.
By Remark \ref{weak_conv_traces_sets}, we obtain
\begin{equation*}
\int_{\partial U^{\eps_{k}}} \phi \FF \cdot \nu_{U^{\eps_{k}}} \, \dr \Haus{n - 1} \,
\longrightarrow\, \int_{\redb U} \phi \FF \cdot \nu_{U} \, \dr \Haus{n - 1}
\qquad\mbox{for any $\phi \in C_{c}(\Omega)$},
\end{equation*}
since $\phi \FF \in C_{c}(\Omega; \R^{n})$
and
$$
\nu_{U^{\eps_{k}}} \, \Haus{n - 1} \res \redb U^{\eps_{k}} = D \chi_{U^{\eps_{k}}} \weakto D \chi_{U}
= \nu_{U} \, \Haus{n - 1} \res \redb U \qquad \mbox{in $\mathcal{M}(\Omega)$}.
$$
This implies
\begin{equation*}
\ban{\FF \cdot \nu, \phi}_{\partial U} = -\int_{\redb U} \phi \FF \cdot \nu_{U} \, \dr \Haus{n - 1}
\qquad\mbox{for any $\phi \in C_{c}(\Omega)$},
\end{equation*}
which means that $\ban{\FF \cdot \nu, \cdot}_{\partial U} = -\FF \cdot \nu_{U} \Haus{n - 1} \res \redb U$ in $\mathcal{M}_{\rm loc}(\Omega)$.

We can argue in a similar way with \eqref{weak_conv_tr_ext} and the fact that $D \chi_{U_{\eps_{k}}} \weakto D \chi_{\overline{U}}$
in $\mathcal{M}(\Omega)$ to prove the second part of the statement.

Finally, if $U \Subset \Omega$, there exists $\eta \in C_{c}(\Omega)$ such that $\eta \equiv 1$ on $\overline{U}$.
Hence, for any $\phi \in C^{0}(\Omega)$,  $\eta \phi F \in C_{c}(\Omega; \R^{n})$ so that
\begin{align*}
\int_{\partial U^{\eps_{k}}} \phi \FF \cdot \nu_{U^{\eps_{k}}} \, \dr \Haus{n - 1}
&= \int_{\partial U^{\eps_{k}}} \eta \phi \FF \cdot \nu_{U^{\eps_{k}}} \, \dr \Haus{n - 1}\\
& \to \int_{\redb U} \eta \phi \FF \cdot \nu_{U} \, \dr \Haus{n - 1}
= \int_{\redb U} \phi \FF \cdot \nu_{U} \, \dr \Haus{n - 1},
\end{align*}
which implies that $\ban{\FF \cdot \nu, \cdot}_{\partial U} = \FF \cdot \nu_{U} \Haus{n - 1} \res \redb U$ in $\mathcal{M}(\Omega)$.
Arguing similarly for $\overline{U}$, we complete the proof.
\end{proof}

This result also allows us to obtain  Green's identities for scalar functions in $C^{1}(\Omega)$
with gradient in $\DM^{p}(\Omega)$ and open sets of finite perimeter.

\begin{proposition} Let $u \in C^{1}(\Omega) \cap W^{1, p}(\Omega)$ for $1 \le p \le \infty$
be such that $\Delta u \in \mathcal{M}(\Omega)$, and let $U \Subset \Omega$ be an open set of finite perimeter.
Then, for any $\phi \in C^{0}(\Omega)$ with $\nabla \phi \in L^{p'}(\Omega; \R^{n})$,
\begin{equation} \label{first_Green_id_exact}
\int_{U} \phi \, \dr \Delta u + \int_{U} \nabla u \cdot \nabla \phi \, \dr x
 = - \int_{\redb U} \phi \nabla u \cdot \nu_{U} \, \dr \Haus{n - 1}.
\end{equation}
In particular, if $u \in C^{1}(\Omega) \cap W^{1, 2}(\Omega)$ with $\Delta u \in \mathcal{M}(\Omega)$, then
\begin{equation} \label{first_Green_id_bis_exact}
\int_{U} u \, \dr \Delta u + \int_{U} |\nabla u|^{2} \, \dr x
 = - \int_{\redb U} u \nabla u \cdot \nu_{U} \, \dr \Haus{n - 1}.
\end{equation}
In addition, if $u \in C^{1}(\Omega) \cap W^{1, p}(\Omega)$ and $v \in C^{1}(\Omega) \cap W^{1, p'}(\Omega)$ for $1 \le p \le \infty$
with $\Delta u, \Delta v \in \mathcal{M}(\Omega)$, then
\begin{equation}
\label{second_Green_id_exact}
\int_{U} v \, \dr \Delta u - u \, \dr \Delta v
 = - \int_{\redb U} (v \nabla u - u \nabla v) \cdot \nu_{U} \,\, \dr \Haus{n - 1}.
\end{equation}
Finally, we can also consider open sets of finite perimeter $U \subset \Omega$,
if
the supports of $\phi, u$, and $v$ are required to be compact in $\Omega$.
\end{proposition}

\begin{proof}
\, Clearly, $\nabla u \in C^{0}(\Omega; \R^n) \cap \DM^{p}(\Omega)$ and  $\nabla v \in C^{0}(\Omega; \R^n) \cap \DM^{p'}(\Omega)$.
Thus, it suffices to combine the results of Theorem \ref{first Green's identity},
Corollary \ref{second Green's identity}, and Proposition \ref{normal_trace_classical_repr_cont} to complete the proof.
\end{proof}

We notice that \eqref{first_Green_id_exact} and \eqref{second_Green_id_exact} are closely
related to the results
of Comi-Payne \cite[Proposition 4.5]{comi2017locally}, where Green's identities are achieved
for $C^{1}$ functions (whose gradients are essentially bounded $\DM$--fields)
and sets of finite perimeter.

Arguing in a similar way and employing the refinement of the Gauss-Green formula
for $\DM^{\infty}$--fields given in Proposition \ref{normal trace p infty},
we now achieve all Green's identities for Lipschitz functions with Laplacian
measure and sets of finite perimeter.

\begin{proposition}
Let $u \in \Lip_{\rm loc}(\Omega)$ be such that $\Delta u \in \mathcal{M}_{\rm loc}(\Omega)$,
and let $E \subset \Omega$ be a set of locally finite perimeter.
Then there exist interior and exterior normal traces of $\nabla u${\rm :}\,
$(\nabla u_{\ii} \cdot \nu_{E}), (\nabla u_{\ee} \cdot \nu_{E}) \in L^{\infty}_{\rm loc}(\redb E; \Haus{n - 1})$
such that, for any $v \in C^{0}(\Omega)$ satisfying $\nabla v \in L^{1}_{\rm loc}(\Omega; \R^{n})$
and $\mathrm{supp}(\chi_{E} v) \Subset \Omega$,
\begin{align}
&\int_{E^{1}} v \, \dr \Delta u + \int_{E} \nabla v \cdot \nabla u \, \dr x
  = - \int_{\redb E} v (\nabla u_{\ii} \cdot \nu_{E}) \, \dr \Haus{n - 1},\label{Green_id_1} \\
&\int_{E^{1} \cup \redb E} v \, \dr \Delta u + \int_{E} \nabla v \cdot \nabla u \, \dr x
  = - \int_{\redb E} v (\nabla u_{\ee} \cdot \nu_{E}) \, \dr \Haus{n - 1}. \label{Green_id_2}
\end{align}
For any open set $U \Subset \Omega$, the following estimates hold{\rm :}
\begin{align}
\| \nabla u_{\ii} \cdot \nu_{E} \|_{L^{\infty}(\redb E \cap U; \Haus{n  - 1})}
& \le \| \nabla u \|_{L^{\infty}(U \cap E; \R^{n})}, \label{norm_trace_grad_bound_int} \\
\| \nabla u_{\ee} \cdot \nu_{E} \|_{L^{\infty}(\redb E \cap U; \Haus{n  - 1})}
& \le \| \nabla u \|_{L^{\infty}(U \setminus E; \R^{n})}. \label{norm_trace_grad_bound_ext}
\end{align}
In addition, if $v \in \Lip_{\rm loc}(\Omega)$ with $\Delta v \in \mathcal{M}_{\rm loc}(\Omega)$,
and $\mathrm{supp}(\chi_{E} v), \mathrm{supp}(\chi_{E} u) \Subset \Omega$,
then the following formulas hold{\rm :}
\begin{align}
&\int_{E^{1}} v \, \dr \Delta u - u \, \dr \Delta v
  = - \int_{\redb E} \big(v (\nabla u_{\ii} \cdot \nu_{E}) - u (\nabla v_{\ii} \cdot \nu_{E})\big) \, \dr \Haus{n - 1},\label{Green_id_3} \\
&\int_{E^{1} \cup \redb E} v \, \dr \Delta u - u \, \dr \Delta v
  = - \int_{\redb E} \big(v (\nabla u_{\ee} \cdot \nu_{E}) - u (\nabla v_{\ee} \cdot \nu_{E})\big) \, \dr \Haus{n - 1}. \label{Green_id_4}
\end{align}
In particular, if $\mathrm{supp}(\chi_{E} u) \Subset \Omega$, then
\begin{align}
&\int_{E^{1}} u \, \dr \Delta u + \int_{E} |\nabla u|^{2} \, \dr x
  = - \int_{\redb E} u (\nabla u_{\ii} \cdot \nu_{E}) \, \dr \Haus{n - 1}, \label{Green_id_5}  \\
&\int_{E^{1} \cup \redb E} u \, \dr \Delta u + \int_{E} |\nabla u|^{2} \, \dr x
= - \int_{\redb E} u (\nabla u_{\ee} \cdot \nu_{E}) \, \dr \Haus{n - 1}. \label{Green_id_6}
\end{align}
\end{proposition}

\begin{proof}
\, Since $\nabla u \in \DM^{\infty}_{\rm loc}(\Omega)$, the existence of interior and exterior normal
traces in $L^{\infty}_{\rm loc}(\redb E; \Haus{n - 1})$
and estimates \eqref{norm_trace_grad_bound_int}--\eqref{norm_trace_grad_bound_ext} follow
from \cite[Theorem 4.2]{comi2017locally} and Proposition \ref{normal trace p infty}.
Analogously, \eqref{Green_id_1}--\eqref{Green_id_2} are an immediate consequence
of \eqref{G-G phi Sobolev int}--\eqref{G-G phi Sobolev ext},
with $\FF = \nabla u$ and $\phi = v$.

In addition, if $\mathrm{supp}(\chi_{E} u) \Subset \Omega$
and $v \in \Lip_{\rm loc}(\Omega)$ with $\Delta v \in \mathcal{M}_{\rm loc}(\Omega)$,
then we can exchange the role of $u$ and $v$ in \eqref{Green_id_1} and \eqref{Green_id_2}:
\begin{align}
&\int_{E^{1}} u \, \dr \Delta v + \int_{E} \nabla v \cdot \nabla u \, \dr x
  = - \int_{\redb E} u (\nabla v_{\ii} \cdot \nu_{E}) \, \dr \Haus{n - 1},\label{Green_id_1_v}  \\
&\int_{E^{1} \cup \redb E} v \, \dr \Delta v + \int_{E} \nabla v \cdot \nabla u \, \dr x
 = - \int_{\redb E} u (\nabla v_{\ee} \cdot \nu_{E}) \, \dr \Haus{n - 1}.\label{Green_id_2_v}
\end{align}
Thus, it suffices to subtract \eqref{Green_id_1_v} from \eqref{Green_id_1}
to obtain \eqref{Green_id_3}, and to subtract \eqref{Green_id_2_v} from \eqref{Green_id_2} to obtain \eqref{Green_id_4}.
Finally, choosing $u = v$ in \eqref{Green_id_1}--\eqref{Green_id_2},
we obtain \eqref{Green_id_5}--\eqref{Green_id_6}.
\end{proof}

\section{\, Normal Traces for Open Sets as the Limits of the Classical Normal Traces for Smooth Sets}

In this section, we show that the approximations
of a general open set $U$ can be refined in such a way that $\ban{\FF \cdot \nu, \cdot}_{\partial U}$
and $\ban{\FF \cdot \nu, \cdot}_{\partial \overline{U}}$ can be regarded as the limits of the classical normal traces on the boundaries of smooth sets.
In the case that the open set $U$ has continuous boundary,
we can exhibit explicit approximating families of open sets with smooth boundary as deformations to the open set $U$.

\subsection{\, The general case}

In order to achieve the smooth approximation, we recall another remarkable result concerning the approximation
of any open set by an increasing sequence of open sets with smooth boundary,
a very simple proof of which was given by Daners \cite[Proposition 8.2.1]{Daners_2008}.

\begin{proposition} \label{smooth approx general open}
Let $U \subset \R^{n}$ be an open set.
Then there exists a sequence of bounded open sets $U_{k}$ with boundary of class $C^{\infty}$
such that $U_{k} \Subset U_{k + 1} \Subset U$ and $\bigcup_{k} U_{k} = U$.
\end{proposition}

We can use this result to extend Theorem \ref{interior normal trace},
via showing that the normal trace can be approximated by a sequence of the classical normal traces on smooth boundaries.

\begin{theorem} \label{interior trace smooth general}
Let $U \subset \Omega$ be a bounded open set, and let $\FF \in \DM^p(\Omega)$.
Then, for any $\phi \in C^{0}(\Omega) \cap L^{\infty}(\Omega)$ with $\nabla \phi \in L^{p'}(\Omega; \R^{n})$,
there exists a sequence of bounded open sets $U_{k}$ with boundary of class $C^{\infty}$ such that $U_{k} \Subset U$, $\bigcup_{k} U_{k} = U$, and
\begin{equation}
\label{main1 smooth version}
\ban{\FF \cdot \nu, \phi}_{\partial U} = - \lim_{k \to \infty} \int_{\partial U_{k}} \phi \FF \cdot \nu_{U_{k}} \,\, \dr\Haus{n - 1},
\end{equation}
 where $\nu_{U_{k}}$ is the inner unit normal to $U_{k}$.
In addition, \eqref{main1 smooth version} holds also for any open set $U \subset \Omega$,
provided that $\mathrm{supp}(\phi) \cap U^{\delta} \Subset \Omega$ for any $\delta > 0$.
\end{theorem}

\begin{proof}
\, We just need to apply Proposition \ref{smooth approx general open} to $U$
in order to obtain an approximating sequence of smooth sets $U_{m}$,
and then argue as in the proof of Theorem \ref{interior normal trace}
with respect to any $U_{m}$.

We note that sets $U_{m}^{\eps}$ have smooth boundaries, for any $0 < \eps < \delta_{m}$, for some $\delta_{m}$ sufficiently small.
Indeed, the (signed) distance function $d_{m}$ from $\partial U_{m}$ is smooth in $U_{m} \setminus \overline{U_{m}^{\delta_{m}}}$
and satisfies $\nabla d_{m}(x) = \nu_{U_{m}^{\eps}}(x)$ for any $x \in \partial U_{m}^{\eps}$,
which implies that $|\nabla d_{m}(x)| = 1$ for any $x \in U_{m} \setminus \overline{U_{m}^{\delta_{m}}}$
(for a proof of these facts, we refer to \cite[Appendix B]{giusti1984minimal} and \cite[Lemma 14.16]{gilbarg2015elliptic}).
Therefore, the level sets $\{ d_{m} = \eps \} = \partial U_{m}^{\eps}$ are smooth for any $\eps \in [0, \delta_{m})$.

Then we obtain a sequence of open bounded sets $U_{m}^{\eps_{j}}$ with smooth boundary satisfying
\begin{equation} \label{m, j approx eq}
\int_{U_{m}^{\eps_{j}}} \phi \, \dd \div \FF + \int_{U_{m}^{\eps_{j}}} \FF \cdot \nabla \phi \, \dr x
= - \int_{\partial U_{m}^{\eps_{j}}} \phi \FF \cdot \nu_{U_{m}^{\eps_{j}}} \,\, \dd \Haus{n - 1}
\end{equation}
for some decreasing sequence $\eps_{j} \to 0$ and any $m, j \in \N$.

Clearly, $U^{\eps_{j}}_{m} \Subset U^{\eps_{j}}_{m + 1}$ and $U^{\eps_{j}}_{m} \Subset U^{\eps_{j + 1}}_{m}$,
so that we can find a subsequence $U^{\eps_{k}}_{k} =: U_{k}$ satisfying $U_{k} \Subset U_{k + 1}$, $U_{k} \Subset U$,
and $\bigcup_{k} U_{k} = U$.
Therefore, we can pass to the limit on the left-hand side of \eqref{m, j approx eq}
by the Lebesgue theorem to obtain \eqref{main1 smooth version}
(in the case that $U$ is not bounded, we employ the condition on the support of $\phi$).
\end{proof}

Similarly, an analogous kind of approximation can also be shown for closed sets.

\begin{proposition} \label{smooth approx general closed}
Let $C \subset \R^{n}$ be a closed set.
Then there exists a sequence of closed sets $C_{k}$ with boundary of class $C^{\infty}$ such that
$C_{k} \supset \mathring{C}_{k} \supset C_{k + 1} \supset \mathring{C}_{k + 1} \supset C$ and $\bigcap_{k} C_{k} = C$.
In addition, if $C$ is bounded, then  the closed sets $C_{k}$ can be chosen to be bounded.
\end{proposition}

\begin{proof} \,  Let $U := \R^{n} \setminus C$.
Then it suffices to define $C_{k} := \R^{n} \setminus U_{k}$
and apply Proposition \ref{smooth approx general open} to $U$.
The result follows easily.

In the case that $C$ is bounded, then, for any $\delta > 0$,
there exists $k_{0}$ large enough such that $\partial U_{k} \cap \partial C_{\delta} = \emptyset$ for any $k \ge k_{0}$,
where $C_{\delta} = \{ x \in \R^{n} : \dist(x, C) < \delta \}$.
Then we set $C_{k} = \overline{C_{\delta}} \setminus U_{k}$, up to relabeling the sequence $U_{k}$ in such a way that it starts from $k_{0}$.
\end{proof}

Arguing similarly as before, Proposition \ref{smooth approx general closed}
can be used to represent the exterior normal trace as the limit of the classical normal traces on smooth boundaries,
thus improving the result of Theorem \ref{exterior normal trace}.

\begin{theorem} \label{exterior trace smooth general}
Let $U \Subset \Omega$ be an open set, and let $\FF \in \DM^p(\Omega)$.
Then, for any
$\phi \in C^{0}(\Omega)$ with $\nabla \phi \in L^{p'}(\Omega; \R^{n})$,
there exists a sequence of bounded open sets $V_{k}$ with boundary of class $C^{\infty}$
such that $U\Subset V_{k}\subset \Omega$, $\bigcap_{k} V_{k} = \overline{U}$, and
\begin{equation}
\label{main2 smooth version}
\ban{\FF \cdot \nu, \phi}_{\partial \overline{U}}
= - \lim_{k \to \infty} \int_{\partial V_{k}} \phi \FF \cdot \nu_{V_{k}} \,\, \dr \Haus{n - 1},
\end{equation}
where $\nu_{V_{k}}$ is the classical inner unit normal to $V_{k}$.
In addition, \eqref{main2 smooth version} holds also for any open set $U$ satisfying $\overline{U} \subset \Omega$,
provided that $\mathrm{supp}(\phi)$ is compact in $\Omega$.
\end{theorem}

\begin{proof}
\, It suffices to  define $V_{k} := \mathring{C_{k}}$
and apply Proposition \ref{smooth approx general closed} to
$C = \overline{U}$.
Then the result follows in an analogous way as in the proof of Theorem \ref{interior trace smooth general}.
\end{proof}

\subsection{\, The case of $C^0$ open sets}
We now consider the question of constructing the interior and exterior normal traces
as the limit of classical normal traces over {\it smooth} approximations of the open bounded
set $U$ with $C^{0}$ boundary.
In general, as it has been explained in the introduction,
it is a challenging question to approximate an open (bounded) set $U$
with smooth domains $U^{\ve}$ essentially from the {\it inside}
in such a way that $\div \FF (U^{\ve}) \to \div \FF (U)$ and an {\it interior} Gauss-Green formula holds
for unbounded $\DM^p$--fields.
Indeed, in \S 5,
we have used the standard signed distance $d$ to obtain Theorems \ref{interior normal trace}
and \ref{exterior normal trace},
but the approximating sets are not smooth.
We have shown that such results
in Theorems \ref{interior trace smooth general} and \ref{exterior trace smooth general}
can be improved.
On the other hand, such an approximation is quite abstract,
and it gives little insight in the actual shape of the approximating sets.

\begin{remark} \label{remark_C^0_sets}
\, We observe that an open bounded set $U$ with $C^0$ boundary might not have finite perimeter{\rm ;}
such examples include von Koch's snowflake \cite{vonKoch}.
Also, we do not have a notion of unit normals to a $C^0$ open set.
Thus, such a type of sets is  more general in this sense.
On the other hand, an open set of finite perimeter may have really wild topological
boundary even with full Lebesgue measure {\rm (}see {\it e.g.} \cite [Example 12.25] {Maggi}{\rm )}
so that it is not a $C^{0}$ open set in general,
since it is well know that, if $\partial U$ can be seen locally as the graph of a continuous function,
then $|\partial U| = 0$.
\end{remark}

We now exhibit here a rather explicit family of open smooth sets approximating a given bounded open set with $C^{0}$ boundary
from both the interior and the exterior.
To this purpose, we consider a different type of distance, the regularized distance $\rho$,
which was introduced in Lieberman \cite{L}.

\begin{definition}
$\rho$ is a regularized distance for $U$  if the following conditions hold{\rm :}
\begin{enumerate}
\item[\rm (i)] $\rho \in C^2(\R^n \setminus \partial U)  \cap \textnormal{Lip}(\R^n)${\rm ;}
\item[\rm (ii)] The ratios $\frac{\rho(x)}{d(x)}$ and $\frac{d(x)}{\rho(x)}$ are positive and uniformly bounded
   for all $x \in \R^n \setminus \partial U$, where $d$ is the signed distance introduced in \S {\rm 5}.
\end{enumerate}
\end{definition}

It was proved by Lieberman \cite[Lemma 1.1] {L} that any open set $U$ has a regularized distance,
since the signed distance $d$ is a $1$-Lipschitz function (Lemma \ref{properties distance}).
Indeed, given any $\eta \in C^2(\R^n)$, $\textnormal{supp } \eta \subset B(0, 1)$,
and $\int_{\R^n} \eta(z) \dr z =1$,
we can define
\begin{equation}
\label{ladefinicion}
 G(x,\tau)= \int_{B(0, 1)}d( x - \frac{\tau}{2}z) \eta(z) \, \dr z.
\end{equation}

The regularized distance $\rho$ is then given
by the equation:
\begin{equation}
\label{ladefinicion2}
\rho(x)=G(x, \rho(x)),
\end{equation}
which has a unique solution for every $x \in \R^n$.
Moreover,
\begin{equation}
\frac{1}{2} \leq \frac{\rho(x)}{d(x)} \leq 2 \qquad  \textnormal{ for all } x \in \R^n \setminus \partial U.
\end{equation}
Then the following result holds, for which we refer to Lieberman \cite[Lemma 1.1, Corollary 1.2]{L} and the comments before it.

\begin{lemma}
Every open set $U$ has a regularized distance $\rho$.
Moreover, if  $\eta \in C^{\infty}(\R^n)$ is chosen for \eqref{ladefinicion},
then $\rho \in C^{\infty}(\R^n\setminus \partial U)$.
\end{lemma}

Even though every open set $U$ has a regularized distance, it is important to obtain
properties concerning the non-degeneracy of gradient $\nabla \rho$.
Indeed,
if the gradient of $\rho$ does not vanish in a neighborhood of $\partial U$, then
we can apply the techniques in \S 5 to obtain the {\it interior} and {\it exterior} Gauss-Green formulas
for $\DM^p$--fields.

The non-degeneracy of $\nabla \rho$ might not be true for general open sets $U$ (see \cite[Corollary 1.2]{L} and the comments following it).
However, it was proved by Ball-Zarnescu \cite[Proposition 3.1]{BallZarnescu} that this property holds for $C^0$ domains, which yields
the following result.

\begin{theorem}[Ball-Zarnescu]\label{nondegeneracy}
\, If $U$ is an open bounded set with $C^0$ boundary,
then $|\nabla \rho (x)| \neq 0$ for all $x$ in a neighborhood of $\partial U$ but $x \notin\partial U$.
\end{theorem}

\begin{remark} \label{local statement nondegeneracy}
\, The argument of the proof of Theorem {\rm \ref{nondegeneracy}} relies on both the construction of a suitable good neighborhood
for any point of $\partial U$ and the use of the compactness assumption. Hence, we see that the local version of this result
also applies to general open sets with $C^{0}$ boundary{\rm :}
For any compact subset $K\subset \partial U$, there exists a suitable neighborhood $V$ of $K$
such that $|\nabla \rho(x)| \neq 0$ for any $x \in V \setminus \partial U$.
\end{remark}

Thanks to Ball-Zarnescu's theorem (Theorem \ref{nondegeneracy}),
we can proceed as in \S 5 to obtain an analogous statement,
by approximating $U$ and $\overline{U}$ in $\Omega$ with the following smooth sets:
\begin{align*}
U^{\ve, \rho} := \{ x \in \R^{n} : \rho(x) > \eps \}, \quad
U_{\ve, \rho} &:= \{ x \in \R^{n} : \rho(x) > - \eps \} \,\, \qquad\mbox{for $\ve > 0$}.
\end{align*}

\begin{theorem}[Interior normal trace via smooth approximations] \label{interior normal trace smooth}
\, Let $U \subset \Omega$ be a bounded open set with $C^{0}$ boundary,
and let $\FF \in \DM^p(\Omega)$ for $1 \le p \le \infty$.
Then, for any $\phi \in C^{0}(\Omega) \cap L^{\infty}(\Omega)$ with $\nabla \phi \in L^{p'}(\Omega; \R^{n})$,
there exists a set $\mathcal{N} \subset \R$ with $\mathcal{L}^1(\mathcal{N})=0$ such that,
for every nonnegative sequence $\{\ve_k\}$ satisfying $\ve_k \notin \mathcal{N}$ for any $k$ and $\ve_k \to 0$,
\begin{equation}
\label{G-G int C0}
\ban{\FF \cdot \nu, \phi}_{\partial U}
= \int_{U} \phi \, \dd \div \FF + \int_{U} \FF \cdot \nabla \phi \, \dr x
= - \lim_{k \to \infty} \int_{\partial U^{\eps_{k}, \rho}} \phi \FF \cdot \nu_{U^{\eps_{k}, \rho}} \,\, \dr \Haus{n - 1},
\end{equation}
 where $\nu_{U^{\eps_{k}, \rho}}$ is the inner unit normal to the smooth sets $U^{\ve_k, \rho}$.
In addition, \eqref{G-G int C0} holds also for any open set $U \subset \Omega$,
provided that $\mathrm{supp}(\phi) \cap U^{\delta} \Subset \Omega$ for any small $\delta > 0$.
\end{theorem}

\begin{proof}
\, Since $\rho$ is smooth and $|\nabla \rho(x)| \neq 0$ for any $x \in U \setminus U^{\eps, \rho}$ for small enough $\eps$,
it follows that $\{ x \in \R^n: \rho(x)=\ve \}$ is a smooth hypersurface in $\R^n$.
Therefore,
$\partial U^{\ve, \rho} = \partial^* U^{\ve, \rho}$ and
\begin{equation*}
\frac{\nabla \rho}{|\nabla \rho|}(x)
= \nu_{U^{\eps, \rho}}(x) \qquad\,\textnormal{ for every } x \in \partial U^{\ve, \rho}.
\end{equation*}
We can now proceed in the same way as in the second step of the proof of Theorem \ref{interior normal trace}
by noticing that the only difference is in the application of the coarea formula and in the use of $\rho$,
instead of $d$,
in the definition of $\psi^{U^{\delta, \rho}}_{\ve, \rho}$ for $\delta, \eps > 0$.
Indeed, using Theorem \ref{nondegeneracy}, we rewrite \eqref{test psi} in the case $U \Subset \Omega$ as follows:
\begin{align*}
\int_{U} \psi^{U^{\delta, \rho}}_{\eps, \rho} \, \dd \div(\phi \FF)
& = - \int_{U^{\delta, \rho} \setminus U^{\delta + \eps, \rho}} \phi \FF \cdot \nabla \rho \, \dr x \\
& = - \int_{U^{\delta, \rho} \setminus U^{\delta + \eps, \rho}} \phi \FF \cdot \frac{\nabla \rho}{|\nabla \rho|}  |\nabla \rho| \, \dr x \\
& = - \int_{\delta}^{\delta + \eps} \int_{\partial U^{t, \rho}} \phi \FF \cdot \frac{\nabla \rho}{|\nabla \rho|}  \, d \Haus{n - 1} \, \dr t,
\end{align*}
by the coarea formula \eqref{coarea super-sublevel}
with $u = \rho$ and $g = \chi_{U^{\delta, \rho}} \phi \FF \cdot \frac{\nabla \rho}{|\nabla \rho|} \, |\nabla \rho|$,
since
$$
\mathrm{ess inf} |\nabla \rho| > 0 \qquad \mbox{on $U^{\delta, \rho} \setminus U^{\delta + \eps, \rho}$ for any $\delta, \eps > 0$}.
$$

Then we can proceed as in Steps 2--3 of the proof of Theorem \ref{interior normal trace}.
Finally, in the case that $U$ is not bounded, we employ Remark \ref{local statement nondegeneracy}
to obtain the desired result.
\end{proof}

\begin{remark}
\, In particular, Theorem {\rm \ref{interior normal trace smooth}} implies that, if $\Omega$ is of $C^{0}$ boundary,
the  Gauss-Green formula up to the boundary holds by approximating $\partial \Omega$ with a sequence of smooth sets.
This can be seen by taking $U = \Omega$ in \eqref{G-G int C0}.
\end{remark}

Analogously, we also have a smooth version of Theorem \ref{exterior normal trace},
in which we employ the fact that $|\partial U| = 0$  if $U$ has a $C^{0}$ boundary,
by Remark \ref{remark_C^0_sets}, in order to integrate $\FF \cdot \nabla \phi$ only on $U$.

\begin{theorem}[Exterior normal traces via smooth approximations] \label{exterior normal trace smooth}
Let $U \Subset \Omega$ be a $C^0$ open set, and  let $\FF \in \DM^p(\Omega)$ for $1 \le p \le \infty$.
Then, for any $\phi \in C^{0}(\Omega)$ with $\nabla \phi \in L^{p'}(\Omega; \R^{n})$,
there exists a set $\mathcal{N} \subset \R$ with $\mathcal{L}^1(\mathcal{N})=0$ such that,
for every nonnegative sequence $\{\ve_k\}$ satisfying $\ve_k \notin \mathcal{N}$ for any $k$ and $\ve_k \to 0$,
\begin{equation} \label{G-G ext C0}
\ban{\FF \cdot \nu, \phi}_{\partial \overline{U}}
= \int_{\overline{U}} \phi \, \dd \div \FF + \int_{U} \FF \cdot \nabla \phi \, \dr x
= - \lim_{k \to \infty} \int_{\partial U_{\eps_{k}, \rho}} \phi \FF \cdot \nu_{U_{\eps_{k}, \rho}} \,\, \dr \Haus{n - 1},
\end{equation}
where $\nu_{U_{\eps_{k}, \rho}}$ is the inner unit normal to the smooth sets $U_{\ve_k, \rho}$.
In addition, \eqref{G-G ext C0} holds also for any open set $U$ satisfying $\overline{U} \subset \Omega$,
provided that $\mathrm{supp}(\phi)$ is compact in $\Omega$.
\end{theorem}

\section{\, The Gauss-Green Formula on Lipschitz Domains}

In Ball-Zarnescu \cite{BallZarnescu}, it is shown that,
if a domain $U$ is of class $C^0$, then  there is a canonical smooth field of good directions defined
in a suitable neighborhood of $\partial U$,
in terms of which a corresponding flow can be defined.
By means of
this flow,
$U$ can be approximated from both the interior and the exterior
by diffeomorphic domains of class $C^{\infty}$;
see also Hofmann-Mitrea-Taylor \cite{HofmannMitreaTaylor} for the definition
of a continuous vector field transversal to the boundary
of an open set of locally finite perimeter.

In a related issue,  Chen-Frid in \cite{CF1,CF2}
introduced the notion of {\it regular Lipschitz deformable boundary}
(see Definition \ref{aquiaqui}).
Then, for a bounded open set $U$ satisfying this condition,
they proceeded to obtain the Gauss-Green formulas for $\DM^p$--fields $\FF$.
For the case $p \neq \infty$, the normal trace of $\FF$ is defined
as a distribution which is expressed as an average over a neighborhood of Lipschitz
deformable boundary $\partial U$ determined by
the Lipschitz deformations.
However, as explained in the introduction,
the main goal of this paper is to present the Gauss-Green formula for the case $p \neq \infty$,
by using the classical normal traces $\FF \cdot \nu$ which are defined on almost every surface
that approximates $\partial U$.
Thus, the present paper aligns with the later work by Chen-Torres-Ziemer \cite{ctz},
in which the Gauss-Green formula for bounded $\DM$--fields has been
established over arbitrary sets of finite perimeter $E$
via the normal trace on $\partial^*E$ as the limit of classical normal
traces on smooth approximations of $E$.
The goal of this section is to show that the main result in Ball-Zarnescu \cite{BallZarnescu}
implies that any Lipschitz domain satisfies condition (ii) of Definition \ref{aquiaqui}, which indicates that
condition (ii) holds automatically for a Lipschitz domain.

For a domain $U$ in the class $C^0$,
the concept of {\it a good direction} (see Definition \ref{gooddirection})
has been introduced, and the following result has been established in
\cite{BallZarnescu}.

\begin{proposition}{\rm \cite[Proposition 2.1]{BallZarnescu}}\label{prop 2.1}\,\,
Let $U \subset \R^n$ be a bounded open set with boundary of class $C^0$.
Then there exist a neighborhood $V$ of $\partial U$ and a smooth function $G:V \to \Sph^{n-1}$
so that, for each $P \in V$, the unit vector $G(P)$ is a good direction.
\end{proposition}

\begin{remark} \label{local prop 2.1}
\, Proposition {\rm \ref{prop 2.1}} can be localized.
Indeed, if $U$ is an unbounded open set with boundary of class $C^{0}$,
then $U \cap B(0, R)$ is a bounded open set with boundary of class $C^{0}$ for any $R > 0$.
Therefore, there exist a neighborhood $V$ of any compact set $K \subset \partial U$
and a smooth function $G:V \to \Sph^{n-1}$ so that,
for each $P \in V$, the unit vector $G(P)$ is a good direction.
\end{remark}

We recall that $\nu=G(P)$ is a good direction if $\partial U$ is the graph of a continuous function
in a small neighborhood $B(P,\delta)$ and in some system of coordinates $(y',y_n)$,
where $\nu$ is a unit vector in the direction of $y_n$.
Using the field of good directions $G(p)$, a flow
$S(\cdot)(\cdot):\R \times \R^n \to \R^n$
can be defined through:
\begin{equation}   \label{flow good directions}
\dot{x}(t)= \gamma(x(t))G(x(t)), \quad\, x(0)=x_0,
\end{equation}
with $x(t)=S(t)(x_0)$ as the solution of the initial value
problem \eqref{flow good directions} for the differential equation at time $t$,
where $\gamma$  is an appropriate smooth
function (see \cite[\S 4]{BallZarnescu} for the details on
the construction of the flow).

By exploiting the properties of the flow of good directions,
the following theorem has been proved in Ball-Zarnescu \cite{BallZarnescu},
which is very helpful
in the rest of this section.

\begin{theorem}{\rm \cite[Theorem 5.1]{BallZarnescu}}\label{Ball}\,\,
Let $U \subset \R^n, n \geq 2$, be a bounded domain of class $C^0$.
Let $\rho$ be a regularized distance defined in \S {\rm 7}.
For $\varepsilon \in \R$, define
\begin{equation*}
U^{\varepsilon, \rho}=\{x \in \R^n:\rho (x) > \varepsilon\},\qquad U_{\varepsilon, \rho}
=\{x \in \R^n:\rho (x) > -\varepsilon\}.
\end{equation*}
Then there exists $\varepsilon_0 = \varepsilon_0(U)>0$ such that,
if $0 < \varepsilon < \varepsilon_0$, $U_{\varepsilon,\rho}$ and $U^{\varepsilon,\rho}$ are
bounded domains of class $C^{\infty}$ and satisfy the following{\rm :}
\begin{enumerate}
\item[\rm (i)] $\bigcap_{0 <\varepsilon<\varepsilon_{0}} U_{\varepsilon,\rho} =\overline{U},\, \bigcup_{0< \varepsilon< \varepsilon_{0}}U^{\varepsilon,\rho}=U$,
   \text{ and }\\
   $\overline{U}^{\varepsilon',\rho} \subset U^{\varepsilon,\rho}$
     and $U_{\varepsilon',\rho} \supset \overline{U}_{\varepsilon,\rho}$  if $ 0< \varepsilon < \varepsilon' < \varepsilon_0$.

\smallskip
\item[\rm (ii)]  For $-\varepsilon_0 < \varepsilon < \varepsilon_0$,
there is a homeomorphism $f(\varepsilon, \cdot)$ of $\R^n$ onto $\R^n$ with inverse denoted $f^{-1}(\varepsilon,\cdot)$ so that

\smallskip
\begin{itemize}

\item  $f(\varepsilon, \overline{U}) =  \overline{U}^{\varepsilon,\rho}$ and
    $f(\varepsilon,\partial U)=\partial U^{\varepsilon,\rho}\, $ \text{for } $\ve >0$, \\
     $f(\varepsilon, \overline{U}) =  \overline{U}_{-\varepsilon,\rho}$ and $f(\varepsilon, \partial U)=\partial U_{-\varepsilon,\rho}$ \text{for } $\ve <0${\rm ;}

\smallskip
\item $f(\varepsilon, x) = x $ for $|\rho(x)| \ge 3 |\varepsilon|$, so that $f(0,\cdot)$ is the identity{\rm ;}

\smallskip
\item $ f(\varepsilon,\cdot) : \R^n \setminus \partial U \to \R^n \setminus \partial U^{\varepsilon,\rho}$ for $\ve >0$,
  and $f(\varepsilon,\cdot) : \R^n \setminus \partial U \to \R^n \setminus \partial U_{-\varepsilon,\rho}$ for $\ve <0$ are both
  $C^{\infty}$ diffeomorphisms. In addition,
  $$
  f, f^{-1} \in C^{0}((-\eps_{0}, \eps_{0}) \times \R^{n}; \R^{n}).
  $$
\end{itemize}

\item[\rm (iii)] There is a map $g : (0, \eps_{0}) \times (- \eps_{0}, 0) \times \R^{n} \to \R^{n}$ such that,
if $\eps \in (0, \eps_{0})$ and $\eps' \in (- \eps_{0}, 0)$, then

\smallskip
\begin{itemize}
\item $g(\eps, \eps', \cdot)$ is a $C^{\infty}$ diffeomorphism of $\R^{n}$ onto $\R^{n}$ with inverse $g^{-1}(\eps, \eps', \cdot) : \R^{n} \to \R^{n}${\rm ;}

\smallskip
\item $g(\eps, \eps', U^{\eps, \rho}) = U_{-\eps', \rho}$, $g(\eps, \eps', \partial U^{\eps, \rho}) = \partial U_{- \eps', \rho}${\rm ;}

\smallskip
\item $g(\eps, \eps', x) = x $ for $ 3\eps\le \rho(x) \le 3 \eps' $.
\end{itemize}
\end{enumerate}
\end{theorem}

\begin{remark} \label{convergence to the identity}
\, An easy consequence of Theorem {\rm \ref{Ball}(ii)} is that $f(\eps, \cdot)$ converges uniformly to the identity.
Indeed,
\begin{equation*}
|f(\eps, x) - x| = 0 \qquad \text{in} \ \{x\,:\,  |\rho(x)| \ge 3 |\eps| \}, \end{equation*}
and
\begin{equation*}
|f(\eps, x) - x| \le \mathrm{diam}(\{ |\rho| < 3 |\eps| \}) \qquad  \text{in} \ \{x\,:\, |\rho(x)| < 3 |\eps| \}, \end{equation*}
since $f(\eps, \cdot)$ is injective so that, for any $x$ such that $|\rho(x)| < 3 |\eps|$,
$f(\eps, x)-x$ cannot belong to $\{ |\rho| \ge 3 |\eps| \}$.
Then this shows that
\begin{equation*}
\sup_{x \in \R^{n}} |f(\eps, x) - x| \to 0 \qquad \text{as $\eps \to 0$}.
\end{equation*}
\end{remark}

We now proceed to show that condition (ii) of Definition \ref{aquiaqui} is not necessary;
that is, any Lipschitz domain admits a bi-Lipschitz deformation.
Even though the proof of Theorem \ref{bi Lip no eps} is outlined in \cite[Remark 5.3] {BallZarnescu},
we present a detailed proof for our purpose of the subsequent developments.
First we recall a result of Lieberman \cite[Lemma A.1]{L2} and an extension theorem for Sobolev functions.

\begin{lemma} \label{Lieb Lip boundary estimates}
If $U$ is a Lipschitz domain with Lipschitz constant of the local parametrization of $\partial U$ uniformly bounded,
then there exists $\delta > 0$ such that, for any $x$ with $0 < |\rho(x)| < \delta$,
\begin{equation} \label{below estimate normal der int}
\frac{\partial \rho}{\partial x_{n}}(x) \ge \frac{2}{3 \sqrt{1 + L^{2}}},
\end{equation}
where $x = (x', x_{n})$ is an orthonormal coordinate system for which $\partial U$ is locally parametrized
by a Lipschitz function with Lipschitz constant less or equal to $L$, and $e_{n}$ is a good direction.
\end{lemma}

\begin{remark}
\, Lemma {\rm \ref{Lieb Lip boundary estimates}} applies to
the case that $U$ is a bounded Lipschitz domain,
since, by compactness, $\partial U$ can be covered with a finite number of charts of the local Lipschitz parametrization.
However, there are cases of unbounded $U$ which still satisfy the assumption, such as the half-spaces.
In addition, since \eqref{below estimate normal der int} is a local result,
it holds for any open set $U$ with Lipschitz boundary, up to the localization to a bounded subset of $\partial U$.
\end{remark}

\begin{lemma} \label{extension_thm_Sobolev_domain}
Let $U$ be a domain satisfying the minimal smoothness conditions{\rm ;} that is,
there exist $\eps > 0$, $N \in \N$, $M > 0$, and a sequence of open sets $\{V_{i}\}$ such that{\rm :}
\begin{enumerate}
\item[\rm (i)] If $x \in \partial U$, then $B(x, \eps) \subset V_{i}$ for some $i${\rm ;}

\vspace{1mm}
\item[\rm (ii)]  No point in $\R^{n}$ is contained in more than $N$ of $\{V_{i}\}${\rm ;}

\vspace{1mm}
\item[\rm (iii)]  For each $i$, there exists a Lipschitz function $\psi_{i}$ with Lipschitz constant
  less or equal to $M$ such that $V_{i} \cap U$ is the subgraph of $\psi_{i}$ inside $V_{i}$.
\end{enumerate}
Then there exists a continuous linear operator $\mathcal{E}: W^{k, p}(U) \to W^{k, p}(\R^{n})$,
for any $k \in \N$ and $1 \le p \le \infty$, such that $\mathcal{E}(f) = f$ on $U$.
\end{lemma}

For the proof of this result, we refer to
Stein \cite[Chapter VI, \S 3, Theorem 5]{Stein}.
These conditions are satisfied for any bounded open Lipschitz domain $U$.
However, they may fail in the general case of an unbounded open Lipschitz domain.

\begin{theorem} \label{bi Lip no eps}
If $U$ is a bounded Lipschitz domain,
then $f_{\ve}(\cdot): \overline{U} \to \overline{U}^{\ve, \rho}$ is bi-Lipschitz
with Lipschitz constants uniformly bounded in $\eps > 0$.
If $\ve <0$, the corresponding result is also true.
\end{theorem}

\begin{proof}
\, Unless otherwise stated, in this proof, $\nabla$ stands for the gradient
with respect to the spatial variables, denoted by $x, y$, or $z\in \mathbb{R}^n$.
We divide the proof into five steps.

\smallskip
1. We first consider $0 < \ve < \ve_0$, for some $\eps_{0} > 0$ sufficiently small to be assigned.
From Theorem \ref{Ball}, $f(\ve, \cdot): \R^n \to \R^n$ is continuous with continuous inverse $f^{-1}$.
Thus, the restriction:
\begin{equation*}
  f(\ve, \cdot) : \, \overline{U} \to \overline{U}^{\ve, \rho}
\end{equation*}
is also continuous with continuous inverse.
We need to show that $f(\ve, \cdot):  \overline{U} \to \overline{U}^{\ve, \rho}$ is Lipschitz
with the Lipschitz constant uniformly bounded in $\ve\in (0, \eps_0)$.
Following the proof of \cite[Theorem 5.1]{BallZarnescu}, the continuous map $f(\ve, \cdot)$ is defined as
\begin{equation}
\label{asisedefine}
f(\ve,x)=
\begin{cases}
S(t(\ve,x))x  \quad &\mbox{for}\,\, x \in \overline{U}\setminus U^{3 \ve, \rho}, \\
x\quad  &\mbox{for}\,\, x \in U^{3 \ve, \rho},
\end{cases}
\end{equation}
where $t(\ve,x)$ is the unique $t \geq 0$ such that
\begin{equation}
\label{smart}
\rho (S(t(\ve,x))x)=\rho(x) + h(\ve, \rho(x)),
\end{equation}
and $h (\ve, \cdot): \R_{+} \to [0,\ve]$ is smooth with value $\ve$ on $[0,\ve]$
and $0$ on $[\frac{5\ve}{2}, \infty)$ and, for some $\sigma > 0$,
$- 1 + \sigma < \frac{\partial h}{\partial r}(\eps, r) \le 0$
for any $r \ge 0$ and $\eps$ sufficiently small.

\smallskip
2. In view of \eqref{asisedefine}, it suffices to show that $t(\ve,x)$ is Lipschitz, with Lipschitz constant uniformly bounded in $\ve$.
Define
\begin{equation}
\label{completo}
F(t,\ve, x):= \rho (S(t)x) -\rho(x) - h(\ve, \rho(x)).
\end{equation}
Then, from \eqref{smart},
\begin{equation}
\label{takederivative}
 F(t(\ve,x),\ve,x)=0.
\end{equation}
We take derivatives in \eqref{takederivative} with respect to $x_i$
to obtain $\frac{\partial F}{\partial t}\frac{\partial t}{\partial x_i} + \frac{\partial F}{\partial x_i}=0$;
that is, $\frac{\partial t}{\partial x_i}= - \frac{\partial F}{\partial x_i} \big(\frac{\partial F}{\partial t}\big)^{-1}$.
We need to show that $|\frac{\partial F}{\partial x_i}|$ is bounded from above.
By definition \eqref{completo}, it follows that
\begin{equation} \label{partial F partial x}
\frac{\partial F}{\partial x_i}
= \nabla \rho(S(t)x) \cdot \frac{\partial(S(t)x)}{\partial x_{i}} - \frac{\partial \rho}{\partial x_i}
  -\frac{\partial h}{\partial r}(\ve, \rho(x))\frac{\partial \rho}{\partial x_i}.
\end{equation}
Notice that $|\frac{\partial h}{\partial r}(\ve, \rho(x))| < 1$
by the definition of $h$, and $|\nabla (S(t)x)| \le M = M_{\eps_{0}}$ uniformly
on $x \in \overline{U} \setminus U^{3 \eps_{0}, \rho}$ for some $\eps_{0} > 0$,
since it is the flow of the smooth compactly supported vector field $\gamma G$;
see \eqref{flow good directions}.
Thus, it suffices to show that $|\nabla \rho (x)|$ is bounded above.
Relation \eqref{ladefinicion2} implies
\begin{equation}
\label{monica}
\frac{\partial \rho}{\partial x_i}
= \frac{\partial G}{\partial x_i}+ \frac{\partial G}{\partial \tau}\frac{\partial \rho}{\partial x_i} \
\,\,\, \Longrightarrow \,\,\,
\frac{\partial \rho}{\partial x_i} = \frac{ \frac{\partial G}{\partial x_i}} {1 - \frac{\partial G}{\partial \tau}}
\,\,\, \Longrightarrow \,\,\,
\nabla \rho(x) = \frac{ \nabla G} {1 - \frac{\partial G}{\partial \tau}}.
\end{equation}
From \eqref{ladefinicion}, we obtain
\begin{align*}
\nabla G(x, \tau) & = \int_{B(0, 1)} \nabla d(x - \frac{\tau}{2}z ) \eta(z) \, \dr z, \\
\frac{\partial G}{\partial \tau}(x, \tau)
& = - \frac{1}{2}\int_{B(0, 1)} \nabla d (x - \frac{\tau}{2}z) z \eta(z) \, \dr z.
\end{align*}
In turn, it follows
that
\begin{equation} \label{nabla rho estimate}
\left|\nabla G (x,\tau) \right| \leq 1, \quad  \Big|\frac{\partial G}{\partial \tau} (x,\tau) \Big| \leq \frac{1}{2}
\qquad\,\, \mbox{for any $x \in \R^{n}$ and $\tau = \rho(x)$},
\end{equation}
so that
\begin{equation} \label{nabla rho estimate-2}
|\nabla \rho(x)| \leq 2 \qquad \mbox{for any $x\in \R^n$},
\end{equation}
since $|\nabla d| = 1$ $\Leb{n}$--{\it a.e.}
by Lemma \ref{properties distance}.
Therefore, from \eqref{partial F partial x} and \eqref{nabla rho estimate},
we have
\begin{equation*}
|\nabla F| \le 2 (M + 2).
\end{equation*}

It remains to show that $|\frac{\partial F}{\partial t}|$ is bounded away from zero.
From \eqref{completo}, it follows by \cite[Remark 3.1]{BallZarnescu} that
\begin{equation*}
\frac{\partial F}{\partial t}
= \gamma(x(t))\,(\nabla \rho \cdot G)(S(t)x) > 0
\end{equation*}
for any $t$ small enough and $x$ in a suitable neighborhood of $\partial U$.

However, in order to prove the uniformity of this estimate in $x \in \overline{U} \setminus U^{3 \eps_{0}, \rho}$,
independent of $\eps$,
we need to employ the compactness of $\partial U$.
As recalled in \cite[Remark 5.3]{BallZarnescu},
in the case that $\partial U$ is Lipschitz,
Lemma \ref{Lieb Lip boundary estimates} can be applied.
In this way,  since $G(x)$ is a field of good directions,
we use \eqref{below estimate normal der int} to
show that there exists $\delta>0$ such that, for any $x \in U$ with $d(x) < \delta$,
\begin{equation} \label{below estimate rho G}
\nabla \rho(x) \cdot G(x) \ge \frac{2}{3 \sqrt{1 + L^{2}}},
\end{equation}
where $L$ is the maximal Lipschitz constant of the Lipschitz parametrization of $\partial U$.
This implies that there exists $\eps_{0} = \eps_{0}(\delta) > 0$
such that \eqref{below estimate rho G} holds for any $x \in U \setminus U^{3 \eps, \rho}$
and any $\eps \in (0, \eps_{0})$.
Since $\gamma$ can be chosen in such a way that $\gamma \equiv 1$
in $\overline{U} \setminus U^{3 \eps_{0}, \rho}$ (see \cite[\S 4 and Theorem 5.1]{BallZarnescu}),
we obtain
\begin{equation*}
\frac{\partial F}{\partial t} \ge \frac{2}{3 \sqrt{1 + L^{2}}}.
\end{equation*}
Thus, this implies
\begin{equation}
 \label{estimate}
 |\nabla t(\eps, x)| \le 3 (M + 2) \sqrt{1 + L^{2}}
 \qquad \mbox{for any $(\eps, x) \in (0, \eps_{0}) \times (U \setminus U^{3 \eps_{0}, \rho})$}.
 \end{equation}
From the proof of \cite[Theorem 5.1]{BallZarnescu},
we also know that $t \in C^{\infty} ((0, \eps_{0}) \times U) \cap C^{0}( [0, \eps_{0}) \times \overline{U})$
and $t(\eps, x) = 0$ for any $x \in U^{3 \eps, \rho}$.
From \eqref{estimate}, it follows that $f(\eps, \cdot) \in W^{1,\infty}(U; \R^{n})$.
Since $\partial U$ is only Lipschitz, the classical extension theorems
for Sobolev mappings ({\it cf.} \cite[Theorem 1, \S 4.4]{eg} and \cite[Theorem 1, \S 5.4]{eg2})
do not apply,
and we know only that $f(\eps, \cdot) \in \Lip_{\rm loc}(U; \R^{n})$
({\it cf.} \cite[Theorem 5, \S 4.2.3]{eg} and \cite[Theorem 4, \S 5.8.2]{eg2}).
However, Lemma \ref{extension_thm_Sobolev_domain} can be applied  in the case $p = \infty$.

Therefore, there exists a function $\bar{f}(\eps, \cdot) \in W^{1,\infty}(\R^n; \R^{n})$,
uniformly in $\eps \in [0, \eps_{0})$,
such that $f(\eps, \cdot) = \bar{f}(\eps, \cdot)$ on $U$.
Moreover, $f(\eps, \cdot)$ and $\bar{f}(\eps, \cdot)$ also agree on $\partial U$
since, for any $x \in \partial U$, there exists a sequence $\{x_j\} \subset U$ with $x_j \to x$, so that $f(\eps, x)=\bar{f}(\eps, x)$
because both $f(\eps, \cdot)$ and $\bar{f}(\eps, \cdot)$ are continuous in $\overline{U}$.
Clearly, $\bar{f}(\eps, \cdot)$ is Lipschitz in $\R^n$ uniformly in $\eps \in [0, \eps_{0})$
so that $f(\eps, \cdot) \in \Lip(\overline{U}; \R^{n})$ uniformly in $\eps \in [0, \eps_{0})$.

\smallskip
3. Let us now consider the inverse map $f^{-1}(\eps, \cdot) : \overline{U}^{\eps, \rho} \to \overline{U}$.
From the proof of \cite[Theorem 5.1]{BallZarnescu},
we know that it can be defined as
\begin{equation*}
f^{-1}(\eps, z) :=
\begin{cases}
S(\beta(\eps, z))z &\quad \text{if} \ z \in \overline{U}^{\eps, \rho} \setminus U^{3 \eps, \rho}, \\
z &\quad \text{if} \ z \in U^{3 \eps, \rho},
\end{cases}
\end{equation*}
where $\beta(\eps, z)$ is the unique solution to the equation:
\begin{equation*}
g(\eps, z, \beta(\eps, z)) = 0,
\end{equation*}
where
\begin{equation*}
g(\eps, z, \tau) := \rho(z) - \rho(S(\tau)z) - h(\eps, \rho(S(\tau)z)).
\end{equation*}
By the implicit function theorem,
$\beta(\eps, \cdot) \in C^{\infty}(U^{\eps, \rho}) \cap C^{0}(\overline{U}^{\eps, \rho})$.
It is also clear from the definition of $h$
that $\beta(\eps, z) = 0$ for any $z \in U^{3 \eps, \rho}$.
In addition, we have
\begin{equation*}
\nabla \beta(\eps, z) = - \frac{\nabla g (\eps, z, \beta(\eps, z))}{\frac{\partial g}{\partial \tau}(\eps, z, \beta(\eps, z))},
\end{equation*}
and
\begin{align}
\nabla g (\eps, z, \tau) & = \nabla \rho(z) - \nabla (S(\tau)z) \cdot \nabla \rho (S(\tau)z)\nonumber\\
  &\quad - \frac{\partial h}{\partial r}(\eps, \rho(S(\tau)z)) \nabla (S(\tau) z) \cdot \nabla \rho(S(\tau) z),
  \label{nabla beta num} \\
\frac{\partial g}{\partial \tau}(\eps, z, \tau)
& = - \nabla \rho (S(\tau)z) \cdot  (\gamma G)(S(\tau)z)\big( 1 + \frac{\partial h}{\partial r}(\eps, \rho(S(\tau)z))\big).
\label{nabla beta den}
\end{align}
Arguing as above and using \eqref{nabla rho estimate}, we obtain
\begin{equation*}
|\nabla g (\eps, z, \beta(\eps, z))| \le 2 (M_{\eps_{0}} + 2)
\end{equation*}
for any $\eps \in [0, \eps_{0})$ with any fixed $\eps_{0} > 0$,
and $z \in \overline{U}^{\eps, \rho}$.

As for the estimate on the denominator, we can use \eqref{below estimate normal der int} and
the fact that $\displaystyle \frac{\partial h}{\partial r}(\eps, r) \ge - 1 + \sigma > - 1$,
for some $\sigma > 0$ independent on $\eps$,
in order to show that there exists $\eps_{0} = \eps_{0}(\delta) > 0$ such that
\begin{equation*}
\left| \frac{\partial g}{\partial \tau}(\eps, z, \beta(\eps, z)) \right|
> \frac{2}{3 \sqrt{1 + L^{2}}} \sigma
\qquad\mbox{for any $\eps \in [0, \eps_{0})$ and $x \in U^{\eps, \rho} \setminus U^{3 \eps, \rho}$.}
\end{equation*}

Therefore,
for any $\eps \in [0, \eps_{0})$ and $x \in U^{\eps, \rho} \setminus U^{3 \eps, \rho}$,
\begin{equation*}
|\nabla \beta(\eps, z)| \le \frac{3}{\sigma} (M + 2) \sqrt{1 + L^{2}}
\end{equation*}
for some $\eta$ depending on the choice of $h$, $M = M_{\eps_{0}}$ depending on domain $U$,
and $L$ depending on the Lipschitz parametrization of $\partial U$.

This implies that $f^{-1}(\eps, \cdot) \in W^{1, \infty}(U^{\eps, \rho}; \R^{n})$ uniformly in $\eps \in (0, \eps_{0})$.
Thus, arguing as before, there exists an extension $\bar{f}^{-1}(\eps, \cdot) \in W^{1, \infty}(\R^{n}; \R^{n})$
uniformly in $\eps \in (0, \eps_{0})$, which coincides with $f^{-1}(\eps, \cdot)$ on $\overline{U}^{\eps, \rho}$
by the uniform continuity, and it is Lipschitz uniformly in $\eps \in (0, \eps_{0})$.
Thus, we have proved that $f^{-1}(\eps, \cdot) \in \Lip(\overline{U}^{\eps, \rho}; \R^{n})$ uniformly in $\eps \in (0, \eps_{0})$;
that is, the Lipschitz constant of  $f^{-1}(\eps, \cdot)$ on $\overline{U}^{\eps, \rho}$ is uniformly bounded in $\eps \in (0, \eps_0)$.

\smallskip
4. Let now $\eps < 0$. By the proof of \cite[Theorem 5.1]{BallZarnescu}, we know that,
for $\eps \in (- \eps_{0}, 0)$ and $x \in \overline{U}$,  $f$ is defined as
\begin{equation*}
f(\eps, x):= g(- \eps, \eps, f(- \eps, x)),
\end{equation*}
where $g(-\eps, \eps, \cdot)$ is the diffeomorphism introduced in Theorem \ref{Ball} (iii),
and $f$ is the map defined in the first part of the proof;
see also definition (5.16) in the proof of \cite[Theorem 5.1]{BallZarnescu}.

Since
$f(- \eps, \cdot)$ is uniformly Lipschitz as proved in the first part of the proof,
we need to check the same property for $g(- \eps, \eps, \cdot)$.
In order to do so, we need to introduce some auxiliary functions from the proof of \cite[Theorem 5.1]{BallZarnescu}.

Let $\theta = \theta(\tau, y)$ be the unique solution to the equation: $\rho(S(\theta)y) = \tau$.
By \cite[Lemma 4.1]{BallZarnescu}, such a function is well defined for
$|\tau| < 3 \eps_{0}$ and $|\rho(y)| < 3 \eps_{0}$, is smooth for $\tau \neq 0$, and
\begin{equation} \label{gradient theta}
\nabla \theta(\tau, y) = - \frac{\nabla (S(\theta)y) \cdot \nabla \rho(S(\theta)y)}{\nabla \rho(S(\theta)y) \cdot (\gamma G)(S(\theta)y)}(\tau, y)
\end{equation}
by the implicit function theorem,
since $\displaystyle \partial_\theta\big(S(\theta)y\big) = (\gamma G)(S(\theta)y)$.
This implies that, using \eqref{below estimate normal der int} if $0 < |d(S(\theta)y)| < \delta$
for some $\delta > 0$ sufficiently small,
the denominator of $\nabla \theta(\tau, y)$ is bounded away from zero uniformly in $\tau$.
Arguing in a similar way as above, we can show that there exists $\eps_{0} = \eps_{0}(\delta) > 0$ such that
\begin{equation} \label{gradient theta bound}
|\nabla \theta(\tau, y)| \le 3 M_{\eps_{0}} \sqrt{1 + L^{2}}
\end{equation}
for any $\tau \in (- 3 \eps_{0}, 3 \eps_{0})$, $\tau \neq 0$, and $y \in U_{3 \eps_{0}, \rho} \setminus \overline{U}^{3 \eps_{0}, \rho}$,
since $d(S(\theta(\tau, y))y)=0$ if and only if $\rho(S(\theta(\tau, y))y)=0$, which means that $\tau = 0$.
On the other hand, for our purposes, $\tau = 0$ if and only if $\eps = 0$; in such a case, $f(0, x) = x$ for any $x \in \R^{n}$,
which is clearly Lipschitz.

Furthermore, $\theta(\tau, y)$ is increasing in $\tau$ for fixed $y$, and
\begin{equation*}
\frac{\partial \theta}{\partial \tau}(\tau, y) = \frac{1}{(\nabla \rho \cdot (\gamma G))(S(\theta(\tau, y))y)}
> \frac{3}{2} \sqrt{1 + L^{2}},
\end{equation*}
which is uniformly strictly positive in $U_{3 \eps_{0}, \rho} \setminus \overline{U}^{3 \eps_{0}, \rho}$.
Given $\tilde{\eps} \in (0, \eps_{0})$ and $\eps' \in (- \eps_{0}, 0)$, we obtain
\begin{align}
q(\tilde{\eps}, \eps', y)
&:= -\frac{\theta(2 \eps', y)}{\t(2 \tilde{\eps}, y) - \t(2 \eps', y)},\label{p definition} \\
 r(\tilde{\eps}, \eps', y)
&:= \frac{\theta(\tilde{\eps}, y) - \theta(2 \eps', y)}{\t(2 \tilde{\eps}, y) - \t(2 \eps', y)},\label{r definition}\\
 s(\tilde{\eps}, \eps', y)
&:= \frac{\theta(\eps', y) - \theta(2 \eps', y)}{\t(2 \tilde{\eps}, y) - \t(2 \eps', y)}.\label{s definition}
\end{align}
By the monotonicity property of $\t$, it is clear that $r, s \in (0, 1)$.

Let $u : (0, 1) \times (0, 1) \times \R \to \R$ be the smooth function described in \cite[Lemma 5.1]{BallZarnescu},
which particularly satisfies the property that $u(a, b, c) = c$ for any $(a, b) \in (0, 1) \times (0, 1)$ and $c \notin [0, 1]$.

Finally, we define $g$ by
\begin{equation} \label{g definition}
g(\tilde{\eps}, \eps', y) :=
\begin{cases}
S(w(\tilde{\eps}, \eps', y))y &\quad  \text{if} \ \rho(y) \in (3 \eps', 3 \tilde{\eps}), \\
y &\quad \text{otherwise},
\end{cases}
\end{equation}
where
\begin{equation} \label{eta definition}
w(\tilde{\eps}, \eps', y)
:= (\t(2 \tilde{\eps}, y) - \t(2 \eps', y)) u(r(\tilde{\eps}, \eps', y), s(\tilde{\eps}, \eps', y), q(\tilde{\eps}, \eps', y)) + \t(2 \eps', y).
\end{equation}

We observe that, if $y \in U$,
then $\t(2 \eps', y) > 0$ so that $q(\tilde{\eps}, \eps', y) > 0$.
Moreover, if $q(\tilde{\eps}, \eps', y) > 1$,
then $w(\tilde{\eps}, \eps', y) = 0$ by the property of $u$.
Hence, we can restrict ourselves to the case that $0 < q(\tilde{\eps}, \eps', y) \le 1$, without loss of generality.

In our case, we are dealing with $g(- \eps, \eps, y)$ for $\eps < 0$,
and $y = f(- \eps, x) \in U^{- \eps, \rho}$, so that we can select $\tilde{\eps} = - \eps$ and $\eps' = \eps$.

Since $S$ is the flow of the smooth vector field $\gamma G$,
it suffices
to show uniform bounds on the gradient of $w$ for $y \in U \setminus \overline{U}^{3 \eps_{0}, \rho}$.
We have
\begin{align}
&\nabla w(- \eps, \eps, y) \nonumber\\
&=  \nabla \t(2 \eps, y) + \big(\nabla \t(- 2 \eps, y) - \nabla \t(2 \eps, y)\big) u(r(- \eps, \eps, y), s(-\eps, \eps, y), \beta(-\eps, \eps, y))\nonumber \\
&\,\,\,\,+ \big(\t(- 2 \eps, y) - \t(2 \eps, y)\big) \Big ( \frac{\partial u}{\partial r}(r(- \eps, \eps, y), s(- \eps, \eps, y), q(- \eps, \eps, y)) \nabla r(- \eps, \eps, y)\nonumber \\
&\quad \qquad\qquad \qquad\qquad\qquad\, + \frac{\partial u}{\partial s}(r(- \eps, \eps, y), s(- \eps, \eps, y), q(- \eps, \eps, y)) \nabla s(- \eps, \eps, y)\nonumber \\
&\quad \qquad\qquad \qquad\qquad\qquad\, + \frac{\partial u}{\partial q}(r(- \eps, \eps, y), s(- \eps, \eps, y), q(- \eps, \eps, y)) \nabla q(- \eps, \eps, y) \Big ).
\label{nabla eta}
\end{align}

Observe that $h$ is smooth, and $r(- \eps, \eps, y), s(- \eps, \eps, y) \in (0, 1)$ for any $\eps \in (- \eps_{0}, 0)$
and $y \in U \setminus \overline{U}^{3 \eps_{0}, \rho}$ by the properties of $\t$.
In addition, only the intersection with set $\{ (\eps, y) : 0 < q(- \eps, \eps, y) \le 1 \}$ is relevant to us, since $w$ vanishes on the outside of the intersection.
Therefore, $h$ and all its derivatives are uniformly bounded in $U \setminus \overline{U}^{3 \eps_{0}, \rho}$ for any $\eps \in (- \eps_{0}, 0)$.

Moreover, by standard calculations, we have
\begin{align*}
&\big(\t(- 2 \eps, y) - \t(2 \eps, y)\big) \nabla q(- \eps, \eps, y )\nonumber\\
&\quad\, = - \nabla \t(2 \eps, y) - q(- \eps, \eps, y) \big(\nabla \t(- 2 \eps, y) - \nabla \t(2 \eps, y)\big), \\[1.5mm]
&\big(\t(- 2 \eps, y) - \t(2 \eps, y)\big) \nabla r(- \eps, \eps, y )\\
&\quad\, = \nabla \t(-\eps, y) - \nabla \t(2 \eps, y) - r(- \eps, \eps, y) \big( \nabla \t(- 2 \eps, y) - \nabla \t(2 \eps, y)\big), \\[1.5mm]
& \big(\t(- 2 \eps, y) - \t(2 \eps, y)\big) \nabla s(- \eps, \eps, y )\\
&\quad\, = \nabla \t(\eps, y) - \nabla \t(2 \eps, y)
 - s(- \eps, \eps, y) \big( \nabla \t(- 2 \eps, y) - \nabla \t(2 \eps, y)\big).
\end{align*}
From these formulas, the bounds on $(q, r, s)$, and \eqref{gradient theta bound},
we conclude that $\nabla w(- \eps, \eps, \cdot) \in L^{\infty}(U \setminus U^{3 \eps_{0}, \rho}; \R^{n})$
uniformly in $\eps \in (- \eps_{0}, 0)$.

Arguing now as in the previous two cases,
we can extend $f(\eps, \cdot)$ for $\eps < 0$ to a $W^{1, \infty}$--map
on the whole $\R^{n}$, whose restriction on $\overline{U}$ coincides with $f(\eps, \cdot)$;
thus proving that $f(\eps, \cdot) \in \Lip(\overline{U}; \R^{n})$ uniformly in $\eps \in (- \eps_{0}, 0]$.

\smallskip
5. Finally, the inverse map for $\eps < 0$ is given by
\begin{equation} \label{inverse eps < 0}
f^{-1}(\eps, x) := f^{-1}(- \eps, g^{-1}(- \eps, \eps, x))
\qquad \mbox{for $x \in \overline{U}_{\eps, \rho}$ and $\eps \in (- \eps_{0}, 0)$},
\end{equation}
for some $\eps_{0} > 0$
sufficiently small ({\it cf.} definition (5.17) in the proof of \cite[Theorem 5.1]{BallZarnescu}).

The inverse map $g^{-1}(- \eps, \eps, \cdot)$ is defined in a similar way to \eqref{g definition}, by using $h^{-1}(a, b, \cdot)$ instead of $h$;
that is, the inverse function of $h(a, b, \cdot)$.
Since $h^{-1}(a, b, d) = d$ for any $(a, b) \in (0, 1) \times (0, 1)$ and $d \notin [0, 1]$,
then we can argue as before to obtain the uniform essential boundedness of $\nabla g^{-1}(- \eps, \eps, \cdot)$,
which concludes that  $f(\eps, \cdot)$ for $\eps < 0$ is a uniform bi-Lipschitz function.
\end{proof}

\begin{remark} \label{convergence to identity nabla}
\, As a consequence of Theorem {\rm \ref{bi Lip no eps}},
we can show that, if $U$ has Lipschitz boundary,
then $\nabla f(\eps, \cdot) \to {\rm I}_{n}$ in $L^{p}(\R^{n}; \R^{n \times n})$ for any $1 \le p < \infty$.

Indeed, by Theorem {\rm \ref{Ball} (ii)},
$\nabla f(\eps, x) = {\rm I}_{n}$ for any $x$ such that $|\rho(x)| > 3 |\eps|$.
This implies that $\nabla f(\eps, x) \to {\rm I}_{n}$ for any $x \in \R^{n} \setminus \partial U$, and
\begin{equation}\label{8.24a}
\int_{\R^{n}} |\nabla f(\eps, x) - {\rm I}_{n}|^{p} \, \dr x \le C^{p} \Leb{n}(\{ |\rho(x)| \le 3 |\eps| \}),
\end{equation}
where $C$ is a constant depending only on $U$ and $n$,
since the Lipschitz constants of $f(\eps, \cdot)$ are uniformly bounded for $\eps \in (- \eps_{0}, \eps_{0})$,
by Theorem {\rm \ref{bi Lip no eps}}.
This implies the convergence, since $\Leb{n}(\partial U) = 0$.
\end{remark}

As an immediate consequence, we have the following result.

\begin{theorem} \label{Lip is Lip deformable}
The boundary of any Lipschitz domain is Lipschitz deformable in the sense of Definition {\rm \ref{aquiaqui}}.
\end{theorem}

\begin{proof}
\, Indeed, we can employ Theorem \ref{bi Lip no eps} to construct a Lipschitz deformation $\Psi$ as in Definition \ref{aquiaqui}.
It suffices to set
\begin{equation*}
\Psi(x, \tau) := f( \tau \eps_{1}, x) \qquad \mbox{for any $0 < \eps_{1} < \eps_{0}$},
\end{equation*}
where $f$ is given in Theorem \ref{Ball}.
By the properties of $f(\eps, \cdot)$,
$\Psi(\cdot, \tau)$ is a bi-Lipschitz homeomorphism over its image uniformly
in $\tau\in [0, 1]$ and $\Psi(\cdot, 0) = {\rm Id}$.
\end{proof}

\begin{remark}
\, In fact, Definition {\rm \ref{aquiaqui}} refers to open sets with Lipschitz boundary, while,
thanks to Theorem {\rm \ref{bi Lip no eps}}, we are able to deal with open bounded Lipschitz domains.
However, the connectedness assumption is not relevant,
since one can work separately with each connected component of a bounded open set with Lipschitz boundary
to achieve Theorems {\rm \ref{bi Lip no eps}} and {\rm \ref{Lip is Lip deformable}} for each component.
In a similar way, one can also consider an unbounded open set with Lipschitz boundary $U$,
and then localize the problem by considering, for instance, $U \cap B(0, R)$ for $R > 0$,
which are
open bounded sets with Lipschitz boundary.
It is then clear that Theorems {\rm \ref{bi Lip no eps}} and {\rm \ref{Lip is Lip deformable}}
apply to $U \cap B(0, R)$ for any $R > 0$.
Thus, we can conclude that any open set with Lipschitz boundary has a regular Lipschitz deformable boundary, at least locally.
\end{remark}

An immediate consequence of the existence of such Lipschitz diffeomorphism between $\partial U$ and
$\partial U^{\eps, \rho}$ or $\partial U_{\eps, \rho}$
is that the area formula can be employed in order to consider only integrals on $\partial U$.

\begin{theorem} \label{area formula U eps}
Let $U \Subset \Omega$ be an open set with Lipschitz boundary, let $\FF \in \DM^p(\Omega)$ for $1 \le p \le \infty$,
and let $\phi \in C^{0}(\Omega)$ with $\nabla \phi \in L^{p'}(\Omega; \R^{n})$.
Then there exists a set $\mathcal{N} \subset \R$ with $\mathcal{L}^1(\mathcal{N})=0$ such that,
for every nonnegative sequence $\{\ve_k\}$ satisfying $\ve_k \notin \mathcal{N}$ for any $k$ and $\ve_k \to 0$,
\begin{align} \label{G-G area eps int}
& \int_{U} \phi \, \dd \div \FF + \int_{U} \FF \cdot \nabla \phi \, \dr x\nonumber\\
& = - \lim_{k \to \infty} \int_{\partial U}
\Big ( \phi \FF \cdot \frac{\nabla \rho}{|\nabla \rho|} \Big) (f(\eps_{k}, x)) J^{\partial U} f(\eps_{k}, x) \, \dr \Haus{n - 1},
\end{align}
and
\begin{align} \label{G-G area eps ext}
&\int_{\overline{U}} \phi \, \dd \div \FF + \int_{U} \FF \cdot \nabla \phi \, \dr x\nonumber\\
& = - \lim_{k \to \infty}  \int_{\partial U}
\Big(\phi \FF \cdot \frac{\nabla \rho}{|\nabla \rho|} \Big) (f(- \eps_{k}, x)) J^{\partial U} f(- \eps_{k}, x) \, \dr \Haus{n - 1},
\end{align}
where $f(\pm \eps, \cdot)$ is the bi-Lipschitz diffeomorphism introduced in Theorem {\rm \ref{Ball}}.

In addition, \eqref{G-G area eps int} holds also for any bounded open set $U$ with Lipschitz boundary if $\phi \in L^{\infty}(\Omega)$,
and even for an unbounded open set $U$ with Lipschitz boundary if $\mathrm{supp}(\phi) \cap U^{\delta} \Subset \Omega$ for any $\delta > 0$.
Similarly, \eqref{G-G area eps ext} also holds  for any open set $U$ satisfying $\overline{U} \subset \Omega$,
provided that $\mathrm{supp}(\phi)$ is compact in $\Omega$.
\end{theorem}

\begin{proof}
\, We need to apply the area formula to the Lipschitz maps $f(\eps, \cdot) : \partial U \to \partial U^{\ve, \rho}$
in
\eqref{G-G int C0}--\eqref{G-G ext C0}.

We denote by $J^{\partial U} f(\eps, \cdot)$ the $(n - 1)$--dimensional Jacobian of $f(\eps, \cdot)$ on $\partial U$,
and recall that the inner unit normal to $\partial U^{\eps, \rho}$
is given by $\frac{\nabla \rho}{|\nabla \rho|}$ from Theorem \ref{interior normal trace smooth}.
Then
\begin{equation*}
\int_{\partial U^{\ve_{k}, \rho}} \Big( \phi \FF \cdot \frac{\nabla \rho}{|\nabla \rho|}\Big) (x) \, \dr \Haus{n - 1}
= \int_{\partial U} \Big(\phi \FF \cdot \frac{\nabla \rho}{|\nabla \rho|} \Big) (f(\eps_{k}, x)) J^{\partial U} f(\eps_{k}, x) \, \dr \Haus{n - 1}.
\end{equation*}
We can argue analogously with $\partial U_{\eps_{k}, \rho}$.
Therefore, we can rewrite \eqref{G-G int C0}--\eqref{G-G ext C0} as \eqref{G-G area eps int}--\eqref{G-G area eps ext}.
\end{proof}

From this result, we can deduce  some known facts again from the theory of $\DM$--fields.

\begin{corollary} \label{G-G F continuous U Lip}
Let $\FF \in \DM^{p}(\Omega) \cap C^{0}(\Omega; \R^{n})$ for $1 \le p \le \infty$,
let $\phi \in C^{0}(\Omega) \cap L^{\infty}(\Omega)$ with $\nabla \phi \in L^{p'}(\Omega; \R^{n})$,
and let $U \subset \Omega$ be a bounded Lipschitz domain.
Then
\begin{equation} \label{G-G F continuous}
\ban{\FF \cdot \nu, \phi}_{\partial U} =  \int_{U} \phi \, d \div \FF + \int_{U} \FF \cdot \nabla \phi \, \dr x
= - \int_{\partial U} \phi \FF \cdot \nu_{U} \, \dr \Haus{n - 1}.
\end{equation}
If $\FF \in \DM^{\infty}(\Omega)$, then the normal trace functional on $\partial U$ is indeed a Radon measure,
absolutely continuous with respect to $\Haus{n - 1} \res \partial U$,
with essentially bounded density function $-\mathfrak{F}_{\ii} \cdot \nu_{U} \in L^{\infty}(\partial U; \Haus{n - 1})$.
\end{corollary}

\begin{proof} \, In order to prove the first statement, we consider $\psi \in C^{1}_{c}(\R^{n}; \R^{n})$.

Notice that
\begin{equation*}
\int_{U} \div\,\psi \, \dr x = - \int_{\partial U} \psi \cdot \nu_{U} \, \dr \Haus{n - 1},\quad
\int_{U^{\eps, \rho}} \div\,\psi \, \dr x
= - \int_{\partial U^{\eps, \rho}} \psi \cdot \frac{\nabla \rho}{|\nabla \rho|} \, \dr \Haus{n - 1}.
\end{equation*}
Arguing as in the proof of Theorem \ref{area formula U eps}, we obtain
\begin{equation*}
\int_{\partial U^{\eps, \rho}} \psi \cdot \frac{\nabla \rho}{|\nabla \rho|} \, \dr \Haus{n - 1}
= \int_{\partial U} (\psi \circ f)(\eps, \cdot) \cdot \Big(\frac{\nabla \rho}{|\nabla \rho|} \circ f\Big)(\eps, \cdot) \, J^{\partial U} f(\eps, \cdot) \, \dr \Haus{n - 1}.
\end{equation*}
Since
\begin{equation*}
\lim_{\eps \to 0}  \int_{U^{\eps, \rho}} \div\,\psi \, \dr x = \int_{U} \div\,\psi \, \dr x,
\end{equation*}
we conclude
\begin{equation} \label{conv boundary psi}
\lim_{\eps \to 0}
\int_{\partial U} (\psi \circ f)(\eps, \cdot) \cdot \Big(\frac{\nabla \rho}{|\nabla \rho|} \circ f\Big)(\eps, \cdot) \, J^{\partial U} f(\eps, \cdot) \, \dr \Haus{n - 1}
= \int_{\partial U} \psi \cdot \nu_{U} \, \dr \Haus{n - 1}
\end{equation}
for any $\psi \in C^{1}_{c}(\R^{n}; \R^{n})$.
By the density of $C^{1}_{c}(\R^{n}; \R^{n})$ in $C_{c}(\R^{n}; \R^{n})$ with respect to the supremum norm,
we can deduce that \eqref{conv boundary psi} holds also for any $\psi \in C_{c}(\R^{n}; \R^{n})$.
Thus, by \eqref{G-G area eps int}, we conclude that \eqref{G-G F continuous} holds.

As for the second part of the statement, we can argue as in the proof of \cite[Theorem 2.2]{CF1},
since $U$ has a Lipschitz deformable boundary, by Theorem \ref{Lip is Lip deformable}.
\end{proof}

\begin{remark}
\, Corollary {\rm \ref{G-G F continuous U Lip}} can also be regarded as a consequence of
Proposition {\rm \ref{normal_trace_classical_repr_cont}}, together with the well-known fact that $\Haus{n - 1}(\partial U \setminus \redb U) = 0$
for any open set $U$ with Lipschitz boundary.
In addition, this implies that, in the case that $\partial U$ is Lipschitz regular,
$\redb U$ can be substituted with $\partial U$ in \eqref{first_Green_id_exact}--\eqref{second_Green_id_exact}.
\end{remark}

We end this section by recalling an alternative result concerning the approximation of open bounded sets
with Lipschitz boundary which has been proved by Ne\v{c}as in \cite{nevcas1962}.
For this exposition, we refer mostly to the paper of  Verchota \cite{Verchota_1984},
in which the result in \cite{nevcas1962} is extended and applied.

\begin{definition}
We denote by $Z(P, r)$ the truncated cylinder centered at point $P$ and with basis radius $r$.
Given a Lipschitz domain $U$ and a point $P \in \partial U$,
we say that $Z(P, r)$ is a coordinate cylinder if
\begin{enumerate}
\item[\rm (i)] The bases of $Z(P, r)$ have a positive distance from $\partial U${\rm ;}
\item[\rm (ii)]  There exists a coordinate system $(\hat{x}_{n}, x_{n})$ such that $\{ \hat{x}_{n} = 0 \}$ is the axis
  of $Z(P, r)$, and there exists a Lipschitz function $\varphi = \varphi_{Z} : \R^{n - 1} \to \R$ such that
$$
Z(P, r) \cap U = Z(P, r) \cap \{ (\hat{x}_{n}, x_{n}) : x_{n} > \varphi(\hat{x}_{n}) \};
$$
\item[\rm (iii)]  $P = (0, \varphi(0))$ or, equivalently, $P$ is the origin of the coordinate system and $\varphi(0) = 0$.
\end{enumerate}
The pair $(Z, \varphi)$ is called a coordinate pair.
\end{definition}

\begin{remark}
\, If the Lipschitz domain $U$ is bounded, then $\partial U$ can be covered by a finite number of
coordinate cylinders $\{ Z_{j} \}_{j = 1}^{N}$,
to which corresponds a finite number of coordinate pairs.
In addition, cylinders $Z_{j}$ can be selected
in such a way that some dilation $Z_{j}^{*} = \lambda_{j} Z_{j}$, $\lambda_{j} > 1$,
still gives a coordinate pair $(Z_{j}^{*}, \varphi_{j})$.
We denote by $L$ the maximum of the Lipschitz constants of functions $\varphi_{j}$.
Also we may assume that $\varphi_{j} \in \Lip_{c}(\R^{n - 1})$ without loss of generality.
\end{remark}

\begin{remark}
\, Given $\varphi \in \Lip_{c}(\R^{n})$,
there exists a sequence $\psi_{k} \in C^{\infty}_{c}(\R^{n})$ such that $\psi_{k} \to \psi$ uniformly,
$\|\nabla \psi_{k}\|_{L^{\infty}(\R^{n}; \R^{n})} \le \|\nabla \varphi \|_{L^{\infty}(\R^{n}; \R^{n})}$,
and $\nabla \psi_{k} \to \nabla \varphi$ in $L^{q}(\R^{n}; \R^{n})$ for any $1 \le q < \infty$.
This can be achieved by taking the convolution of $\varphi$ with a smooth mollifier.
\end{remark}

The following approximation results hold, for which we
refer to \cite[Theorem 1.1]{nevcas1962}, \cite[Lemma 1.1]{nevcas1964equations}, \cite[Appendix]{verchota1982layer},
and \cite[Theorem 1.12]{Verchota_1984}; see also the alternative proof given in \cite{doktor1976approximation}.
For the self-containedness,
we also give here a sketch of the proof.

\begin{proposition} \label{Nevcas_approx}
Let $U$ be a bounded Lipschitz domain. Then the following statements hold{\rm :}
\begin{enumerate}
\item[\rm (i)] There exists a sequence of open sets $U_{k}$ satisfying that $\partial U_{k}$ are of class $C^{\infty}$,
     $U_{k} \Subset U_{k + 1} \Subset U$, and $\bigcup_{k} U_{k} = U${\rm ;}
\item[\rm (ii)] There exists a covering of $\partial U$ by coordinate cylinders such that,
  for any coordinate pair $(Z, \varphi)$ with $\varphi \in \Lip_{c}(\R^{n - 1})$, $Z^{*} \cap \partial U_{k}$ for each $k$ is the graph of a function
  $\varphi_{k} \in C^{\infty}_{c}(\R^{n - 1})$ satisfying that $\varphi_{k} \to \varphi$ uniformly,
  $\|\nabla \varphi_{k} \|_{L^{\infty}(\R^{n - 1}; \R^{n - 1})} \le \| \nabla \varphi \|_{L^{\infty}(\R^{n - 1}; \R^{n - 1})}$,
  and $\nabla \varphi_{k} \to \nabla \varphi$ $\Leb{n - 1}$--a.e.
   and in $L^{q}(\R^{n - 1}; \R^{n - 1})$ for any $1 \le q < \infty${\rm ;}
\item[\rm (iii)]  There exists a sequence of Lipschitz diffeomorphisms $f_{k} : \R^{n} \to \R^{n}$ such that
  $f_{k}(\partial U) = \partial U_{k}$, the Lipschitz constants are uniformly bounded in $k$,
  $f_{k} \to {\rm Id}$ uniformly on $\partial U$, and $\nabla f_{k} \to {\rm I}_{n}$ for $\Haus{n - 1}$--{\it a.e.} $x \in \partial U${\rm ;}
\item[\rm (iv)] There exists a sequence of nonnegative functions $\omega_{k} = J^{\partial U} f_{k}$ uniformly bounded
  and bounded away from zero such that $(f_{k}^{-1})_{\#} (\Haus{n - 1} \res \partial U_{k}) = \omega_{k} \Haus{n - 1} \res \partial U${\rm :}
$$
\Haus{n - 1}(\partial U_{k} \cap f_{k}(E)) = \int_{E} \omega_{k} \, \dr \Haus{n - 1}
\qquad \mbox{for any Borel set $E \subset \partial U$},
$$
and that
$\omega_{k} \to 1$ $\Haus{n - 1}$--{\it a.e.} on $\partial U$ and in $L^{q}(\partial U; \Haus{n - 1})$ for any $1 \le q < \infty${\rm ;}
\item[\rm (v)] The normal vector to $U_{k}$ satisfies that $\nu_{U_{k}} \circ f_{k} \to \nu_{U}$ for $\Haus{n - 1}$--{\it a.e.}  $x \in \partial U$
   and in $L^{q}(\partial U; \Haus{n - 1})$ for any $1 \le q < \infty${\rm ,} and an analogous statement holds for the tangent vectors{\rm ;}
\item[\rm (vi)]  There exists a $C^{\infty}$ vector field $H$ in $\R^{n}$ such that
$$
H(f_{k}(P)) \cdot \nu_{U_{k}}(f_{k}(P)) \ge C > 0 \qquad\mbox{for any $P \in \partial U$},
$$
where $C = C(H, L)$, and $L$ is the maximal Lipschitz constant of the parametrization of $\partial U$.
\end{enumerate}
\end{proposition}

\smallskip
{\em Sketch of Proof}.
Results (i)--(ii) have been proved by Ne\v{c}as in \cite{nevcas1962} (see also \cite[Appendix]{verchota1982layer}); while the others follows from the first two.

Indeed, we can define the homeomorphisms $f_{k}$ in each coordinate cylinder $Z_{j}$ by
$$
f_{k}(x) = (\hat{x}_{n}, x_{n} + \varphi_{k}(\hat{x}_{n}) - \varphi(\hat{x}_{n}))
$$
for the coordinate system $(\hat{x}_{n}, x_{n})$ related to $Z_{j}$,
and then glue these definitions together with the aid of some cutoff functions,
by using the fact that the same coordinate pair can also be used in the larger cylinder $Z_{j}^{*}$.
In this way, the uniform convergence follows immediately.

As for result (iv), we can find that
$$
f_{k}^{-1}(y) = (\hat{y}_{n}, y_{n} - \varphi_{k}(\hat{y}_{n}) + \varphi(\hat{y}_{n})),
$$
where $(\hat{y}_{n}, y_{n})$ is the coordinate system related to some $Z_{j}$.
This also shows that $f_{k}$ is invertible with continuous inverse, so that it is indeed a homeomorphism.
In fact, since $\varphi \in \Lip_{c}(\R^{n-1})$ and $\varphi_{k} \in C^{\infty}_{c}(\R^{n-1})$,
we can conclude that $f_{k}$ is a Lipschitz diffeomorphism, with Lipschitz constants uniformly bounded in $k$,
by using that $\nabla \varphi_{k} \to \nabla \varphi$ for $\Leb{n - 1}$--{\it a.e.} $\hat{x} \in \R^{n - 1}$.

Moreover, by employing the area formula,
it follows that $\omega_{k}$ is exactly the $(n - 1)$--dimensional Jacobian of $f_{k}$
on $\partial U$, $J^{\partial U} f_{k}$.
Notice that
$$
\nabla f_{k}
= \left ( \begin{array}{c|c}
{\rm I}_{n - 1} & 0 \\
\hline\\[-3.5mm]
(\nabla (\varphi_{k} - \varphi))^{\top} & 1
\end{array} \right ),
$$
where ${\rm I}_{n - 1}$ is the $(n - 1) \times (n - 1)$ identity matrix.
Therefore, the convergence of $\nabla \varphi_{k}(x)$ to $\nabla \varphi(x)$ for $\Leb{n - 1} = \Haus{n - 1}$--{\it a.e.} $x \in \partial U$
implies that $\nabla f_{k} \to {\rm I}_{n}$ for $\Haus{n - 1}$--{\it a.e.} $x \in \partial U$,
which in turn implies  that $J^{\partial U} f_{k} \to 1$ for $\Haus{n - 1}$--{\it a.e.} $x \in \partial U$.
Then the $L^{q}$--convergence follows by the Lebesgue theorem and the boundedness properties,
which can be shown by calculating  the Jacobian explicitly.
\qed

\medskip
Proposition \ref{Nevcas_approx} allows us to refine Theorem \ref{Lip is Lip deformable},
by showing that any open bounded set with Lipschitz boundary admits a {\em regular} Lipschitz deformation
in the sense of Definition \ref{aquiaqui}.
Analogously, if the set is unbounded, such a statement should hold locally.

\begin{theorem} \label{Lip_reg_def}
If $U$ is a bounded open set with Lipschitz boundary in $\R^{n}$,
then there exists a regular Lipschitz deformation $\Psi(x, \tau) = \Psi_{\tau}(x)$ of $\partial U$ satisfying
\begin{equation} \label{regular_deformation_cor}
\lim_{\tau \to 0^{+}} J^{\partial U} \Psi_{\tau} = 1 \qquad \text{in $L^{1}(\partial U; \Haus{n - 1})$}.
\end{equation}
\end{theorem}

\begin{proof}
\, Set
\begin{equation*}
\Psi(x, \tau) := \big(k + 1 - k(k + 1) \tau\big) f_{k + 1}(x) + \big(k (k + 1)\tau - k\big) f_{k}(x) \quad
\text{if  $ \tau \in (\frac{1}{k + 1}, \frac{1}{k}]$},
\end{equation*}
where functions $f_{k}$ are given by Proposition \ref{Nevcas_approx}.
It is clear that $\Psi(\cdot, \tau)$ is a bi-Lipschitz diffeomorphism from $U$ over its image,
by Proposition \ref{Nevcas_approx}(iii), with Lipschitz constants uniformly bounded in $\tau > 0$.
Since $J^{\partial U} f_{k} \to 1$ in $L^{q}(\partial U; \Haus{n - 1})$
for any $1 \le q < \infty$,
by Proposition \ref{Nevcas_approx} (iv), and
\begin{align*}
0 \le k (k + 1) \tau - k  \le 1, \quad 0 \le k + 1 - k (k + 1)\tau  \le 1 \qquad\,\,\mbox{for $\tau \in \Big ( \frac{1}{k + 1}, \frac{1}{k} \Big ]$,}
\end{align*}
we conclude that $J^{\partial U} \Psi(x, \tau) \to 1$ in $L^{q}(\partial U; \Haus{n - 1})$ for $1 \le q < \infty$,
which implies \eqref{regular_deformation_cor}.
\end{proof}

\begin{remark}
\, Hofmann-Mitrea-Taylor in \cite[Proposition 4.19]{HofmannMitreaTaylor}
worked with a strongly Lipschitz domain $U$ in $\R^{n}$ such that there exists a $C^{1}$--vector field $h$ satisfying
\begin{equation*}
|h(x)| = 1, \quad h(x) \cdot \nu_{U}(x) \ge \kappa \qquad\,\, \text{for} \ \Haus{n - 1}\text{--{\it a.e.}} \ x \in \partial U
\end{equation*}
for some $\kappa \in (0, 1)$.
In the literature, a domain is said to be strongly Lipschitz if the Lipschitz constants of the parametrization
of $\partial U$ are uniformly bounded, so that any open bounded set $U$ with Lipschitz boundary is a strongly Lipschitz domain,
by compactness. For a more detailed exposition, we refer to \cite[Appendix B]{auscher2005hardy}.
Then, if $U_{\tau} := \{ x - \tau h(x) : x \in U \}$, there exists $\tau_{0} > 0$ such that,
for any $\tau \in (0, \tau_{0})$, $U_{\tau}$ is a strongly Lipschitz domain satisfying $\overline{U_{\tau}} \subset U$
and $\partial U_{\tau} = \{ x - \tau h(x) : x \in \partial U \}$.
In addition, the results of Proposition {\rm \ref{Nevcas_approx}(ii)--(vi)} hold with Lipschitz regularity,
instead of $C^{\infty}$.
However, it is clear that more regularity on the vector field $h$ would imply more regularity of $\partial U_{\tau}$.

Moreover, a similar approximation holds from the exterior of $U$,
if we consider $U_{- \tau}$ for $\tau \in (- \tau_{1}, 0]$, for some $\tau_{1} > 0$.

Following Theorem {\rm \ref{Lip_reg_def}},
we see that the assumptions of Hofmann-Mitrea-Taylor \cite[Proposition 4.19]{HofmannMitreaTaylor}
are not strictly necessary{\rm ;} however, they allow to have this particular representation of the approximating sets.
\end{remark}

\section{\, Cauchy Fluxes and Divergence-Measure Fields}

In Continuum Physics, the fundamental principle of balance law
can be stated in the most general terms ({\it cf.} Dafermos \cite{dafermos2010hyperbolic} and Lax \cite{Lax1973}):
{\it A balance law in an open set $\Omega$ of $\R^n$ postulates that
the production of a vector-valued ``extensive" quantity in any bounded open
subset $U\Subset \Omega$ is balanced by the Cauchy flux
of this quantity through the boundary of $U$.}

For smooth continuum media,
the physical principle of balance law can be formulated in the classical form:
\begin{equation}
\int_{U} b(y)\, \dr y = \int_{\partial U} f(y)\, \dr \Haus{n-1}(y)
\label{balance}
\end{equation}
for any given open set $U$ that is of smooth boundary,
where $f$ is a density function of the Cauchy flux, and $b$ is a production density function.
In
mechanics, $f$ represents the surface force per unit area on
$\partial U$, while  $f$ gives the heat flow per
unit area across the boundary $\partial E$ in thermodynamics.

In 1823, Cauchy \cite{Cauchy1} (also see \cite{Cauchy2}) established
the {\it stress theorem},
which states that, if $f(y):=f(y,\nu(y))$, defined for each $y$
in an open region $\Omega$ and every unit vector $\nu$, is
continuous in $y$, and $b(y)$ is uniformly bounded on $\Omega$, and
if (\ref{balance}) is satisfied for every smooth region
$U\Subset\Omega$, then $f(y,\nu)$ must be linear in $\nu$; that is,
there exists a vector field $\FF$
such that
$$
f(y,\nu)=\FF(y) \cdot \nu.
$$
The
{\it Cauchy postulate} states that the density flux $f$ through a
surface depends on the surface solely through the normal at that
point.
Since the time of Cauchy's stress result \cite{Cauchy1,Cauchy2},
many efforts have been made to generalize his ideas and remove some
of his hypotheses. The first results in this direction were obtained
by Noll \cite{Noll} in 1959, who set up a basis for an axiomatic
foundation for continuum thermodynamics. In particular, Noll
\cite{Noll} showed that the Cauchy postulate may directly follow
from the balance law. In \cite{gm}, Gurtin-Martins introduced the
concept of Cauchy flux and removed the continuity assumption on $f$.
They represented the Cauchy flux as an additive mapping $\mathcal{F}$
on surfaces $S$ such that there exists a constant $C>0$ so that
\begin{equation}
\label{original}
|\mathcal{F}(S)| \leq C \Haus{n - 1}(S), \quad |\mathcal{F}(\partial B)| \leq C \Leb{n}(B)
\end{equation}
for any surface $S$ and subbody $B$.

In 1983, Ziemer \cite{Ziemer1} proved Noll's theorem in the context of
geometric measure theory, in which the Cauchy fluxes were first
formulated via employing sets $E$ of finite perimeter to represent the bodies
and $\partial^*E$ to represent the surfaces.
His formulation of the balance law for
the flux function yields the existence of a vector field $\FF\in L^\infty$
with $\div \FF\in L^\infty$.
The papers by \v{S}ilhav\'y \cite{S2,S3}
extended definition \eqref{original} by requiring
\begin{equation}
\label{original2}
|\mathcal{F}(S)| \leq \int_{S}h \,\dr \Haus{n - 1}, \quad |\mathcal{F}(\partial B)| \leq  \int_{B} g \, d x
\end{equation}
for suitable functions $g$ and $h$ in $L^p$ for $p \geq 1$ and almost every surface $S$.
The vector fields obtained under these conditions  have distributional divergences
that are integrable; that is, $\FF \in L^1$ and $\div \FF \in L^{1}$.
However, all the previous formulations of Cauchy fluxes
do not allow the presence of ``shock waves'' since $\div \FF$ is absolutely continuous with respect to $\mathcal{L}^n$.

Degiovanni-Marzocchi-Musesti \cite{degiovanni1999cauchy}
further generalized conditions \eqref{original} and considered the Cauchy fluxes defined on {\it almost} every surface and satisfying
\begin{equation}
\label{original3}
  |\mathcal{F}(S)|  \leq  \int_{S} h \, \dr \Haus{n - 1}, \quad |\mathcal{F}(\partial B)| \leq \sigma(B)
\end{equation}
for a suitable function $h \in L^1_{\rm loc}$ and a nonnegative Radon measure $\sigma$.
This definition of Cauchy fluxes induced the existence of a vector field $\FF \in \DM^{1}_{\rm loc}$.
Schuricht \cite{Sch} studied an alternative formulation to \eqref{original},
which consists in considering the contact interactions $f$ as maps on pairs of disjoint subbodies (instead of surfaces).
Thus, $f(B,A)$ is the resultant force exerted on $B$ by $A$.
The function $f$ is assumed to be countable additive in the first argument ({\it i.e.}, a measure)
and finitely additive with respect to the second argument.
This alternative formulation also implies the existence of $\FF \in \DM^{1}_{\rm loc}$,
depending on $A$, such that $\div \FF = f(\cdot, A)$.
The Gauss-Green formulas obtained in \cite{degiovanni1999cauchy} and \cite{Sch}
are valid for $\FF \in \DM^{p}(\Omega)$ for any $p \ge 1$, but only on the sets of finite perimeter,
$E \subset \Omega$, which lie in a suitable subalgebra related to the particular representative of $\FF$.
In other words, these Gauss-Green formulas are valid only on {\it almost every} set,
thus missing  the exceptional surfaces or ``shock waves''.
In order to recover the flux on {\it every surface}, it is necessary to develop
a theory of normal traces for divergence-measure fields.

In Chen-Torres-Ziemer \cite{ctz}, such a theory of normal traces
on reduced boundaries of sets of finite perimeter has been established
for $\DM^\infty(\Omega)$--fields.
The method in  \cite{ctz} consists in constructing the normal trace as the limit
of the classical normal traces over smooth approximations of the set of finite perimeter.
This approach requires a new approximation theorem of sets of perimeter
that can distinguish between the measure-theoretic interior and exterior
of the set.
The Cauchy flux introduced in  \cite{ctz} is defined on {\it every} set
of finite perimeter, $E\Subset \Omega$, and on {\it every} $\Haus{n - 1}$--rectifiable
surface $S \subset \redb E$ (so that $S$ is oriented with the normal to the set).
The conditions that there exists a nonnegative Radon measure $\sigma$ such that
\begin{equation} \label{Cauchy_flux_cond_CTZ}
|\mathcal{F}(\partial^* E)|\le \sigma(E^1), \qquad\,\, |\mathcal{F}(S)|\le C \Haus{n-1}(S)
\end{equation}
imply the existence of a $\DM^\infty$--field $\FF$ so that the Cauchy flux
over every surface can be recovered through the normal traces of $\FF$ on the oriented
surface (see Remark \ref{Special-Case}(ii) below).

The Cauchy fluxes in the sense of Chen-Torres-Ziemer  \cite{ctz} allow the presence of
exceptional surfaces ({\it i.e.} shock waves)
in the formulation of the axioms.
In this setting, ${\rm div} \FF=g\sigma$ for some function $g$ and Radon measure $\sigma$,
and hence ${\rm div} \FF$ is not in general
absolutely continuous with respect to $\mathcal{L}^n$.
The measure $\sigma$ does not vanish on the exceptional surfaces
and the Cauchy flux $\mathcal{F}$ has a discontinuity since
$\mathcal{F}(S)\ne -\mathcal{F}(-S)$.
Formulations \eqref{original}--\eqref{original2} deal with the particular case $\sigma=\mathcal{L}^n$,
and hence the measure vanishes on
any $\Haus{n-1}$--dimensional surface, excluding the shock waves.

\begin{example} \label{Cauchy_flux_ex_1}
Let
$$
\FF(x_1, x_2) = f(x_2) g(x_1) (1, 0) \,\qquad \mbox{for some $f \in L^{\infty}(\R)$ and $g \in C^{1}_{c}(\R)$},
$$
so that $\FF \in \DM^{\infty}(\R^{2})$,
and let $E$ be as in Remark {\rm \ref{rem:3.22}}.
Then, by \eqref{divergence_chi_E_F_example},
$\GG = \chi_{E} \FF$ is also in $\DM^{\infty}(\R^{2})$,
while $\GG \notin BV_{\rm loc}(\R^{2}; \R^{2})$, since $E$ is not a set of locally finite perimeter.

If $U:= (0, 2)^{2} = U^{1}$, then $E \cap U = E$ and $\chi_{U} \GG = \chi_{E} \FF$.
Hence, applying \eqref{divergence_chi_E_F_example} again, we obtain
\begin{align} \label{normal_trace_ex_1_eq}
&\ban{\GG \cdot \nu, \cdot}_{\partial U}\nonumber\\
& = \chi_{U} \div\, \GG - \div(\chi_{U} \GG) = (\chi_{U^{1}} - 1) \div\, \GG
 = (\chi_{U^{1}} - 1) f(x_{2}) g(x_{1}) D_{x_{1}} \chi_{E}\nonumber \\
& = - f(x_{2}) g(x_{1}) \Haus{1} \res \big(\{0\} \times (0, 1)\big)
 + f(x_{2}) g(x_{1}) \Haus{1} \res \big(\{2\} \times (\frac{1}{2}, \frac{3}{4})\big).
\end{align}
On the other hand, if $L := U \cup \redb U$ instead, we still obtain that
$\chi_{L} \GG = \chi_{E} \FF = \GG$ and
\begin{align} \label{normal_trace_ex_2_eq}
\ban{\GG \cdot \nu, \cdot}_{\partial L}
& = \chi_{L} \div\, \GG - \div(\chi_{L} \GG) = (\chi_{L} - 1) \div\, \GG \nonumber\\
& = (\chi_{U^{1} \cup \redb U}(x_{1}, x_{2}) - 1) f(x_{2}) g(x_{1}) D_{x_{1}} \chi_{E} = 0,
\end{align}
since $\mathrm{supp}(|D_{x_{1}} \chi_{E}|) \subset U^{1} \cup \redb U$.
Hence, it follows that
\begin{equation} \label{jump_normal_trace_ex_eq}
\ban{\GG \cdot \nu, \cdot}_{\partial U} \neq \ban{\GG \cdot \nu, \cdot}_{\partial(U \cup \redb U)}
\end{equation}
in general.
This condition is satisfied, for instance, if $g(0) < 0$, $g(2) = 0$, and $f > 0$ in $(0, 1)$,
since $\ban{\GG \cdot \nu, \cdot}_{\partial U}$ is a nontrivial nonnegative Radon measure in this case.

Thanks to the fact that $\GG \in \DM^{\infty}(\R^{2})$, we can define an associated Cauchy flux $\mathcal{F}$
by using the theory developed in \cite{ctz}.
Given a bounded set of finite perimeter $M$, there exist the interior and exterior normal
traces of $\GG$: $(\mathfrak{G}_{\ii} \cdot \nu_{M})$ and $(\mathfrak{G}_{\ee} \cdot \nu_{M}) \in L^{\infty}(\redb M; \Haus{1})$
{\rm (}see Proposition {\rm \ref{normal trace p infty}}{\rm )}.
Then we define
\begin{equation} \label{flux_ex_def_1}
\mathcal{F}(S) := - \int_{S} \mathfrak{G}_{\ii} \cdot \nu_{M} \, \dr \Haus{1},
\end{equation}
if $S$ is an $\Haus{1}$--rectifiable surface such that $S \subset \redb M$ which is oriented by $\nu_{M}$,
for some bounded set of finite perimeter $M${\rm ;} and
\begin{equation} \label{flux_ex_def_2}
\mathcal{F}(S) := - \int_{S} \mathfrak{G}_{\ee} \cdot \nu_{M} \, \dr \Haus{1},
\end{equation}
if $S$ is an $\Haus{1}$--rectifiable surface such that $S \subset \redb M$ which is oriented by $- \nu_{M}$,
for some bounded set of finite perimeter $M$.

It is not difficult to check that $\mathcal{F}$ is a Cauchy flux in the sense of \cite{ctz}.
Indeed, by definition, $\mathcal{F}$ is a finitely additive functional on disjoint surfaces.
Since the normal traces are essentially bounded, from \eqref{flux_ex_def_1}--\eqref{flux_ex_def_2},
we obtain
\begin{equation*}
|\mathcal{F}(S)| \le C \Haus{1}(S).
\end{equation*}
Then  $|\mathcal{F}(\redb M)| \le \sigma(M^{1})$ for any bounded set $M$ of finite perimeter,
if $\sigma = |\div \,\GG|$ is chosen.
Indeed, we need just to employ the Gauss-Green formulas and the fact that $\GG$ has compact support.

In particular, if we apply \eqref{G-G phi Sobolev int} to $\GG$,
a bounded set of finite perimeter $M$ and $\phi \in \Lip_{c}(\R^{2})$ with $\phi \equiv 1$ on $\overline{M}$,
then
\begin{equation*}
\mathcal{F}(\redb M) = - \int_{\redb M} \mathfrak{G}_{\ii} \cdot \nu_{M} \, {\rm d} \Haus{1} = \div\,\GG (M^{1}).
\end{equation*}
Arguing analogously, from \eqref{G-G phi Sobolev ext}, we have
\begin{equation*} \label{eq:GG_flux_2}
\mathcal{F}(- \redb M) = - \int_{\redb M} \mathfrak{G}_{\ee} \cdot \nu_{M} \, {\rm d} \Haus{1} = \div\,\GG (M^{1} \cup \redb M).
\end{equation*}
Since $\GG$ has compact support in $\R^{2}$, by \cite[Lemma 3.1]{comi2017locally},
we obtain
\begin{equation*}
0 = \div\,\GG (\R^{2}) = \div\,\GG (M^{1} \cup \redb M) + \div\,\GG ((\R^{2} \setminus M)^{1}),
\end{equation*}
from which it follows that
\begin{equation*}
\mathcal{F}(\redb (\R^{2} \setminus M)) = \mathcal{F}(- \redb M) = - \div\,\GG ((\R^{2} \setminus M)^{1}).
\end{equation*}
Thus, \eqref{Cauchy_flux_cond_CTZ} is satisfied. Then
we have proved that $\mathcal{F}$ is a Cauchy flux in the sense of \cite{ctz}.

Choose $S := \redb U$ oriented by $\nu_{U}$, by \eqref{flux_ex_def_2},
Proposition {\rm \ref{normal trace p infty}},  and \eqref{normal_trace_ex_1_eq},
we have
\begin{equation*}
\mathcal{F}(\redb U) = \ban{\GG \cdot \nu, \phi}_{\partial U} = - \int_{\redb U} \phi(x_{1}, x_{2}) f(x_{2}) g(x_{1}) \, {\rm d} D_{x_{1}} \chi_{E}
\end{equation*}
for any $\phi \in \Lip_{c}(\R^{2})$ with $\phi \equiv 1$ on $\overline{U}$.
Arguing analogously, by \eqref{flux_ex_def_2}, Proposition {\rm \ref{normal trace p infty}}, and \eqref{normal_trace_ex_2_eq},
we have
\begin{equation*}
\mathcal{F}(- \redb U) = \ban{\GG \cdot \nu, \phi}_{\partial(U \cup \redb U)} = 0,
\end{equation*}
for any $\phi \in \Lip_{c}(\R^{2})$ with $\phi \equiv 1$ on $\overline{U}$.
Then, by \eqref{jump_normal_trace_ex_eq}, it follows that, in general,
\begin{equation*}
\mathcal{F}(\redb U) \neq  - \mathcal{F}( - \redb U),
\end{equation*}
which shows that $\mathcal{F}$ may have a discontinuity on the rectifiable surface $\redb U$.
\end{example}

In this paper above, we have developed a more general theory of normal traces
for unbounded $\DM^p$ fields.
In particular, we have shown that the normal trace can be represented as the
limit of the classical normal traces on smooth approximations or deformations.
Hence, we can now give a more general definition of Cauchy fluxes over general open sets
(not necessarily of finite perimeter).

\begin{definition}[Side surfaces] \label{surface}
A side surface in $\Omega$ is a pair $(S, U)$ so that $S\Subset \Omega$ is a Borel set and $U\Subset \Omega$ is a open set
such that $S\subset \partial U$. The side surface $(S, U)$ is often written as $S$ for simplicity,
when no confusion arises from the context.
\end{definition}

\begin{definition}[Cauchy fluxes] \label{general}
\label{def5}
Let $\Omega$ be a bounded open set.
A Cauchy flux is a functional $\mathcal{F}$ defined on the side surfaces $(S, U)$
such that the following properties hold{\rm :}
\begin{enumerate}
\item[\rm (i)] $\Fc(S_1 \cup S_2) = \Fc(S_1)+ \Fc(S_2)$ for any pair of
      disjoint side surfaces $S_1$ and $S_2$ in $\partial U$, for some $U \Subset \Omega${\rm ;}

\item[\rm (ii)]\label{tres}  There exists a nonnegative Radon measure $\sigma$ in $\Omega$
such that
$$
|\Fc(\partial U)| \leq \sigma (U) \qquad\,\,\mbox{for every open set  $U \Subset\Omega$};
$$

\item[\rm (iii)]  There exists a nonnegative Borel function $h \in L^{1}_{\rm loc}(\Omega)$ such that
      $$
      |\Fc(S)| \leq \int_{S} h \, \dr \Haus{n - 1}
      $$
for any side surface $S \subset \partial U$ and any open set $U \Subset \Omega$ {\rm (}the integral
could be $\infty$, in which case the axiom is also true{\rm )}.
   \label{dos}
\end{enumerate}
For simplicity, the Cauchy flux is often written as $\mathcal{F}(S)$ as (i)--(iii) above,
when no confusion arises from the context.
\end{definition}

We state now our main result on the representation of general Cauchy fluxes.

\begin{theorem}\label{Cauchy_flux_p}
Let $\F$ be a  Cauchy flux in $\Omega$ with $h \in L^{1}_{\rm loc}(\Omega)$
as in Definition {\rm \ref{general}}.
Then there exists a unique
$\FF \in \mathcal{DM}^{1}_{\rm loc}(\Omega)$ such that, for every open set $U \Subset \Omega$,

\begin{enumerate}
\item[\rm (i)] For any $\phi \in C^{1}_{c}(\Omega)$ such that $\phi \equiv 1$ on a neighborhood of $\partial U$,
\begin{equation} \label{main_representation_flux}
\Fc (\partial U) = \ban{\FF\cdot\nu,\, \phi}_{\partial U},
\end{equation}
and there exists an interior smooth approximation $U^{\eps}$ of $U$ as
in Theorem {\rm \ref{interior trace smooth general}} such that, for a suitable subsequence $\eps_{k} \to 0$ as $k\to\infty$,
\begin{equation*}
\Fc (\partial U) = - \lim_{\epsilon_{k} \to 0} \int_{\partial U^{\eps_{k}}}\FF \cdot \nu_{U^{\eps_{k}}} \, \dr \Haus{n - 1}
\end{equation*}
where $\FF \cdot \nu_{U_{\eps_{k}}}$ denotes the classical dot product{\rm ;}

\smallskip
\item[\rm (ii)]
If $\chi_{U} \FF \in \DM^{1}_{\rm loc}(\Omega)$, then there exists $\mu_b \in \mathcal{M}(\partial U)$ such that
\begin{equation*}
\Fc (\partial U) =  \int_{\partial U }  \dr \mu_b;
\end{equation*}

\item[\rm (iii)] If $U$ is a $C^0$ domain, then there exists a sequence of smooth set $U^{\eps, \rho}$
as in Theorem {\rm \ref{interior normal trace smooth}}, which can be represented as
a deformation generated by $f(\eps,x)$ {\rm (}defined in Theorem {\rm \ref{Ball}}{\rm )}
that is $C^\infty$ in $x$ when $\eps>0$ and $C^0$ in $x$ when $\eps=0$,
such that, for a suitable subsequence $\eps_{k} \to 0$,
\begin{equation}
\nonumber
\Fc (\partial U)= - \lim_{\eps_{k} \to 0} \int_{\partial U^{\eps_{k}, \rho}} \FF \cdot \nu_{U^{\eps_{k}, \rho}} \, \dr \Haus{n - 1} ;
\end{equation}

\item[\rm (iv)] If $U$ has a Lipschitz boundary, then there exists a regular
Lipschitz deformation $\Psi(x,\eps)=:\Psi_\eps(x)$ of $\partial U$
such that, for a suitable subsequence $\eps_{k} \to 0$ as $k\to\infty$,
\begin{equation*}
\Fc (\partial U)
= - \lim_{\eps_{k} \to 0} \int_{\partial U }\FF(\Psi_{\eps_{k}}(x))\cdot
\nu_{U} (\Psi_{\eps_{k}}(x))J\Psi_{\eps_{k}}(x) \dr \Haus{n - 1} (x).
\end{equation*}
\end{enumerate}
\end{theorem}

\begin{proof} \, We divide the proof into four steps.

\smallskip
1. We first show the existence of such an $\FF\in \DM^{1}_{\rm loc}(\Omega)$.
Let $\mathcal{I}_{\Omega}$ be the collection of all closed cubes in $\R^n$ of the form:
$$
I=[a_1,b_1]\times \cdots \times [a_n, b_n],
$$
such that $I \Subset \Omega$.
For almost every $s\in [a_j,b_j]$, define
$$
I_{j,s}:=\{y\in I \, :\, y_j=s\}.
$$

Let $\{e_1, \dots, e_n\}$ be the canonical basis of $\R^n$. We fix $j\in \{1, \dots, n\}$.
For every cube $I \in \mathcal{I}_{\Omega}$, define
$$
\mu^j(I):=\int_{a_j}^{b_j}\mathcal{F}(I_{j,s}) \, \dr s.
$$
From Definition \ref{general}(iii), we have
\begin{align}
|\mu^j(I)|
\le \int_{a_j}^{b_j} |\mathcal{F}(I_{j,s})|\, \dr s
\le \int_{a_j}^{b_j} \int_{I_{j,s}} h\, d\Haus{n-1} \dr s
\le \int_{I} |h|\, \dr x
=\|h\|_{L^1(I)}, \nonumber
\end{align}
where the Fubini theorem has been used.
Thus, for any finite collection of disjoint cubes $I_1, \dots, I_K$, we have
\begin{equation}\label{7.10a}
\sum_{i=1}^K|\mu^j(I_i)|\le \sum_{i=1}^K  \|h\|_{L^1(I_i)}\le  \|h\|_{L^1(\cup_{i=1}^K I_i)}.
\end{equation}

Since $h\in L^1_{\rm loc}(\Omega)$, then, for every $\ep>0$, there exists $\delta>0$ such that
$$
\Leb{n}
(A)<\delta \quad \Longrightarrow \quad \int_A|h|\, \dr x <\ep.
$$
Hence, if $\{I_i\}_{i=1}^K$ is a finite collection of disjoint cubes $I_1, \dots, I_k$,
satisfying
$\sum_{i=1}^K\mathcal{L}^n(I_i)=\mathcal{L}^n(\cup_{i=1}^K I_i)<\delta$, then
\begin{align}
\sum_{i=1}^K|\mu^j(I_i)|
\le \|h\|_{L^1(\cup_{i=1}^K I_i)}
<\ep.  \label{7.10b}
\end{align}
Hence, $\mu^j$ is an additive set function defined on $\mathcal{I}_{\Omega}$.  We can now apply a generalization
of Riesz's theorem, due to Fuglede \cite{Fuglede} (see also \cite[Theorem 9.5]{ctz}), to conclude
that there exists $f_j\in L^1_{\rm loc}(\Omega)$ such that
$$
\mu^j(I)=\int_I f_j\, \dr x \qquad \mbox{for every $I \in \mathcal{I}_{\Omega}$}.
$$
We take sequences $\alpha_{k,j} \uparrow s$ and $\beta_{k,j}\downarrow s$ as $k\to \infty$. We have
$$
\frac{1}{\beta_{k,j}-\alpha_{k,j}}\int_{\alpha_{k,j}}^{\beta_{k,j}}\mathcal{F}(I_{j,s})\, \dr s
=\frac{1}{\beta_{k,j}-\alpha_{k,j}}\int_{\alpha_{k,j}}^{\beta_{k,j}}\int_{I_{j,s}} f_{j} \, \dr x \dr s.
$$
Letting $k\to \infty$ yields
$$
\mathcal{F}(I_{j,s})=\int_{I_{j,s}} f_{j} \, \dr\Haus{n-1} \qquad \mbox{for $\Leb{1}$--{\it a.e.} $s$}.
$$
Define
$$
\FF:=(f_{1}, \dots, f_{n}).
$$
We obtain that, for every $j\in \{1, \dots, n\}$,
\begin{equation}\label{7.10c}
\mathcal{F} (I_{j,s})=-\int_{I_{j,s}}\FF(y)\cdot e_j\, \dr\Haus{n-1}(y)
\qquad \mbox{for $\Leb{1}$--{\it a.e.} $s$}.
\end{equation}
From this point on, we say that a statement holds for almost every cube if it holds for all cubes
whose side intervals
with endpoints in $\R \setminus \mathcal{N}$, for some $\Leb{1}$ negligible set $\mathcal{N}$.

From \eqref{7.10c}, it follows that, for almost every cube $I \in \mathcal{I}_{\Omega}$,
\begin{equation} \label{flux_intervals}
\mathcal{F}(\partial I)
= - \int_{\partial I} \FF(y)\cdot \nu_{I}(y)\, \dr\Haus{n-1}(y), \end{equation}
which, by Definition \ref{general}(ii), implies
\begin{equation}\label{7.10d}
\Big|\int_{\partial I} \FF(y)\cdot \nu_{I}(y)\, \dr \Haus{n-1}(y)\Big| = |\mathcal{F}(\partial I)|
\le \sigma(\mathring{I})\le \sigma(I),
\end{equation}
where $\mathring{I}$ denotes the open cube.

Using \eqref{7.10d}, we can now proceed as in \cite[Lemma 9.6]{ctz}, or use \cite[Theorem 5.3]{degiovanni1999cauchy},
to conclude that $\FF$ is a vector field with divergence measure
satisfying $|{\rm div} \FF|\le \sigma$, which means that $\FF \in \DM^{1}_{\rm loc}(\Omega)$.

\medskip
2. {\it Uniqueness of the $\DM^1$--field $\FF$}.
\smallskip
Assume now that there exists another vector field $\mathbf{G}=(g_1, \cdots, g_n)$ such that \eqref{7.10c} holds.
For fixed $j\in \{1, \cdots, n\}$, we obtain that, for any cube $I \in \mathcal{I}_{\Omega}$,
\begin{align*}
\int_I f_j \dr x =\int_{a_j}^{b_j}\int_{I_{s,j}}f_j \, d\Haus{n-1}(y) \dr s
=\int_{a_j}^{b_j}\int_{I_{s,j}}g_j \, \dr \Haus{n-1}(y) ds=\int_I g_j \dr x.
\end{align*}
Hence,
$f_j(x)=g_j(x)$ for $\mathcal{L}^n$--{\it a.e.} $x$.

\smallskip
3. In order to prove (i), we approximate $\partial U$ with closed cubes in such a way that
$$
\partial U=\bigcap_{i=1}^\infty J_i,
$$
where each $J_i$ is a finite union of closed cubes in $\mathcal{I}_{\Omega}$, which can be chosen so that
\eqref{flux_intervals} holds, and $J_{i+1}\subset J_i$. In addition, we can also choose $\{J_{i}\}$ in such a way that
\begin{equation} \label{normal_trace_J_i}
\ban{\FF \cdot \nu, \cdot}_{\partial \mathring{J}_i} = - \FF\cdot\nu_{J_i}\Haus{n-1}\res \partial J_i.
\end{equation}
This follows for instance from \cite[Theorem 7.2]{degiovanni1999cauchy},
which states that, for almost every closed cube $I \in \mathcal{I}_{\Omega}$, we have
\begin{equation} \label{Degiovanni_interval_div}
\div \FF (I) = - \int_{\partial I} \FF \cdot \nu_{I} \, \dr \Haus{n - 1}.
\end{equation}
On the other hand, for any $\varphi \in \Lip_{c}(\Omega)$, $\varphi \FF \in \DM^{1}(\Omega)$
by Proposition \ref{product rule p}. It is clear that, by \eqref{product rule},
\begin{equation*}
\div(\varphi \FF)(\mathring{I}) = \div(\varphi \FF)(I) - \int_{\partial I} \varphi \, \dd \div \FF.
\end{equation*}
By Remark \ref{normal_trace_functional_fract_Sobolev} and \eqref{Degiovanni_interval_div}, it follows that
\begin{equation*}
\ban{\FF \cdot \nu, \varphi}_{\partial \mathring{I}} = \div(\varphi \FF)(\mathring{I})
= - \int_{\partial I} \varphi \, \dd \div \FF - \int_{\partial I} \varphi \FF \cdot \nu_{I} \, \dr \Haus{n - 1}.
\end{equation*}
Then, arguing as in \cite[Example 1.63]{afp}, we obtain that $|\div \FF|(\partial I) = 0$ for almost every
$I \in \mathcal{I}_{\Omega}$, since $\div \FF$ is a Radon measure.
All in all, we conclude \eqref{normal_trace_J_i} for almost
every finite union $J$  of closed cubes; then the sequence $\{J_{i}\}$ is chosen from these finite unions $\{J\}$
of closed cubes.

Now, from Definition \ref{general}(ii), we have
$$
|\mathcal{F}(\partial (\mathring{J}_i\cap U))|\le \sigma (\mathring{J}_i\cap U).
$$
Standard measure theory arguments imply that
\begin{equation}\label{7.7a}
\lim_{i\to \infty} \sigma(\mathring{J}_i\cap U)=\sigma((\cap_{i} J_i)\cap U)
=\sigma(\partial U \cap U)=0,
\end{equation}
so that
\begin{equation}\label{7.8a}
\lim_{i\to \infty} \mathcal{F}(\partial(\mathring{J}_i \cap U)) =0.
\end{equation}

Now, we consider $\ban{\FF\cdot\nu, \phi}_{\partial (\mathring{J}_i\cap U)}$ for some $\phi \in C^{1}_{c}(\Omega)$.
We notice by Theorem \ref{interior normal trace} that, for any $\phi\in C_c^1(\Omega)$,
\begin{equation} \label{7.9a}
\ban{\FF\cdot\nu, \phi}_{\partial U} = - \lim_{k\to \infty} \int_{\redb U^{\eps_k}}\phi \FF\cdot\nu_{U^{\eps_k}}\, \dr\Haus{n-1}.
\end{equation}
In the same way, we have
\begin{equation}
\ban{\FF\cdot\nu, \phi}_{\partial(\mathring{J}_i\cap U)}
=-\lim_{k\to \infty}
\int_{\partial^* W^{\eps_k}}\phi \FF\cdot\nu_{W^{\eps_k}}\, \dr\Haus{n-1}
\qquad \mbox{for any $\phi\in C_c^1(\Omega)$},
\label{7.10ab}
\end{equation}
where $\partial W^{\eps_k}$ is defined as the superlevel set of
the signed distance function associated to $W = \mathring{J}_i\cap U$, as in \eqref{U^eps}.

We now choose a test function $\phi\in C_c^1(\Omega)$ such
that $\phi\ne 0$ in a neighborhood of $\partial U$,
but $\phi\equiv 0$ on $\Omega \setminus J_i$.
Since $J_i$ is closed, then $\phi=0$ in a neighborhood of $\partial J_i\cap U$.
With this choice of $\phi$, \eqref{7.10ab} reduces to
\begin{align}
\ban{\FF\cdot\nu, \phi}_{\partial (\mathring{J}_i\cap U)}
=-\lim_{k\to \infty}
\int_{\partial^* U^{\eps_k}} \chi_{\mathring{J}_i} \phi \FF\cdot\nu_{U^{\eps_k}}\, \dr\Haus{n-1}
\end{align}
and, from \eqref{7.9a} and the fact that $\mathring{J}_{i} \cap \redb U^{\eps_{k}} = \redb U^{\eps_{k}}$
for $\eps_{k}$ small enough and $i$ fixed, we obtain
\begin{equation} \label{normal_trace_res_partial_U}
\ban{\FF\cdot\nu, \phi}_{\partial(\mathring{J}_i\cap U)}
=\ban{\FF\cdot\nu, \phi}_{\partial U}
\qquad
\mbox{for any such $\phi$}.
\end{equation}
Therefore, the distribution
$\ban{\FF\cdot\nu, \cdot}_{\partial(\mathring{J}_i\cap U)}$ coincides
with $\ban{\FF\cdot\nu, \cdot}_{\partial U}$ on $\partial U$.

Arguing similarly, we can show
\begin{equation} \label{normal_trace_res_partial_J_i}
\ban{\FF\cdot\nu,\, \cdot}_{\partial(\mathring{J}_i\cap U)}
=\ban{\FF\cdot\nu,\, \cdot}_{\partial \mathring{J}_i}
\qquad \mbox{on $U\cap \partial J_i$}.
\end{equation}
Therefore, we conclude
\begin{align} \label{decomposition_flux_U_J}
\ban{\FF\cdot\nu, \cdot}_{\partial(\mathring{J}_i\cap U)}
&=\ban{\chi_{J_i}\FF\cdot\nu,\, \cdot}_{\partial U}
   +\ban{\chi_{U}\FF\cdot\nu,\, \cdot}_{\partial \mathring{J}_i}\\
&=\ban{\FF\cdot\nu,\, \cdot}_{\partial U}
   +\ban{\chi_{U}\FF\cdot\nu,\, \cdot}_{\partial \mathring{J}_i}, \nonumber
\end{align}
since $J_{i} \supset \partial U$. By \eqref{trace_div_meas_repr}, we have
\begin{equation} \label{trace_div_repr_U_J}
\langle \FF\cdot\nu, \cdot\rangle_{\partial(\mathring{J}_i\cap U)}
= \chi_{\mathring{J}_i\cap U} \div \FF - \div ( \chi_{\mathring{J}_i\cap U} \FF).
\end{equation}
If we now choose $\phi \in C^{1}_{c}(\Omega)$ such that $\phi \equiv 1$ on a neighborhood of $\partial U$,
then $\phi \equiv 1$ on $\partial U \cup \partial J_{i}$ for any $i$ large enough.
Then, from \eqref{normal_trace_J_i} and \eqref{normal_trace_res_partial_U}--\eqref{trace_div_repr_U_J},
we obtain
\begin{align*}
\ban{\FF\cdot\nu,\, \phi}_{\partial U} - \int_{U \cap \partial J_{i}} \FF \cdot \nu_{J_{i}} \, \dr \Haus{n - 1}
& = \ban{\FF\cdot\nu,\, \phi}_{\partial U} - \int_{\partial J_{i}} \phi \chi_{U} \FF \cdot \nu_{J_{i}} \, \dr \Haus{n - 1} \\
& = \ban{\FF\cdot\nu,\, \phi}_{\partial U} + \ban{\chi_{U}\FF\cdot\nu,\, \phi}_{\partial \mathring{J}_i} \\
& =  \ban{\FF\cdot\nu, \phi}_{\partial(\mathring{J}_i\cap U)} \\
& = \int_{\mathring{J}_i\cap U}\phi \, \dr \div \FF - \int_{\Omega} \phi \, \dr \div ( \chi_{\mathring{J}_i\cap U} \FF) \\
& = \int_{\mathring{J}_i\cap U}\phi \, \dr \div \FF  + \int_{J_{i} \cap U} \FF \cdot \nabla \phi \, \dr x \to 0
\end{align*}
as $i \to \infty$, since $|J_{i} \cap U| \to 0$ and $|\div \FF|(\mathring{J}_{i} \cap U) \to 0$. This implies
\begin{equation}\label{7.13a}
\ban{\FF\cdot\nu,\, \phi}_{\partial U} = \lim_{i\to \infty}  \int_{U\cap \partial J_i}\FF(y)\cdot \nu_{J_{i}}(y) \dr\Haus{n-1}(y).
\end{equation}

On the other hand, by Definition \ref{general}(i), we have
$$
\mathcal{F}(\partial(\mathring{J}_i\cap U))
=\mathcal{F}(\partial U)
+\mathcal{F} (U\cap \partial J_i),
$$
so that, from \eqref{flux_intervals} and \eqref{7.8a}, we obtain
\begin{equation}\label{7.12a}
\mathcal{F}(\partial U) = \lim_{i\to \infty} \int_{U\cap \partial J_i}\FF(y)\cdot \nu_{J_{i}} (y) \dr\Haus{n-1}(y).
\end{equation}
Finally, from \eqref{7.13a}--\eqref{7.12a}, we conclude \eqref{main_representation_flux}.
Then Theorem \ref{interior trace smooth general} implies
the second part of point (i)
(see also Theorem \ref{interior normal trace}).

\smallskip
4. For (ii), we notice that $\chi_{U} \FF \in \DM^{1}(\Omega)$, since $U \Subset \Omega$.
Hence, Theorem \ref{equivalence trace prod}
implies the existence of a finite Radon measure $\mu_b$ on $\partial U$ such that
$$
\ban{\FF\cdot\nu, \phi}_{\partial U}=\int_{\partial U} \phi \, \dr \mu_b
\qquad \mbox{for any $\phi \in C^{1}_{c}(\Omega)$}.
$$
Thus, if we take $\phi \equiv 1$ on a neighborhood of $\partial U$, result (i) immediately implies (ii).
Cases (iii) and (iv) follow analogously from (i): for (iii), we need to apply Theorem \ref{interior normal trace smooth};
while (iv) is obtained by employing the bi-Lipschitz regular deformation $\Psi_{\eps}(x)$ from
Theorem \ref{Lip_reg_def} as in Theorem \ref{area formula U eps}.
\end{proof}

\begin{remark}\label{Special-Case}
\, In particular, the following assertions also hold{\rm :}
\begin{enumerate}
\item[\rm (i)] If $U$ is an open set of finite perimeter, $\sigma (\partial U) =0$,
and $\int_{\partial*U} h \, \dr \Haus{n-1} < \infty$,  then
\begin{equation}
\nonumber
\Fc (\partial U) =  - \int_{\partial ^* U} \FF \cdot \nu_{U} \, \dr \Haus{n-1},
\end{equation}
where $\FF \cdot \nu_{U} $ is the classical dot product.

To see this, we use the assumption that $\sigma(\partial U)=0$ and $\int_{\partial^* U} h\, \dr\Haus{n-1}<\infty$
to apply \cite[Theorem 5.4]{degiovanni1999cauchy},
which shows that $\langle\FF\cdot\nu, \cdot\rangle_{\partial U}$ is
represented by the classical dot product between $\FF$ and  $\nu_U$.
Hence, from Theorem {\rm \ref{Cauchy_flux_p}(i)},
we have
$$
\mathcal{F}(\partial U) =-\int_{\partial^*U}
\FF(x)\cdot\nu_{U}(x)\, \dr\Haus{n-1}(x).
$$

\smallskip
\item[\rm (ii)]
If the Borel function $h$ in Definition {\rm \ref{general}(iii)} is constant and $U$ is an open set of finite perimeter,
then $\FF \in \DM^{\infty}_{\rm loc}(\Omega)$, and
$$
\Fc (\partial U) =  - \int_{\partial ^* U} (\mathfrak{F}_{\ii} \cdot \nu_{U})  \, \dr\Haus{n-1},
$$
where $\mathfrak{F}_{\ii}\cdot\nu_U\in L^\infty(\partial^* U; \Haus{n-1})$ is the interior normal trace of $\FF$ on $\redb U$.

This corresponds to the case already treated in \cite[Theorem 9.4]{ctz},
since $U \Subset \Omega$ is a set of finite perimeter and $h$ is constant.
Hence, we obtain $\FF \in \DM^{\infty}_{\rm loc}(\Omega)$,
and the normal trace $\ban{\FF\cdot\nu, \cdot}_{\partial U}$ is represented by the measure
$-(\mathfrak{F}_{\ii} \cdot\nu_{U}) \, \Haus{n - 1} \res \partial^*U$, for some $(\mathfrak{F}_{\ii} \cdot\nu_{U}) \in L^\infty(\partial^*U; \Haus{n-1})$;
see also Proposition {\rm \ref{normal trace p infty}}.
\end{enumerate}
\end{remark}

\begin{remark}
\, The importance of Theorem {\rm \ref{Cauchy_flux_p}} and Remark {\rm \ref{Special-Case}} is that the flux can be recovered on every open set $U$.
\end{remark}

\begin{remark}
\, It has been discussed above that Schuricht \cite{Sch} considered an alternative formulation for the axioms,
representing the contact interactions as maps $f(B,A)$ on pairs of disjoint subbodies {\rm (}instead of surfaces{\rm )}.
In this formulation, our results on normal traces for $\DM^p$--fields improve those in
\cite[Theorem 5.20, equation (5.21)]{Sch},
since $f(B,A)$ can be written as the limit of the classical normal traces on the approximations of $B$,
versus an integral average.
\end{remark}

\begin{acknowledgements}
	\quad
	The research of
Gui-Qiang G. Chen was supported in part by
the UK
Engineering and Physical Sciences Research Council Award
EP/E035027/1 and
EP/L015811/1, and the Royal Society--Wolfson Research Merit Award (UK).
The research of Giovanni E. Comi was
supported in part by the PRIN2015 MIUR Project ``Calcolo delle Variazioni".
The research of Monica Torres was supported in part by
the Simons Foundation Award No. 524190 and by the National
Science Foundation Award No. 1813695.
\end{acknowledgements}

\bigskip

\noindent {\bf Compliance with ethical standards}


\vspace*{2mm}
\noindent {\bf Conflict of interest}\,  The authors declare that they have no conflict of interest.


\end{document}